\documentclass[11pt]{aart}

\usepackage[letterpaper, hmargin=1in, top=1in, bottom=1.2in, footskip=0.6in]{geometry}

\usepackage{titlesec}
\titleformat{\section}[block]{\filcenter\normalfont\bfseries\large}{\thesection.}{.5em}{}\titlespacing*{\section}{0pt}{2\baselineskip}{1\baselineskip}
\titleformat{\subsection}[runin]{\normalfont\bfseries}{\thesubsection.}{.4em}{}[.]\titlespacing{\subsection}{0pt}{2ex plus .1ex minus .2ex}{.8em}
\titleformat{\subsubsection}[runin]{\normalfont\itshape}{\thesubsubsection.}{.3em}{}[.]\titlespacing{\subsubsection}{0pt}{1ex plus .1ex minus .2ex}{.5em}
\titleformat{\paragraph}[runin]{\normalfont\itshape}{\theparagraph.}{.3em}{}[.]\titlespacing{\paragraph}{0pt}{1ex plus .1ex minus .2ex}{.5em}

\usepackage[labelfont=bf,font=small,labelsep=period]{caption}
\usepackage[T1]{fontenc}
\usepackage{lmodern}

\usepackage{amsmath}
\usepackage{amssymb}
\usepackage{amsfonts}
\usepackage{latexsym}
\usepackage{amsthm}
\usepackage{amsxtra}
\usepackage{amscd}
\usepackage{bbm}
\usepackage{mathrsfs}
\usepackage{bm}
\usepackage[utf8x]{inputenc}
\usepackage{mathrsfs}  
\usepackage{dsfont} 
\usepackage{enumitem}
\usepackage{subcaption}
\usepackage{wrapfig}
\usepackage{pdfpages}
\usepackage{pgfplots}

\usepackage{tikz}
\usetikzlibrary{arrows,decorations.pathmorphing,backgrounds,positioning,fit,petri}

\usepackage{graphicx, color}

\definecolor{darkred}{rgb}{0.9,0,0.3}
\definecolor{darkblue}{rgb}{0,0.3,0.9}

\definecolor{vdarkred}{rgb}{0.6,0,0.2}
\definecolor{vdarkblue}{rgb}{0,0.2,0.6}
\usepackage[pdftex, colorlinks, linkcolor=vdarkblue,citecolor=vdarkred]{hyperref}

\usepackage{booktabs}
\usepackage[nottoc,notlof,notlot]{tocbibind}
\usepackage{cite} %Enabling `[1-4]` - style citing
\let\originalleft\left
\let\originalright\right
\renewcommand{\left}{\mathopen{}\mathclose\bgroup\originalleft}
\renewcommand{\right}{\aftergroup\egroup\originalright}

\numberwithin{equation}{section}
\numberwithin{figure}{section}

\theoremstyle{plain} %plain, definition, remark
\newtheorem{theorem}{Theorem}[section]
\newtheorem*{theorem*}{Theorem}
\newtheorem{lemma}[theorem]{Lemma}
\newtheorem*{lemma*}{Lemma}
\newtheorem{corollary}[theorem]{Corollary}
\newtheorem*{corollary*}{Corollary}
\newtheorem{proposition}[theorem]{Proposition}
\newtheorem*{proposition*}{Proposition}

\newtheorem*{conjecture*}{Conjecture}

\theoremstyle{definition} %plain, definition, remark
\newtheorem{definition}[theorem]{Definition}
\newtheorem{convention}[theorem]{Convention}
\newtheorem*{definition*}{Definition}

\newtheorem*{example*}{Example}
\newtheorem{remark}[theorem]{Remark}
\newtheorem*{remark*}{Remark}

\newtheorem*{assumption*}{Assumption}

\newcommand{\f}[1]{\boldsymbol{\mathrm{#1}}} %bold
\renewcommand{\r}{\mathrm}  %upright
\newcommand{\bb}{\mathbb} %blackboard bold
\renewcommand{\cal}{\mathcal} 
\newcommand{\scr}{\mathscr} 
\newcommand{\fra}{\mathfrak} 
\newcommand{\ul}[1]{\underline{#1} \!\,} %underline
\newcommand{\ol}[1]{\overline{#1} \!\,} %overline
\newcommand{\wh}{\widehat}

\newcommand{\op}{\operatorname}

\renewcommand{\P}{\mathbb{P}}
\newcommand{\E}{\mathbb{E}}
\newcommand{\R}{\mathbb{R}}
\newcommand{\C}{\mathbb{C}}
\newcommand{\N}{\mathbb{N}}
\newcommand{\Z}{\mathbb{Z}}

\newcommand{\ee}{\mathrm{e}}
\newcommand{\ii}{\mathrm{i}}
\newcommand{\dd}{\mathrm{d}}
\newcommand{\col}{\mathrel{\vcenter{\baselineskip0.75ex \lineskiplimit0pt \hbox{.}\hbox{.}}}}
\newcommand*{\deq}{\mathrel{\vcenter{\baselineskip0.5ex \lineskiplimit0pt\hbox{\scriptsize.}\hbox{\scriptsize.}}}=}
\newcommand*{\eqd}{=\mathrel{\vcenter{\baselineskip0.5ex \lineskiplimit0pt\hbox{\scriptsize.}\hbox{\scriptsize.}}}}
\newcommand{\eqdist}{\overset{\r d}{=}}

\renewcommand{\leq}{\leqslant}
\renewcommand{\geq}{\geqslant}
\renewcommand{\epsilon}{\varepsilon}
\newcommand{\floorb}[1] {\big\lfloor #1 \big\rfloor}

\newcommand{\ceilb}[1]  {\big\lceil  #1 \big\rceil}

\newcommand{\ind}[1]{\mathbbm 1_{#1}}

\newcommand{\p}[1]{(#1)}
\newcommand{\pb}[1]{\bigl(#1\bigr)}
\newcommand{\pB}[1]{\Bigl(#1\Bigr)}
\newcommand{\pbb}[1]{\biggl(#1\biggr)}
\newcommand{\pBB}[1]{\Biggl(#1\Biggr)}

\newcommand{\q}[1]{[#1]}
\newcommand{\qb}[1]{\bigl[#1\bigr]}
\newcommand{\qB}[1]{\Bigl[#1\Bigr]}
\newcommand{\qbb}[1]{\biggl[#1\biggr]}
\newcommand{\qBB}[1]{\Biggl[#1\Biggr]}

\newcommand{\h}[1]{\{#1\}}
\newcommand{\hb}[1]{\bigl\{#1\bigr\}}

\newcommand{\hbb}[1]{\biggl\{#1\biggr\}}

\newcommand{\abs}[1]{\lvert #1 \rvert}
\newcommand{\absb}[1]{\bigl\lvert #1 \bigr\rvert}

\newcommand{\absbb}[1]{\biggl\lvert #1 \biggr\rvert}
\newcommand{\absBB}[1]{\Biggl\lvert #1 \Biggr\rvert}

\newcommand{\norm}[1]{\lVert #1 \rVert}
\newcommand{\normb}[1]{\bigl\lVert #1 \bigr\rVert}
\newcommand{\normB}[1]{\Bigl\lVert #1 \Bigr\rVert}
\newcommand{\normbb}[1]{\biggl\lVert #1 \biggr\rVert}
\newcommand{\normBB}[1]{\Biggl\lVert #1 \Biggr\rVert}

\newcommand{\scalar}[2]{\langle#1 \mspace{2mu}, #2\rangle}
\newcommand{\scalarb}[2]{\bigl\langle#1 \mspace{2mu}, #2\bigr\rangle}

\DeclareMathOperator{\diag}{diag}
\DeclareMathOperator{\tr}{Tr}

\DeclareMathOperator{\supp}{supp}
\DeclareMathOperator{\re}{Re}
\DeclareMathOperator{\im}{Im}

\DeclareMathOperator{\dist}{dist}

\DeclareMathOperator{\spec}{spec}
\DeclareMathOperator{\Span}{Span}

\newcommand{\eps}{\varepsilon}

\newcommand*{\defeq}{\mathrel{\vcenter{\baselineskip0.5ex \lineskiplimit0pt\hbox{\scriptsize.}\hbox{\scriptsize.}}}=}
\renewcommand{\Im}{\mathrm{Im}\,} 					%imaginary part of a complex number
\renewcommand{\Re}{\mathrm{Re}\,} 					%real part of a complex number

\newcommand{\oo}{\r o} 

\newcommand{\Cnu}{\mathcal{C}}

\newcommand{\fa}{\fra a}
\newcommand{\fc}{\fra c}

\usepackage{enumitem}

\title{Delocalization transition for critical Erd{\H o}s--R\'enyi graphs} 
\author{Johannes Alt \and Raphael Ducatez \and Antti Knowles}

\begin{document}
\maketitle

\begin{abstract}
We analyse the eigenvectors of the adjacency matrix of a critical Erd\H{o}s-R\'enyi graph $\bb G(N,d/N)$, where $d$ is of order $\log N$. We show that its spectrum splits into two phases: a delocalized phase in the middle of the spectrum, where the eigenvectors are completely delocalized, and a semilocalized phase near the edges of the spectrum, where the eigenvectors are essentially localized on a small number of vertices. In the semilocalized phase the mass of an eigenvector is concentrated in a small number of disjoint balls centred around resonant vertices, in each of which it is a radial exponentially decaying function. The transition between the phases is sharp and is manifested in a discontinuity in the localization exponent $\gamma(\f w)$ of an eigenvector $\f w$, defined through $\norm{\f w}_\infty / \norm{\f w}_2 = N^{-\gamma(\f w)}$. Our results remain valid throughout the optimal regime $\sqrt{\log N} \ll d \leq O(\log N)$.
\end{abstract}

\tableofcontents

\section{Introduction}

\subsection{Overview} \label{sec:overview}

Let $A$ be the adjacency matrix of a graph with vertex set $[N] = \{1, \dots, N\}$. We are interested in the geometric structure of the eigenvectors of $A$, in particular their \emph{spatial localization}. An $\ell^2$-normalized eigenvector $\f w = (w_x)_{x \in [N]}$ gives rise to a probability measure $\sum_{x \in [N]} w_x^2 \delta_x$ on the set of vertices. Informally, $\f w$ is \emph{delocalized} if its mass is approximately uniformly distributed throughout $[N]$, and \emph{localized} if its mass is essentially concentrated in a small number of vertices.

There are several ways of quantifying spatial localization. One is the notion of concentration of mass, sometimes referred to as scarring \cite{sarnak1995arithmetic}, stating that there is some set $\cal B \subset [N]$ of small cardinality and a small $\epsilon > 0$ such that $\sum_{x \in \cal B} w_x^2 = 1 - \epsilon$. In this case, it is also of interest to characterize the geometric structure of the vertex set $\cal B$ and of the eigenvector $\f w$ restricted to $\cal B$. Another convenient quantifier of spatial localization is the \emph{$\ell^p$-norm} $\norm{\f w}_p$ for $2 \leq p \leq \infty$. It has the following interpretation: if the mass of $\f w$ is uniformly distributed over some set $\cal B \subset [N]$ then $\norm{\f w}_p^2 = \abs{\cal B}^{-1 + 2/p}$. Focusing on the $\ell^\infty$-norm for definiteness, we define the \emph{localization exponent} $\gamma(\f w)$ through
\begin{equation} \label{def_gamma}
\norm{\f w}_\infty^2 \eqd N^{-\gamma(\f w)}\,.
\end{equation}
Thus, $0 \leq \gamma(\f w) \leq 1$, and $\gamma(\f w) = 0$ corresponds to localization at a single vertex while $\gamma(\f w) = 1$ to complete delocalization.

In this paper we address the question of spatial localization for the random Erd\H{o}s-R\'enyi graph $\bb G(N,d/N)$. We consider the limit $N \to \infty$ with $d \equiv d_N$. It is well known that $\bb G(N,d/N)$ undergoes a dramatic change in behaviour at the \emph{critical scale} $d \asymp \log N$, which is the scale at and below which the vertex degrees do not concentrate. Thus, for $d \gg \log N$, with high probability all degrees are approximately equal and the graph is \emph{homogeneous}. On the other hand, for $d \lesssim \log N$, the degrees do not concentrate and the graph becomes highly \emph{inhomogeneous}: it contains for instance hubs of exceptionally large degree, leaves, and isolated vertices. As long as $d > 1$, the graph has with high probability a unique giant component, and we shall always restrict our attention to it.

Here we propose the Erd\H{o}s-R\'enyi graph at criticality as a simple and natural model on which to address the question of spatial localization of eigenvectors. It has the following attributes.
\begin{enumerate}[label=(\roman*)]
\item
Its graph structure provides an intrinsic and nontrivial notion of distance.
\item
Its spectrum splits into a \emph{delocalized phase} and a \emph{semilocalized phase}. The transition between the phases is sharp, in the sense of a discontinuity in the localization exponent.
\item
Both phases are amenable to rigorous analysis.
\end{enumerate}

Our results are summarized in the phase diagram of Figure \ref{fig:phase_diagram}, which is expressed in terms of the parameter $b$ parametrizing $d = b \log N$ on the critical scale and the eigenvalue $\lambda$ of $A / \sqrt{d}$ associated with the eigenvector $\f w$. To the best of our knowledge, the phase coexistence for the critical Erd\H{o}s-R\'enyi graph established in this paper had previously not been analysed even in the physics literature.

\begin{figure}[!ht]
\begin{center}
{\small 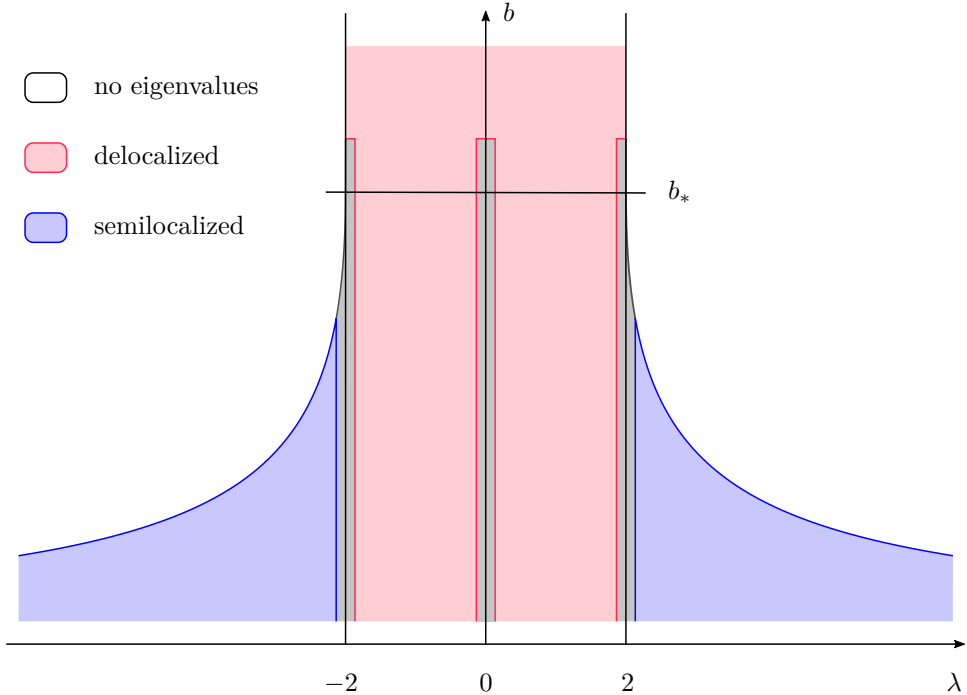}
\end{center}
\caption{The phase diagram of the adjacency matrix $A / \sqrt{d}$ of the Erd\H{o}s-R\'enyi graph $\bb G(N,d/N)$ at criticality, where $d = b \log N$ with $b$ fixed. The horizontal axis records the location in the spectrum and the vertical axis the sparseness parameter $b$. The spectrum is confined to the coloured region. In the red region the eigenvectors are delocalized while in the blue region they are semilocalized. The grey regions have width $o(1)$ and are not analysed in this paper. For $b > b_*$ the spectrum is asymptotically contained in $[-2,2]$ and the semilocalized phase does not exist. For $b < b_*$ a semilocalized phase emerges in the region $(-\lambda_{\max}(b), -2) \cup (2, \lambda_{\max}(b))$ for some explicit $\lambda_{\max}(b) > 2$.
\label{fig:phase_diagram}}
\end{figure}

Throughout the following, we always exclude the largest eigenvalue of $A$, its Perron-Frobenius eigenvalue, which is an outlier separated from the rest of the spectrum. The \emph{delocalized phase} is characterized by a localization exponent asymptotically equal to $1$. It exists for all fixed $b > 0$ and consists asymptotically of energies in $(-2,0) \cup (0,2)$. The \emph{semilocalized phase} is characterized by a localization exponent asymptotically less than $1$. It exists only when $b < b_*$, where
\begin{equation} \label{def_b_star}
b_* \deq \frac{1}{2\log 2 - 1} \approx 2.59\,.
\end{equation}
It consists asymptotically of energies in $(-\lambda_{\max}(b), -2) \cup (2, \lambda_{\max}(b))$, where $\lambda_{\max}(b) > 2$ is an explicit function of $b$ (see \eqref{lambda_max_rho} below). The density of states at energy $\lambda \in \R$ is equal to $N^{\rho_b(\lambda) + o(1)}$, where $\rho_b$ is an explicit exponent defined in \eqref{lambda_max_rho} below and illustrated in Figure \ref{fig:gamma}. It has a discontinuity at $2$ (and similarly at $-2$), jumping from $\rho_b(2^-) = 1$ to $\rho_b(2^+) = 1 - b / b^*$. The localization exponent $\gamma(\f w)$ from \eqref{def_gamma}  of an eigenvector $\f w$ with associated eigenvalue $\lambda$ satisfies with high probability
\begin{equation*}
\gamma(\f w) = 1 + o(1) \;\; \text{if} \;\; \abs{\lambda} < 2\,, \qquad \gamma(\f w) \leq \rho_b(\lambda) + o(1) \;\; \text{if} \;\; \abs{\lambda} > 2\,.
\end{equation*}
This establishes a discontinuity, in the limit $N \to \infty$, in the localization exponent $\gamma(\f w)$ as a function of $\lambda$ at the energies $\pm 2$. See Figure \ref{fig:gamma} for an illustration; we also refer to Appendix \ref{sec:simulation} for a simulation depicting the behaviour of $\norm{\f w}_\infty$ throughout the spectrum. Moreover, in the semilocalized phase scarring occurs in the sense that a fraction $1 - o(1)$ of the mass of the eigenvectors is supported in a set of at most $N^{\rho_b(\lambda) + o(1)}$ vertices. 

\begin{figure}[!ht]
\begin{center}
{\small %% Creator: Inkscape 1.0 (4035a4f, 2020-05-01), www.inkscape.org
%% Accompanies image file 'fig6.pdf' (pdf, eps, ps)
\begingroup%
  \makeatletter%
  \providecommand\color[2][]{%
    \errmessage{(Inkscape) Color is used for the text in Inkscape, but the package 'color.sty' is not loaded}%
    \renewcommand\color[2][]{}%
  }%
  \providecommand\transparent[1]{%
    \errmessage{(Inkscape) Transparency is used (non-zero) for the text in Inkscape, but the package 'transparent.sty' is not loaded}%
    \renewcommand\transparent[1]{}%
  }%
  \providecommand\rotatebox[2]{#2}%
  \newcommand*\fsize{\dimexpr\f@size pt\relax}%
  \newcommand*\lineheight[1]{\fontsize{\fsize}{#1\fsize}\selectfont}%
  \ifx\svgwidth\undefined%
    \setlength{\unitlength}{324.63314145bp}%
    \ifx\svgscale\undefined%
      \relax%
    \else%
      \setlength{\unitlength}{\unitlength * \real{\svgscale}}%
    \fi%
  \else%
    \setlength{\unitlength}{\svgwidth}%
  \fi%
  \global\let\svgwidth\undefined%
  \global\let\svgscale\undefined%
  \makeatother%
  \begin{picture}(1,0.35826505)%
    \lineheight{1}%
    \setlength\tabcolsep{0pt}%
    \put(0,0){\includegraphics[width=\unitlength,page=1]{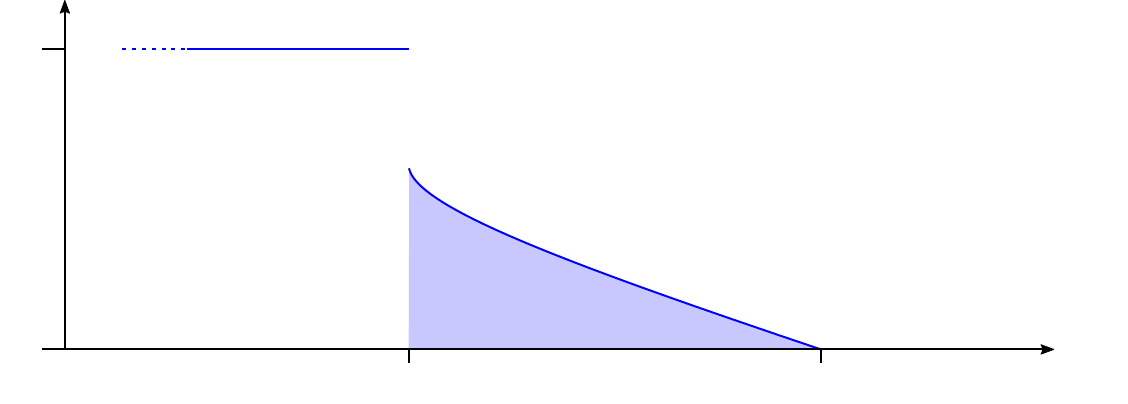}}%
    \put(0.92163268,0.00226546){\color[rgb]{0,0,0}\makebox(0,0)[lt]{\lineheight{1.25}\smash{\begin{tabular}[t]{l}$\lambda$\end{tabular}}}}%
    \put(-0.00118603,0.30727566){\color[rgb]{0,0,0}\makebox(0,0)[lt]{\lineheight{1.25}\smash{\begin{tabular}[t]{l}1\end{tabular}}}}%
    \put(-0.00097888,0.03906214){\color[rgb]{0,0,0}\makebox(0,0)[lt]{\lineheight{1.25}\smash{\begin{tabular}[t]{l}0\end{tabular}}}}%
    \put(0.35398415,0.0016082){\color[rgb]{0,0,0}\makebox(0,0)[lt]{\lineheight{1.25}\smash{\begin{tabular}[t]{l}2\end{tabular}}}}%
    \put(0.67956367,0.00226546){\color[rgb]{0,0,0}\makebox(0,0)[lt]{\lineheight{1.25}\smash{\begin{tabular}[t]{l}$\lambda_{\max}(b)$\end{tabular}}}}%
    \put(0.13129242,0.00018651){\makebox(0,0)[lt]{\lineheight{1.25}\smash{\begin{tabular}[t]{l}delocalized\end{tabular}}}}%
    \put(0.45615456,0.00018651){\makebox(0,0)[lt]{\lineheight{1.25}\smash{\begin{tabular}[t]{l}semilocalized\end{tabular}}}}%
  \end{picture}%
\endgroup%
}
\end{center}
\caption{The behaviour of the exponents $\rho_b$ and $\gamma$ as a function of the energy $\lambda$. The dark blue curve is the exponent $\rho_b(\lambda)$ characterizing the density of states $N^{\rho_b(\lambda) + o(1)}$ of the matrix $A / \sqrt{d}$ at energy $\lambda$. The entire blue region (light and dark blue) is the asymptotically allowed region of the localization exponent $\gamma(\f w)$ of an eigenvector of $A / \sqrt{d}$ as a function of the associated eigenvalue $\lambda$. Here $d = b \log N$ with $b = 1$ and $\lambda_{\max}(b) \approx 2.0737$. We only plot a neighbourhood of the threshold energy $2$. The discontinuity at $2$ of $\rho_b$ is from $\rho_b(2^-) = 1$ to $\rho_b(2^+) = 1 - b / b^* = 2 - 2 \log 2$.
\label{fig:gamma}}
\end{figure}
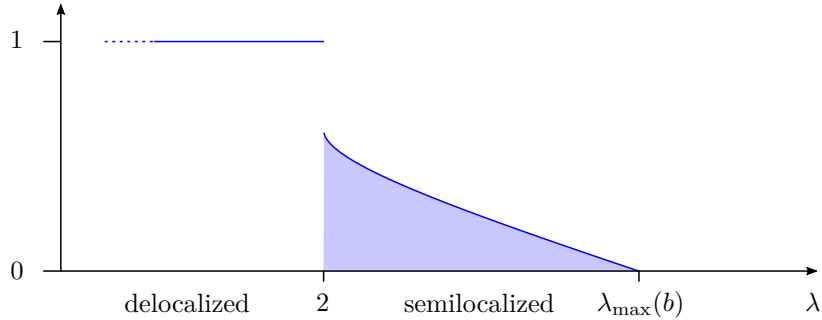

The eigenvalues in the semilocalized phase were analysed in \cite{ADK19}, where it was proved that they arise precisely from vertices $x$ of abnormally large degree, $D_x \geq 2 d$. More precisely, it was proved in \cite{ADK19} that each vertex $x$ with $D_x \geq 2 d$ gives rise to two eigenvalues of $A / \sqrt{d}$ near $\pm \Lambda(D_x / d)$, where $\Lambda(\alpha) \deq \frac{\alpha}{\sqrt{\alpha-1}}$. The same result for the $O(1)$ largest degree vertices was independently proved in \cite{tikhomirov2019outliers} by a different method. We refer also to \cite{BBK1, BBK2} for an analysis in the supercritical and subcritical phases.

In the current paper, we prove that the eigenvector $\f w$  associated with an eigenvalue $\lambda$ in the semilocalized phase is highly concentrated around \emph{resonant vertices} at energy $\lambda$, which are defined as the vertices $x$ such that $\Lambda(D_x/d)$ is close to $\lambda$. For this reason, we also call the resonant vertices \emph{localization centres}. With high probability, and after a small pruning of the graph, all balls $B_r(x)$ of a certain radius $r \gg 1$ around the resonant vertices are disjoint, and within any such ball $B_r(x)$ the eigenvector $\f w$ is an approximately radial exponentially decaying function. The number of resonant vertices at energy $\lambda$ is comparable to the density of states, $N^{\rho_b(\lambda) + o(1)}$, which is much less than $N$. See Figure \ref{fig:balls} for a schematic illustration of the mass distribution of $\f w$.

\begin{figure}[!ht]
\begin{center}
\includegraphics{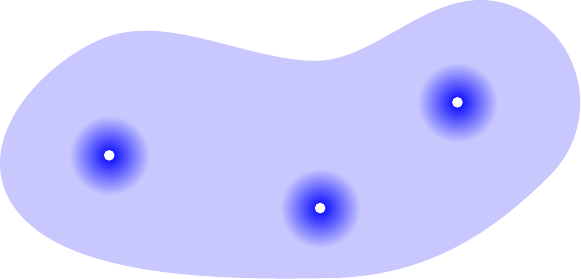}
\end{center}
\caption{A schematic representation of the geometric structure of a typical eigenvector in the semilocalized phase. The giant component of the graph is depicted in pale blue. The eigenvector's mass (depicted in dark blue) is concentrated in a small number of disjoint balls centred around resonant vertices (drawn in white), and within each ball the mass decays exponentially in the radius. The mass outside the balls is an asymptotically vanishing proportion of the total mass.
\label{fig:balls}}
\end{figure}

The behaviour of the critical Erd\H{o}s-R\'enyi graph described above has some similarities but also differences to that of the Anderson model \cite{anderson1958absence}. The Anderson model on $\Z^n$ with $n \geq 3$ is conjectured to exhibit a metal-insulator, or delocalization-localization, transition: for weak enough disorder, the spectrum splits into a delocalized phase in the middle of the spectrum and a localized phase near the spectral edges. See e.g.\ \cite[Figure 1.2]{aizenman2015random} for a phase diagram of its conjectured behaviour. So far, only the localized phase of the Anderson model has been understood rigorously, in the landmark works \cite{FroSpe, AizMol}, as well as contributions of many subsequent developments. The phase diagram for the Anderson model bears some similarity to that of Figure \ref{fig:phase_diagram}, in which one can interpret $1/b$ as the disorder strength, since smaller values of $b$ lead to stronger inhomogeneities in the graph.

As is apparent from the proofs in \cite{FroSpe, AizMol}, in the localized phase the local structure of an eigenvector of the Anderson model is similar to that of the critical Erd\H{o}s-R\'enyi graph described above: exponentially decaying around well-separated localization centres associated with resonances near the energy $\lambda$ of the eigenvector. The localization centres arise from exceptionally large local averages of the potential.  
The phenomenon of localization can be heuristically understood using the following well-known \emph{rule of thumb}: one expects localization around a single localization centre if \emph{the level spacing is much larger than the tunnelling amplitude between localization centres}. It arises from perturbation theory around the block diagonal model where the complement of balls $B_r(x)$ around localization centres is set to zero. On a very elementary level, this rule is illustrated by the matrix $H(t) = \bigl( \begin{smallmatrix}0 & t\\ t & 1\end{smallmatrix}\bigr)$, whose eigenvectors are localized for $t = 0$, remain essentially localized for $t \ll 1$, where perturbation theory around $H(0)$ is valid, and become delocalized for $t \gtrsim 1$, where perturbation theory around $H(0)$ fails.

More precisely, it is a general heuristic that the tunnelling amplitude decays exponentially in the distance between the localization centres \cite{combes1973asymptotic}. Denoting by $\beta(\lambda) > 1$ the rate of exponential decay at energy $\lambda$, the rule of thumb hence reads 
\begin{equation} \label{rule_thumb}
\beta(\lambda)^{-L}\ll\epsilon(\lambda)\,,
\end{equation}
where $L$ is the distance between the localization centres and $\epsilon(\lambda)$ the level spacing at energy $\lambda$. For the Anderson model restricted to a finite cube of $\mathbb{Z}^n$ with side
length $N^{1/n}$, the level spacing $\epsilon(\lambda)$ is of order
$N^{-1}$ (see \cite{wegner1981bounds} and \cite[Chapter 4]{aizenman2015random}) whereas the diameter of the graph
is of order $N^{1/n}$. Hence, the rule of thumb \eqref{rule_thumb} becomes
\[
\beta(\lambda)^{-N^{1/n}}\ll N^{-1}\,,
\]
which is satisfied and one therefore expects localization. For the critical Erd\H{o}s-R\'enyi
graph, the level spacing $\epsilon(\lambda)$ is $N^{-\rho(\lambda)+o(1)}$
but the diameter of the giant component is only $\frac{\log N}{\log d}$.
Hence, the rule of thumb \eqref{rule_thumb} becomes
\[
N^{-\frac{\log\beta(\lambda)}{\log d}}\ll N^{-\rho(\lambda)+o(1)}\,,
\]
which is never satisfied because $\frac{\log \beta(\lambda)}{\log d}\rightarrow 0$ as $N \to \infty$. Thus, the rule of
thumb \eqref{rule_thumb} is satisfied in the localized phase of the Anderson model
but not in the semilocalized phase of the critical Erd\H{o}s-Rényi
graph.
The underlying reason behind this difference is that the diameter of the Anderson model is polynomial in $N$, while the diameter of the critical Erd\H{o}s-R\'enyi graph is logarithmic in $N$. 
Thus, the critical Erd\H{o}s-R\'enyi graph is far more connected than the Anderson model; this property tends to push it more towards the delocalized behaviour of mean-field systems. As noted above, another important difference between the localized phase of the Anderson model and the semilocalized phase of the critical Erd\H{o}s-R\'enyi graph is that the density of states is of order $N$ in the former and a fractional power of $N$ in the latter.

Up to now we have focused on the Erd\H{o}s-R\'enyi graph on the critical scale $d \asymp \log N$. It is natural to ask whether this assumption can be relaxed without changing its behaviour. The question of the upper bound on $d$ is simple: as explained above, there is no semilocalized phase for $d > b_* \log N$, and the delocalized phase is completely understood up to $d \leq N/2$, thanks to Theorem \ref{thm:delocalization} below and \cite{HeKnowlesMarcozzi2018, EKYY1}. The lower bound is more subtle. In fact, it turns out that all of our results remain valid throughout the regime
\begin{equation} \label{optimal_regime}
\sqrt{\log N} \ll d \leq O(\log N)\,.
\end{equation}
The lower bound $\sqrt{\log N}$ is optimal in the sense that below it both phases are disrupted and the phase diagram from Figure \ref{fig:phase_diagram} no longer holds. Indeed, for $d \lesssim \sqrt{\log N}$ a new family of localized states, associated with so-called \emph{tuning forks} at the periphery of the graph, appear throughout the delocalized and semilocalized phases. We refer to Section \ref{sec:forks_intro} below for more details.

Previously, strong delocalization with localization exponent $\gamma(\f w) = 1 + o(1)$ has been established for many mean-field models, such as Wigner matrices \cite{ESY2, EYY3, EKYY4,EKYY1,Agg16}, supercritical Erd\H{o}s-R\'enyi graphs \cite{EKYY1, HeKnowlesMarcozzi2018}, and random regular graphs \cite{BKY15,BHY1}. All of these models are homogeneous and only have a delocalized phase.

Although a rigorous understanding of the metal-insulator transition for the Anderson model is still elusive, some progress has been made for random band matrices. Random band matrices \cite{MFDQS,Wig,CMI1,FM1} constitute an attractive model interpolating between the Anderson model and mean-field Wigner matrices. They retain the $n$-dimensional structure of the Anderson model but have proved somewhat more amenable to rigorous analysis. They are conjectured \cite{FM1} to have a similar phase diagram as the Anderson model in dimensions $n \geq 3$. As for the Anderson model, dimensions $n > 1$ have so far seen little progress, but for $n = 1$ much has been understood both in the localized \cite{Sch,PSSS1} and the delocalized \cite{So1,EKY2,EKYY3, EK1, EK2, EK3, EK4,shcherbina2019universality,BYY1,BYY2,BEYY1,YY1,HM1} phases. A simplification of band matrices is the ultrametric ensemble \cite{FOR1}, where the Euclidean metric of $\Z^n$ is replaced with an ultrametric arising from a tree structure. For this model, a phase transition was rigorously established in \cite{VW1}.

Another modification of the $n$-dimensional Anderson model is the Anderson model on the Bethe lattice, an infinite regular tree corresponding to the case $n = \infty$. For it, the existence of a delocalized phase was shown in \cite{ASW1,FHS1,K94}. In \cite{AW1,AW2} it was shown that for unbounded random potentials the delocalized phase exists for arbitrarily weak disorder. It extends beyond the spectrum of the unperturbed adjacency matrix into the so-called Lifschitz tails, where the density of states is very small. The authors showed that, through the mechanism of resonant delocalization, the exponentially decaying tunnelling amplitudes between localization centres are counterbalanced by an exponentially large number of possible channels through which tunnelling can occur, so that the rule of thumb \eqref{rule_thumb} for localization is violated. As a consequence, the eigenvectors are delocalized across many resonant localization centres. We remark that this analysis was made possible by the absence of cycles on the Bethe lattice. In contrast, the global geometry of the critical Erd\H{o}s-R\'enyi graph is fundamentally different from that of the Bethe lattice (through the existence of a very large number of long cycles), which has a defining impact on the nature of the delocalization-semilocalization transition summarized in Figure \ref{fig:phase_diagram}.

Transitions in the localization behaviour of eigenvectors have also been analysed in several mean-field type models. In \cite{LS1, LS2} the authors considered the sum of a Wigner matrix and a diagonal matrix with independent random entries with a large enough variance. They showed that the eigenvectors in the bulk are delocalized while near the edge they are partially localized at a single site. Their partially localized phase can be understood heuristically as a rigorous (and highly nontrivial) verification of the rule of thumb for localization, where the perturbation takes place around the diagonal matrix. Heavy-tailed Wigner matrices, or Lévy matrices, whose entries have $\alpha$-stable laws for $0 < \alpha < 2$, were proposed in \cite{CB1} as a simple model that exhibits a transition in the localization of its eigenvectors; we refer to \cite{ALY1} for a summary of the predictions from \cite{CB1, TBT1}. In \cite{BG1, BG2} it was proved that for energies in a compact interval around the origin, eigenvectors are weakly delocalized, and for $0 < \alpha < 2/3$ for energies far enough from the origin, eigenvectors are weakly localized. In \cite{ALY1}, full delocalization was proved in a compact interval around the origin, and the authors even established GOE local eigenvalue statistics in the same spectral region. In \cite{ALM1}, the law of the eigenvector components of Lévy matrices was computed.

\paragraph{Conventions}
Throughout the following, every quantity that is not explicitly \emph{constant} depends on the fundamental parameter $N$. We almost always omit this dependence from our notation. We use $C$ to denote a generic positive universal constant, and write $X = O(Y)$ to mean $\abs{X} \leq C Y$. For $X,Y > 0$ we write $X \asymp Y$ if $X = O(Y)$ and $Y = O(X)$. We write $X \ll Y$ or $X = o(Y)$ to mean $\lim_{N \to \infty} X/Y = 0$. A vector is \emph{normalized} if its $\ell^2$-norm is one.

\subsection{Results -- the semilocalized phase} \label{sec:localized_results}
Let $\bb G = \bb G(N,d/N)$ be the Erd{\H o}s--R\'enyi graph with vertex set $[N] \defeq \{1, \ldots, N\}$ and edge probability $d/N$ for $0 \leq d \leq N$. Let $A = (A_{xy})_{x,y \in [N]} \in \{0,1\}^{N\times N}$ be the adjacency matrix of $\bb G$. Thus, $A =A^*$, $A_{xx}=0$ for all $x \in [N]$, and  $( A_{xy} \col x < y)$ are independent $\op{Bernoulli}(d/N)$ random variables.

The entrywise nonnegative matrix $A/\sqrt{d}$ has a \emph{trivial} Perron-Frobenius eigenvalue, which is its largest eigenvalue. In the following we only consider the other eigenvalues, which we call \emph{nontrivial}. In the regime $d \gg \sqrt{\log N / \log \log N}$, which we always assume in this paper, the trivial eigenvalue is located at $\sqrt{d} (1 + o(1))$, and it is separated from the nontrivial ones with high probability; see \cite{BBK2}. Moreover, without loss of generality in this subsection we always assume that $d \leq 3 \log N$, for otherwise the semilocalized phase does not exist (see Section \ref{sec:overview}).

For $x \in [N]$ we define the \emph{normalized degree} of $x$ as
\begin{equation} \label{eq:def_alpha_x} 
\alpha_x \deq \frac{1}{d} \sum_{y \in [N]} A_{xy}\,.
\end{equation}
In Theorem \ref{thm:one_to_one} below we show that the nontrivial eigenvalues of $A / \sqrt{d}$ outside the interval $[-2,2]$ are in two-to-one correspondence with vertices with normalized degree greater than $2$: each vertex $x$ with $\alpha_x > 2$ gives rise to two eigenvalues of $A / \sqrt{d}$  located with high probability near $\pm \Lambda(\alpha_x)$, where we defined the bijective function $\Lambda \col [2,\infty) \to [2,\infty)$ through
\begin{equation} \label{eq:def_Lambda}
\Lambda(\alpha) \deq \frac{\alpha}{\sqrt{\alpha-1}}. 
\end{equation}

Our main result in the semilocalized phase is about the eigenvectors associated with these eigenvalues. To state it, we need the following notions.

\begin{definition}
Let $\lambda>2$ and $0 < \delta \leq \lambda - 2$. We define the set of \emph{resonant vertices at energy $\lambda$} through
\begin{equation} \label{def_calW}
\cal W_{\lambda,\delta} \deq \hb{x \col \alpha_x \geq 2, \abs{\Lambda(\alpha_x) - \lambda} \leq \delta}\,.
\end{equation}
\end{definition}

We denote by $B_r(x)$ the ball around the vertex $x$ of radius $r$ for the graph distance in $\bb G$. Define
\begin{equation} \label{def_r_star}
r_\star = \big \lfloor c \sqrt{\log N} \big \rfloor \,;
\end{equation}
all of our results will hold provided $c > 0$ is chosen to be a small enough universal constant. The quantity $r_\star$ will play the role of a maximal radius for balls around localization centres.

We introduce the basic control parameters
\begin{equation} \label{eq:main_error_term}
\xi \deq \frac{\sqrt{\log N}}{d} \log d\,, \qquad  \xi_u \deq \frac{\sqrt{\log N}}{d} \frac{1}{u}\,,
\end{equation}
which under our assumptions will always be small (see Remark \ref{rem:xi} below). We now state our main result in the semilocalized phase.

\begin{theorem}[Semilocalized phase] \label{thm:localisation}
For any $\nu > 0$ there exists a constant $\cal C$ such that the following holds.
Suppose that
\begin{equation} \label{d_assumption_localization}
\cal C \sqrt{\log N} \log \log N \leq d \leq 3\log N\,.
\end{equation}
Let $\f w$ be a normalized eigenvector of $A/\sqrt{d}$ with nontrivial eigenvalue $\lambda \geq 2+\cal C \xi^{1/2}$. Let $0<\delta\leq (\lambda-2)/2$.
Then for each $x \in \cal W_{\lambda, \delta}$ there exists a normalized vector $\f v(x)$, supported in $B_{r_\star}(x)$, such that the supports of $\f v(x)$ and $\f v(y)$ are disjoint for $x \neq y$, and
$$ \sum_{x \in \cal W_{\lambda,\delta}} \scalar{\f v(x)}{\f w}^2 \geq  1 - \cal C \left(\frac{\xi+\xi_{\lambda-2}}{\delta}\right)^2$$  with probability at least $1 - \cal C N^{-\nu}$.
Moreover, $\f v(x)$ decays exponentially around $x$ in the sense that for any $r \geq 0$ we have
\begin{equation*}
\sum_{y \notin B_r(x)} (\f v(x))_y^2 \leq \frac{1}{(\alpha_x - 1)^{r+1}}\,.
\end{equation*}
\end{theorem}

\begin{remark} 
An analogous result holds for negative eigenvalues $-\lambda \leq -2 - \cal C \xi^{1/2}$, with a different vector $\f v(x)$.
See Theorem \ref{thm:general} and Remark \ref{rem:general_negative} below for a precise statement.
\end{remark}

\begin{remark}
The upper bound $d \leq 3 \log N$ in \eqref{d_assumption_localization} is made for convenience and without loss of generality, because if $d > 3 \log N$ then, as explained in Section \ref{sec:overview}, with high probability the semilocalized phase does not exist, i.e.\ eigenvalues satisfying the conditions of Theorem \ref{thm:localisation} do not exist.
\end{remark}

Theorem \ref{thm:localisation} implies that $\f w$ is almost entirely concentrated in the balls around the resonant vertices, and in each such ball $B_{r_\star}(x)$, $x \in \cal W_{\lambda,\delta}$, the vector  $\f w$ is almost collinear to the vector $\f v(x)$. Thus, $\f v(x)$ has the interpretation of the \emph{localization profile around the localization centre $x$}.  Since it has exponential decay, we deduce immediately from Theorem \ref{thm:localisation} that the radius $r_\star$ can be made smaller at the expense of worse error terms. In fact, in Definition \ref{def:bv} and Theorem \ref{thm:general} below, we give an explicit definition of $\f v(x)$, which shows that it is \emph{radial} in the sense that its value at a vertex $y$ depends only on the distance between $x$ and $y$, in which it is an exponentially decaying function. To ensure that the supports of the vectors $\f v(x)$ for different $x$ do not overlap, $\f v(x)$ is in fact defined as the restriction of a radial function around $x$ to a subgraph of $\bb G$, the \emph{pruned graph}, which differs from $\bb G$ by only a small number of edges and whose balls of radius $r_\star$ around the vertices of $\cal W_{\lambda,\delta}$ are disjoint (see Proposition \ref{prop:subgraph_separating_large_degrees} below).
For positive eigenvalues, the entries of $\f v(x)$ are nonnegative, while for negative eigenvalues its entries carry a sign that alternates in the distance to $x$. The set of resonant vertices $\cal W_{\lambda,\delta}$ is a small fraction of the whole vertex set $[N]$; its size is analysed in Lemma \ref{lem:alpha_distr} below.

\begin{remark} \label{rem:xi}
Note that, by the lower bounds imposed on $d$ and $\lambda$ in Theorem \ref{thm:localisation}, we always have $\xi, \xi_{\lambda - 2} \leq 1/ \cal C$.
\end{remark}

Using the exponential decay of the localization profiles, it is easy to deduce from Theorem \ref{thm:localisation} that a positive proportion of the eigenvector mass concentrates at the resonant vertices.

\begin{corollary} \label{cor:loc_centres}
Under the assumptions of Theorem \ref{thm:localisation} we have $$\sum_{y \in \cal W_{\lambda,\delta}} w_y^2 =  \frac{\sqrt{\lambda^2-4}}{ \lambda + \sqrt{\lambda^2-4}} +O \bigg(\frac{\cal C(\xi+\xi_{\lambda-2})}{\delta}+\frac{\cal C \delta}{\lambda^{5/2} \sqrt{\lambda-2}}\bigg) $$
with probability at least $1 - \cal C N^{-\nu}$.
\end{corollary}

Next, we state a rigidity result on the eigenvalue locations in the semilocalized phase. It generalizes \cite[Corollary~2.3]{ADK19} by improving the error bound and extending it to the full regime \eqref{optimal_regime} of $d$, below which it must fail (see Section \ref{sec:forks_intro} below). Its proof is a byproduct of the proof of our main result in the semilocalized phase, Theorem \ref{thm:localisation}. We denote the ordered eigenvalues of a Hermitian matrix $M\in \C^{N\times N}$ by $\lambda_1(M) \geq \lambda_2(M) \geq \cdots \geq \lambda_N(M)$. We only consider the nontrivial eigenvalues of $A / \sqrt{d}$, i.e.\ $\lambda_i(A / \sqrt{d})$ with $2 \leq i \leq N$.
For the following statements we order the normalized degrees by choosing a (random) permutation $\sigma \in S_N$ such that $i \mapsto \alpha_{\sigma(i)}$ is nonincreasing.

\begin{theorem}[Eigenvalue locations in semilocalized phase] \label{thm:one_to_one}
For any $\nu > 0$ there exists a constant $\cal C$ such that the following holds.
Suppose that \eqref{d_assumption_localization} holds. Let
\begin{equation*}
\cal U \deq \h{x \in [N] \col \Lambda(\alpha_x) \geq 2 + \xi^{1/2}}\,.
\end{equation*}
Then with probability at least $1 - \cal C N^{-\nu}$,
for all $1\leq i\leq |\cal U|$ we have
\begin{equation} \label{lambda_estimate}
|\lambda_{i + 1}(A/\sqrt{d})-\Lambda(\alpha_{\sigma(i)})| + |\lambda_{N-i+1}(A/\sqrt{d})+\Lambda(\alpha_{\sigma(i)})| \leq \cal C (\xi+\xi_{\Lambda(\alpha_{\sigma(i)})-2})
\end{equation}
and for all $\abs{\cal U} + 2 \leq i \leq N - \abs{\cal U}$ we have
\begin{equation} \label{lambda_bulk_estimate}
\abs{\lambda_i(A/\sqrt{d})} \leq 2 + \xi^{1/2}\,.
\end{equation}
\end{theorem}

We remark that the upper bound on $d$ from \eqref{d_assumption_localization}, which is necessary for the existence of a semilocalized phase, can be relaxed in Theorem \ref{thm:one_to_one} to obtain an estimate on $\max_{2 \leq i \leq N} \abs{\lambda_i(A / \sqrt{d})}$ in the supercritical regime $d \geq 3 \log N$, which is sharper than the one in \cite{ADK19}. The proof is the same and we do not pursue this direction here.

We conclude this subsection with a discussion on the counting function of the normalized degrees, which we use to give estimates on the number of resonant vertices \eqref{def_calW}.
For $b \geq 0$ and $\alpha \geq 2$ define the exponent
\begin{equation} \label{def_theta}
\theta_b(\alpha) \deq [1 - b (\alpha \log \alpha - \alpha + 1)]_+\,.
\end{equation}
Define $\alpha_{\max}(b) \deq \inf \{\alpha \geq 2 \col \theta_b(\alpha) = 0\}$. Thus, $\theta_b$ is a nonincreasing function that is nonzero on $[0, \alpha_{\max}(b))$. Moreover, $\theta_b(2) = [1 - b/b_*]_+$, so that $\alpha_{\max}(b) > 2$ if and only if $b < b_*$.
From Lemma \ref{lem:degree_distr} below it is easy to deduce that if $d \gg 1$ then $\alpha_{\sigma(1)} = \alpha_{\max}(d/\log N) + O(\zeta / d)$  with probability at least $1 - o(1)$ for any $\zeta \gg 1$. Thus, $\alpha_{\max}(d/\log N)$ has the interpretation of the deterministic location of the largest normalized degree. See Figure \ref{fig:theta} for a plot of $\theta_b$.

\begin{figure}[!ht]
\begin{center}
{\small %% Creator: Inkscape 1.0beta2 (2b71d25, 2019-12-03), www.inkscape.org
%% Accompanies image file 'fig5.pdf' (pdf, eps, ps)
\begingroup%
  \makeatletter%
  \providecommand\color[2][]{%
    \errmessage{(Inkscape) Color is used for the text in Inkscape, but the package 'color.sty' is not loaded}%
    \renewcommand\color[2][]{}%
  }%
  \providecommand\transparent[1]{%
    \errmessage{(Inkscape) Transparency is used (non-zero) for the text in Inkscape, but the package 'transparent.sty' is not loaded}%
    \renewcommand\transparent[1]{}%
  }%
  \providecommand\rotatebox[2]{#2}%
  \newcommand*\fsize{\dimexpr\f@size pt\relax}%
  \newcommand*\lineheight[1]{\fontsize{\fsize}{#1\fsize}\selectfont}%
  \ifx\svgwidth\undefined%
    \setlength{\unitlength}{242.62827218bp}%
    \ifx\svgscale\undefined%
      \relax%
    \else%
      \setlength{\unitlength}{\unitlength * \real{\svgscale}}%
    \fi%
  \else%
    \setlength{\unitlength}{\svgwidth}%
  \fi%
  \global\let\svgwidth\undefined%
  \global\let\svgscale\undefined%
  \makeatother%
  \begin{picture}(1,0.55853628)%
    \lineheight{1}%
    \setlength\tabcolsep{0pt}%
    \put(0,0){\includegraphics[width=\unitlength,page=1]{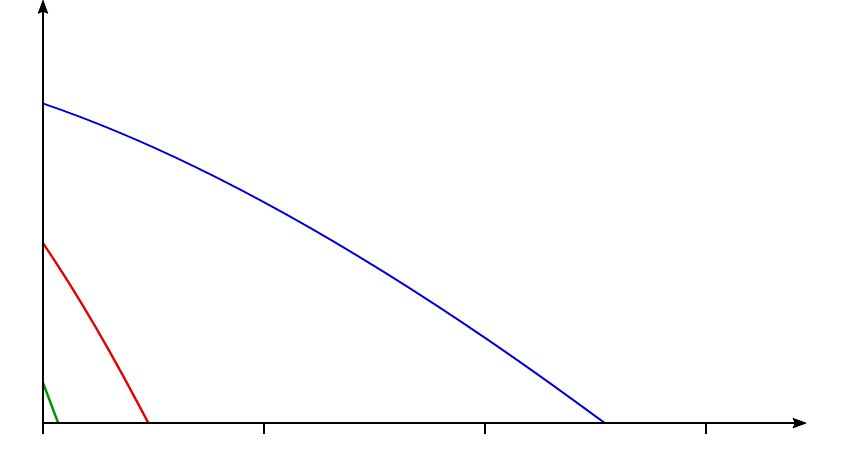}}%
    \put(0.04143192,0.00268619){\makebox(0,0)[lt]{\lineheight{1.25}\smash{\begin{tabular}[t]{l}$2$\end{tabular}}}}%
    \put(0.30363682,0.00268619){\makebox(0,0)[lt]{\lineheight{1.25}\smash{\begin{tabular}[t]{l}$3$\end{tabular}}}}%
    \put(0.56578244,0.00268619){\makebox(0,0)[lt]{\lineheight{1.25}\smash{\begin{tabular}[t]{l}$4$\end{tabular}}}}%
    \put(0.82792799,0.00268619){\makebox(0,0)[lt]{\lineheight{1.25}\smash{\begin{tabular}[t]{l}$5$\end{tabular}}}}%
    \put(0.00513488,0.0467659){\makebox(0,0)[lt]{\lineheight{1.25}\smash{\begin{tabular}[t]{l}$0$\end{tabular}}}}%
    \put(0.00513488,0.47949754){\makebox(0,0)[lt]{\lineheight{1.25}\smash{\begin{tabular}[t]{l}$1$\end{tabular}}}}%
    \put(0,0){\includegraphics[width=\unitlength,page=2]{fig5.pdf}}%
    \put(0.9392094,0.00268619){\makebox(0,0)[lt]{\lineheight{1.25}\smash{\begin{tabular}[t]{l}$\alpha$\end{tabular}}}}%
    \put(-0.00104741,0.53513821){\makebox(0,0)[lt]{\lineheight{1.25}\smash{\begin{tabular}[t]{l}$\theta_b$\end{tabular}}}}%
  \end{picture}%
\endgroup%
}
\end{center}
\caption{
A plot of the exponent $\theta_b(\alpha)$ as a function of $\alpha \geq 2$ for the values $b = 0.3$ (blue), $b = 1.3$ (red), and $b = 2.3$ (green). The graph hits the value $0$ at $\alpha_{\max}(b)$.
\label{fig:theta}
}
\end{figure}

In Appendix \ref{sec:degrees} below, we obtain estimates on the density of the normalized degrees $(\alpha_x)_{x \in [N]}$ and combine it with Theorem \ref{thm:localisation} to deduce a lower bound on the $\ell^p$-norm of eigenvectors in the semilocalized phase. The precise statements are given in Lemma \ref{lem:alpha_distr} and Corollary \ref{cor:lower_bound_w}, which provide quantitative error bounds throughout the regime \eqref{d_assumption_localization}. Here, we summarize them, for simplicity, in simple qualitative versions in the critical regime $d \asymp \log N$. For $b < b_*$ we abbreviate
\begin{equation} \label{lambda_max_rho}
\lambda_{\max}(b) \deq \Lambda(\alpha_{\max}(b))\,, \qquad \rho_b(\lambda) \deq
\begin{cases}
\theta_b(\Lambda^{-1}(\lambda)) & \text{if } \abs{\lambda} \geq 2
\\
1 & \text{if } \abs{\lambda} < 2\,,
\end{cases}
\end{equation}
where $\Lambda^{-1}(\lambda) = \frac{\lambda^2}{2}(1 + \sqrt{1 - 4/\lambda^2})$ for $\abs{\lambda} \geq 2$.
Let $d = b \log N$ with some constant $b < b_*$, and suppose that $2 + \kappa \leq \lambda \leq \lambda_{\max}(b) - \kappa$ for some constant $\kappa > 0$. Then Lemma \ref{lem:alpha_distr} \ref{itm:W_size} implies (choosing $1/d \ll \delta \ll 1$)
\begin{equation} \label{eq:expansion_size_W_lambda_delta} 
\abs{\cal W_{\lambda, \delta}} = N^{\rho_b(\lambda) + o(1)}
\end{equation}
with probability $1 - o(1)$. From \eqref{eq:expansion_size_W_lambda_delta} 
and Theorem \ref{thm:localisation} we obtain, for any $2 \leq p \leq \infty$,
\begin{equation} \label{w_localization}
\norm{\f w}_p^{2} \geq N^{(2/p - 1) \rho_b(\lambda) + o(1)}
\end{equation}
with probability $1 - o(1)$ (see Corollary~\ref{cor:lower_bound_w} below). In other words, the localization exponent $\gamma(\f w)$ from \eqref{def_gamma} satisfies $\gamma(\f w) \leq \rho_b(\lambda) + o(1)$. See Figure \ref{fig:gamma} for an illustration of the bound \eqref{w_localization} for $p = \infty$.  We remark that the exponent $\rho_b(\lambda)$ also describes the density of states at energy $\lambda$: under the above assumptions on $b$ and $\lambda$, for any interval $I$ containing $\lambda$ and satisfying $\xi \ll \abs{I} \ll 1$, the number of eigenvalues in $I$ is equal to $N^{\rho_b(\lambda) + o(1)} \abs{I}$ with probability $1 - o(1)$, as can be seen from Lemma \ref{lem:alpha_distr} \ref{itm:alpha_density} and Theorem \ref{thm:one_to_one}.

\subsection{Results -- the delocalized phase} \label{sec:results_delocalized}
Let $A$ be the adjacency matrix of $\bb G(N,d/N)$, as in Section \ref{sec:localized_results}.
For $0 < \kappa < 1/2$ define the spectral region
\begin{equation} \label{def_calS}
\cal S_\kappa \deq [-2 + \kappa, -\kappa] \cup [\kappa, 2 - \kappa]\,.
\end{equation}

\begin{theorem}[Delocalized phase] \label{thm:delocalization} 
For any $\nu>0$ and $\kappa>0$ there exists a constant $\cal C > 0$ such that the following holds. Suppose that
\begin{equation} \label{d_condition_deloc}
\cal C \sqrt{\log N} \leq d \leq (\log N)^{3/2}\,.
\end{equation}
Let $\f w$ be a normalized eigenvector of $A / \sqrt{d}$ with eigenvalue $\lambda \in \cal S_\kappa$. Then
\begin{equation} \label{complete delocalization}
\norm{\f w}_\infty^2 \leq N^{-1 + \kappa}
\end{equation}
with probability at least $1 - \cal CN^{-\nu}$.
\end{theorem}

In the delocalized phase, i.e.\ in $\cal S_\kappa$, we also show that the spectral measure of $A / \sqrt{d}$ at any vertex $x$ is well approximated by the spectral measure at the root of $\bb T_{d\alpha_x,d}$, the infinite rooted $(d\alpha_x,d)$-regular tree, whose root has $d \alpha_x$ children and all other vertices have $d$ children. This approximation is a local law, valid for intervals containing down to $N^\kappa$ eigenvalues. See Remark \ref{rem:spectral measure} as well as Remark \ref{rem:alpha_beta} and Appendix \ref{sec:mu} below for details.

\begin{remark} \label{rem:deloc_supercrit}
In \cite{HeKnowlesMarcozzi2018} it is shown that  \eqref{complete delocalization} holds with probability at least $1 - \cal CN^{-\nu}$ for \emph{all} eigenvectors provided that
\begin{equation} \label{he_assumption}
\cal C \log N \leq d \leq N/2\,.
\end{equation}
This shows that the upper bound in \eqref{d_condition_deloc} is in fact not restrictive.
\end{remark}

\begin{remark}[Optimality of \eqref{d_condition_deloc} and \eqref{he_assumption}] \label{rem:optimality_ass}
Both lower bounds in \eqref{d_condition_deloc} and \eqref{he_assumption} are optimal (up to the value of $\cal C$), in the sense that delocalization fails in each case if these lower bounds are relaxed. See Section \ref{sec:forks_intro} below. \end{remark}

We note that the domain $\cal S_\kappa$ is optimal, up to the choice of $\kappa > 0$. Indeed, as explained in Section \ref{sec:forks_intro} below, delocalization fails in the neighbourhood of the origin, owing to a proliferation highly localized tuning fork states. Similarly, we expect the delocalization to fail in the neighbourhoods of $\pm 2$, where the masses of the eigenvectors become concentrated on vertices $x$ with normalized degrees $\alpha_x$ close to $2$. The neighbourhoods of $0, \pm 2$ are also singled out as the regions where the self-consistent equation used to prove Theorem \ref{thm:delocalization} (see Lemma \ref{lem:self_consistent_G}) becomes unstable. This instability is directly related to the appearance of singularities in the spectral measure of the tree $\bb T_{d \alpha_x,d}$ (see \eqref{m_mu_st} and Figure \ref{fig:mu} for an illustration). The singularity near $0$ occurs when $\alpha_x$ is close to $0$, and the singularities near $\pm 2$ when $\alpha_x$ is close to $2$. See Figure \ref{fig:evector_simulation} for a simulation that demonstrates numerically the failure of delocalization outside of $\cal S_\kappa$.

\subsection{Extension to general sparse random matrices} \label{sec:extension}

Our results, Theorems \ref{thm:localisation}, \ref{thm:one_to_one}, and \ref{thm:delocalization}, hold also for the following family of sparse Wigner matrices. Let $A = (A_{xy})$ be the adjacency matrix of $\bb G(N,d/N)$ as above and $W=(W_{xy})$ be an independent Wigner matrix with bounded entries. That is, $W$ is Hermitian and its upper triangular entries $(W_{xy} \col x \leq y)$ are independent complex-valued random variables with mean zero and variance one, $\E \abs{W_{xy}}^2 = 1$, and $\abs{W_{xy}} \leq K$ almost surely for some constant $K$. Then we define the \emph{sparse Wigner matrix} $M = (M_{xy})$ as the Hadamard product of $A$ and $W$, with entries $M_{xy} \deq A_{xy} W_{xy}$. Since the entries of $M / \sqrt{d}$ are centred, it does not have a trivial eigenvalue like $A / \sqrt{d}$.

\begin{theorem}
\label{thm:sparse_Wigner_matrices} 
Let $M = (M_{xy})_{x,y \in [N]}$ be a sparse Wigner matrix. Define
\begin{equation} \label{def_alpha_general}
\alpha_x = \frac{1}{d} \sum_{y \in [N]} \abs{M_{xy}}^2.
\end{equation}
Theorems \ref{thm:localisation} and \ref{thm:delocalization} hold with \eqref{def_alpha_general} if $A$ is replaced with $M$, and Theorem \ref{thm:one_to_one} holds with \eqref{def_alpha_general} if $\lambda_{i + 1}(A/\sqrt{d})$, $\lambda_{N-i+1}(A/\sqrt{d})$, and $\lambda_i(A/\sqrt{d})$ are replaced with $\lambda_{i}(M / \sqrt{d})$, $\lambda_{N-i+1}(M / \sqrt{d})$, and $\lambda_i(M / \sqrt{d})$, respectively. Here, the constants $\cal C$ depend on $K$ in addition to $\nu$ and $\kappa$. 
\end{theorem}

The modifications to the proofs of Theorems \ref{thm:localisation} and \ref{thm:one_to_one} required to establish Theorem \ref{thm:sparse_Wigner_matrices} are minor and follow along the lines of \cite[Section 10]{ADK19}. The modification to the proof of Theorem~\ref{thm:delocalization} is trivial, since the assumptions of the general Theorem~\ref{thm:local_law} below include the sparse Wigner matrix $M$. We also remark that, with some extra work, one can relax the boundedness assumption on the entries of $W$, which we shall however not do here.

\subsection{The limits of sparseness and the scale $d \asymp \sqrt{\log N}$} \label{sec:forks_intro}
We conclude this section with a discussion on how sparse $\bb G$ can be for our results to remain valid. We show that all of our results -- Theorems \ref{thm:localisation}, \ref{thm:one_to_one}, and \ref{thm:delocalization} -- are wrong below the regime \eqref{optimal_regime}, i.e.\ if $d$ is smaller than order $\sqrt{\log N}$. Thus, our sparseness assumptions -- the lower bounds on $d$ from \eqref{d_assumption_localization} and \eqref{d_condition_deloc} -- are optimal (up to the factor $\log \log N$ in \eqref{d_assumption_localization} and the factor $\cal C$ in \eqref{d_condition_deloc}). The fundamental reason for this change of behaviour will turn out to be that the ratio $\abs{S_2(x)} / \abs{S_1(x)}$ concentrates if and only if $d \gg \sqrt{\log N}$, where $S_i(x)$ denotes the sphere in $\mathbb{G}$ of radius $i$ around $x$. This can be easily made precise with a well-known \emph{tuning fork} construction, detailed below.

In the critical and subcritical regime $1 \ll d = O(\log N)$, the graph $\bb G$ is in general not connected, but with probability $1 - o(1)$ it has a unique giant component $\bb G_{\mathrm{giant}}$ with at least $N (1 - \ee^{- d/4})$ vertices (see Corollary \ref{cor:components} below). Moreover, the spectrum of $A / \sqrt{d}$ restricted to the complement of the giant component is contained in the $O\pb{\frac{\sqrt{\log N}}{d}}$-neighbourhood of the origin (see Corollary \ref{cor:small_components} below). Since we always assume $d \geq \cal C \sqrt{\log N}$ and we only consider eigenvalues in $\R \setminus [-\kappa,\kappa]$, we conclude that all of our results listed above only pertain to the eigenvalues and eigenvectors of the giant component.

For $D = 0,1,2,\dots$ we introduce a \emph{star\footnote{For simplicity we only consider stars, but the same argument can be applied to arbitrary trees.} tuning fork of degree D rooted in $\bb G_{\mathrm{giant}}$}, or \emph{$D$-tuning fork} for short, which is obtained by taking two stars with central degree $D$ and connecting their hubs to a common base vertex in $\bb G_{\mathrm{giant}}$. We refer to Figure \ref{fig:star} for an illustration and Definition \ref{def:star} below for a precise definition.

\begin{figure}[!ht]
\begin{center}
\includegraphics{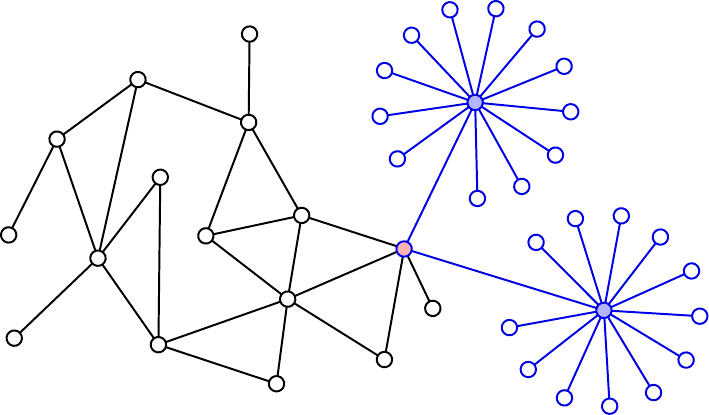}
\end{center}
\caption{A star tuning fork of degree 12 rooted in a graph. The tuning fork is highlighted in blue. Its base is filled with red and its two hubs are filled with blue.
 \label{fig:star}}
\end{figure}

It is not hard to see that every $D$-tuning fork gives rise to two eigenvalues $\pm \sqrt{D/d}$ of $A / \sqrt{d}$ restricted to $\bb G_{\mathrm{giant}}$, whose associated eigenvectors are supported on the stars (see Lemma \ref{lem:forks} below). We denote by $\Sigma \deq \h{\sqrt{D/d} \col \text{a $D$-tuning fork exists}}$ the spectrum of $A / \sqrt{d}$ restricted to $\bb G_{\mathrm{giant}}$ generated by the tuning forks. Any eigenvector associated with an eigenvalue $\sqrt{D/d} \in \Sigma$ is localized on precisely $2D + 2$ vertices. Thus, $D$-tuning forks provide a simple way of constructing localized states. Note that this is a very basic form of concentration of mass, supported at the periphery of the graph on special graph structures, and is unrelated to the much more subtle concentration in the semilocalized phase described in Section \ref{sec:localized_results}.

For $d > 0$ and $D \in \N$ we now estimate the number of $D$-tuning forks in $\bb G(N,d/N)$, which we denote by $F(d,D)$. The following result is proved in Appendix \ref{sec:tuning_forks}. 

\begin{lemma}[Number of $D$-tuning forks] \label{lem:star_localization}
Suppose that $1 \ll d = b \log N = O(\log N)$ and $0 \leq D \ll \log N / \log \log N$. Then $F(d,D) = N^{1 - 2b - 2b D + o(1)}$ with probability $1 - o(1)$.
\end{lemma}

Defining $D_* \deq \frac{\log N}{2d} - 1$, we immediately deduce the following result.

\begin{corollary}
For any constant $\epsilon > 0$ with probability $1 - o(1)$ the following holds. If $D_* \leq -\epsilon$ then $\Sigma = \emptyset$. If $D_* \geq \epsilon$ then $\Sigma = \h{\pm \sqrt{D/d} \col D \in \N, D \leq D_* (1 + o(1))}$.
\end{corollary}

We deduce that if $d \leq (1/2 - \epsilon) \log N$ then $\Sigma \neq \emptyset$ and hence the delocalization for all eigenvectors from Remark \ref{rem:deloc_supercrit} fails. Hence, the lower bound \eqref{he_assumption} is optimal up to the value of $\cal C$.

Similarly, for $d \gg \sqrt{\log N}$ the set $\Sigma$ is in general nonempty, but we always have $\Sigma \subset [-\kappa, \kappa]$ for any fixed $\kappa > 0$, so that eigenvalues from $\Sigma$ do not interfere with the statements of Theorems \ref{thm:localisation}, \ref{thm:one_to_one}, and \ref{thm:delocalization}. On the other hand, if $d = \sqrt{\log N} / t$ for constant $t$, we find that $\Sigma$ is asymptotically dense in the interval $[-t/\sqrt{2}, t / \sqrt{2}]$. Since the conclusions of Theorems \ref{thm:localisation}, \ref{thm:one_to_one}, and \ref{thm:delocalization} are obviously wrong for any eigenvalue from $\Sigma$, they must all be wrong for large enough $t$. This shows that the lower bounds $d$ from \eqref{d_assumption_localization} and \eqref{d_condition_deloc} are optimal (up to the factor $\log \log N$ in \eqref{d_assumption_localization} and the factor $\cal C$ in \eqref{d_condition_deloc}).

In fact, the emergence of the tuning fork eigenvalues of order one and the failure of all of our proofs has the same underlying root cause, which singles out the scale $d \asymp \sqrt{\log N}$ as the scale below which the concentration of the ratio
\begin{equation} \label{S2S1_conc}
\abs{S_2(x)} / \abs{S_1(x)} = d (1 + o(1))
\end{equation}
fails for vertices $x$ satisfying $D_x \asymp d$. Clearly, to have a $D$-tuning fork with $D \asymp d$, \eqref{S2S1_conc} has to fail at the hubs of the stars. Moreover, \eqref{S2S1_conc} enters our proofs of both the semilocalized and the delocalized phase in a crucial way. For the former, it is linked to the validity of the local approximation by the $(D_x,d)$-regular tree from Appendix \ref{sec:mu}, which underlies also the construction of the localization profile vectors (see e.g.\ \eqref{S_tau_S_1} below). For the latter, in the language of Definition \ref{def_typical_vert} below, it is linked to the property that most neighbours of any vertex are typical (see Proposition \ref{pro:typcial_vertices} \ref{item:a2} below).

\paragraph{Acknowledgements}
The authors would like to thank Simone Warzel for helpful discussions. The authors gratefully acknowledge funding from the European Research Council (ERC) under the European Union’s Horizon 2020 research and innovation programme (grant agreement No.\ 715539\_RandMat) and from the Swiss National Science Foundation through the NCCR SwissMAP grant.

\section{Basic definitions and overview of proofs}
In this preliminary section we introduce some basic notations and definitions that are used throughout the paper, and give an overview of the proofs of Theorems \ref{thm:localisation} (semilocalized phase) and \ref{thm:delocalization} (delocalized phase). These proofs are unrelated and, thus, explained separately.
For simplicity, in this overview we only consider qualitative error terms of the form $o(1)$, although all of our estimates are in fact quantitative.

\subsection{Basic definitions}
We write $\N = \{0,1,2,\dots\}$. We set $[n] \defeq \{1, \ldots, n\}$
for any $n \in \N^*$ and $[0] \defeq \emptyset$.
We write $\abs{X}$ for the cardinality of a finite set $X$. 
We use $\ind{\Omega}$ as symbol for the indicator function of the event $\Omega$. 

Vectors in $\R^N$ are denoted by boldface lowercase Latin letters like $\f u$, $\f v$ and $\f w$. We use the notation $\f v = (v_x)_{x \in [N]} \in \R^N$ for the entries of a vector. We denote by $\supp \f v \deq \{x \in [N] \col v_x \neq 0\}$ the support of a vector $\f v$. We denote by $\scalar{\f v}{\f w} = \sum_{x \in [N]} v_x w_x$ the Euclidean scalar product on $\R^N$ and by $\norm{\f v} = \sqrt{\scalar{\f v}{\f v}}$ the induced Euclidean norm. 
For a matrix $M \in \R^{N \times N}$, $\norm{M}$ is its operator norm induced by the Euclidean norm on $\R^N$.
For any $x \in [N]$, we define the standard basis vector $\f 1_x \defeq (\delta_{xy})_{y \in [N]} \in \R^N$.
To any subset $S \subset [N]$ we assign the vector $\f 1_S\in \R^N$ given by $\f 1_S \defeq \sum_{x \in S} \f 1_x$. 
In particular, $\f 1_{\{ x\}} = \f 1_x$.

We use blackboard bold letters to denote graphs. Let $\bb H = (V(\bb H), E(\bb H))$ be a (simple, undirected) graph on the vertex set $V(\bb H) = [N]$. We often identify a graph $\bb H$ with its set of edges $E(\bb H)$. We denote by $A^{\bb H} \in \{0,1\}^{N \times N}$ the adjacency matrix of $\bb H$. For $r \in \N$ and $x \in [N]$, we denote by $B_r^{\bb H}(x)$ the closed ball of radius $r$ around $x$ in the graph $\bb H$, i.e.\ the set of vertices at distance (with respect to $\bb H$) at most $r$ from the vertex $x$.
We denote the sphere of radius $r$ around the vertex $x$ by $S_r^{\bb H}(x) \deq B_r^{\bb H}(x) \setminus B_{r - 1}^{\bb H}(x)$. We denote by $D_x^{\bb H}$ the degree of the vertex $x$ in the graph $\bb H$. For any subset $V \subset [N]$, we denote by $\bb H \vert_V$ the subgraph induced by $\bb H$ on $V$. If $\bb H$ is a subgraph of $\bb G$ then we denote by $\bb G \setminus \bb H$ the graph on $[N]$ with edge set $E(\bb G) \setminus E(\bb H)$. In the above definitions, if the graph $\bb H$ is the Erd\H{o}s-R\'enyi graph $\bb G$, we systematically omit the superscript $\bb G$.

The following notion of very high probability is a convenient shorthand used throughout the paper. It simplifies considerably the probabilistic statements of the kind that appear in Theorems \ref{thm:localisation}, \ref{thm:one_to_one}, and \ref{thm:delocalization}. It also introduces two special symbols, $\nu$ and $\cal C$, which appear throughout the rest of the paper.

\begin{definition} \phantomsection \label{def:very_high_probability_def} 
Let $\Xi \equiv \Xi_{N,\nu}$ be a family of events parametrized by $N \in \N$ and $\nu > 0$. We say that $\Xi$ \emph{holds with very high probability} if for every $\nu > 0$ there exists $\Cnu \equiv \Cnu_\nu$ such that
\begin{equation*}
\P(\Xi_{N,\nu}) \geq 1 - \Cnu_\nu N^{-\nu}
\end{equation*}
for all $N \in \N$.
\end{definition}

\begin{convention}
In statements that hold with very high probability, we use the special symbol $\cal C \equiv \cal C_\nu$ to denote a generic positive constant depending on $\nu$ such that the statement holds with probability at least $1 - \cal C_\nu N^{-\nu}$ provided $\cal C_\nu$ is chosen large enough. 
Thus, the bound \emph{$\abs{X} \leq \Cnu Y$ with very high probability} means that, for each $\nu >0$, there is a constant $\Cnu_\nu>0$, depending on $\nu$, such that 
\[ \P \big( \abs{X} \leq \Cnu_\nu Y \big) \geq 1 - \cal C_\nu N^{-\nu} \] 
for all $N \in \N$. Here, $X$ and $Y$ are allowed to depend on $N$.  We also write $X = \cal O(Y)$ to mean $\abs{X} \leq \cal C Y$.
\end{convention}

We remark that the notion of very high probability from Definition \ref{def:very_high_probability_def} survives a union bound involving $N^{O(1)}$ events.  We shall tacitly use this fact throughout the paper. Moreover, throughout the paper, the constant $\cal C \equiv \cal C_\nu$ in the assumptions \eqref{d_assumption_localization} and \eqref{d_condition_deloc} is always assumed to be large enough.

\subsection{Overview of proof in semilocalized phase} \label{sec:sketch_localized}
The starting point of the proof of Theorem~\ref{thm:localisation} is the following simple observation. Suppose that $M$ is a Hermitian matrix with eigenvalue $\lambda$ and associated eigenvector $\f w$. Let $\Pi$ be an orthogonal projection and write $\ol \Pi \deq I - \Pi$. If $\lambda$ is not an eigenvalue of $\ol \Pi M \ol \Pi$ then from $(M - \lambda) \f w = 0$ we deduce
\begin{equation} \label{intro_ev}
\ol \Pi \f w = - (\ol \Pi M \ol \Pi - \lambda)^{-1} \ol \Pi M \Pi \f w\,.
\end{equation}
If $\Pi$ is an eigenprojection of $M$ whose range contains the eigenspace of $\lambda$ (for instance $\Pi = \f w \f w^*$ if $\lambda$ is simple) then clearly both sides of \eqref{intro_ev} vanish. The basic idea of our proof is to apply an approximate version of this observation to $M = A / \sqrt{d}$, by choosing $\Pi$ appropriately, and showing that the left-hand side of \eqref{intro_ev} is small by estimating the right-hand side.

In fact, we choose\footnote{This projection $\Pi$ is denoted by $\Pi^\tau_{\lambda,\delta}$ in the proof of Theorem \ref{thm:general} below.}
\begin{equation} \label{def_Pi_intro}
\Pi \deq \sum_{x \in \cal W_{\lambda,\delta}} \f v(x) \, \f v(x)^*\,,
\end{equation}
where $\cal W_{\lambda,\delta}$ is the set \eqref{def_calW} of resonant vertices at energy $\lambda$, and $\f v(x)$ is the exponentially decaying localization profile from Theorem \ref{thm:localisation}. The proof then consists of two main ingredients:
\begin{enumerate}[label=(\alph*)]
\item \label{item:pf_a}
$\norm{\ol \Pi M \Pi} = o(1)$; 
\item \label{item:pf_b}
$\ol \Pi M \ol \Pi$ has a spectral gap around $\lambda$.
\end{enumerate}
Informally, \ref{item:pf_a} states that $\Pi$ is close to a spectral projection of $M$, as $\ol \Pi M \Pi = [M,\Pi] \Pi$ quantifies the noncommutativity of $M$ and $\Pi$ on the range of $\Pi$. Similarly, \ref{item:pf_b} states that $\Pi$ projects roughly onto an eigenspace of $M$ of energies near $\lambda$.
Plugging \ref{item:pf_a} and \ref{item:pf_b} into \eqref{intro_ev} yields an estimate on $\norm{\ol \Pi \f w}$ from which Theorem \ref{thm:localisation} follows easily. Thus, the main work of the proof is to establish the properties \ref{item:pf_a} and \ref{item:pf_b} for the specific choice of $\Pi$ from \eqref{def_Pi_intro}.

The construction of the localization profile $\f v(x)$ uses the \emph{pruned graph} $\bb G_\tau$ from \cite{ADK19}, a subgraph of $\bb G$ depending on a threshold $\tau > 1$,
which differs from $\bb G$ by only a small number of edges and whose balls of radius $r_\star$ around the vertices of $\cal V_\tau \deq \h{x \col \alpha_x \geq \tau}$ are disjoint 
(see Proposition \ref{prop:subgraph_separating_large_degrees} below).
Now we define the vector $\f v(x) \deq \f v^\tau_+(x)$, where, for $\sigma = \pm$ and $\tau > 1$,
\begin{equation} \label{def_bv_intro}
\f v^\tau_\sigma(x) \deq \sum_{i = 0}^{r_\star} \sigma^i u_i(x) \f 1_{S_i^{\bb G_\tau}(x)} / \norm{\f 1_{S_i^{\bb G_\tau}(x)}}\,, \qquad u_i(x) \deq \frac{\sqrt{\alpha_x}}{(\alpha_x - 1)^{i/2}} \, u_0 \quad (1 \leq i \leq r_\star)\,.
\end{equation}
The motivation behind this choice is explained in Appendix \ref{sec:mu}: with high probability, the $r_\star$-neighbourhood of $x$ in $\bb G_\tau$ looks roughly like that of the root of infinite regular tree $\bb T_{D_x, d}$ whose root has $D_x$ children and all other vertices $d$ children. The adjacency matrix of $\bb T_{D_x, d}$ has the exact eigenvalues $\pm \sqrt{d} \Lambda(\alpha_x)$ with the corresponding eigenvectors given by \eqref{def_bv_intro} with $\bb G_\tau$ replaced with $\bb T_{D_x, d}$.

The central idea of our proof is the introduction of a block diagonal approximation of the pruned graph.
Define the orthogonal projections 
\begin{equation*}
\Pi^\tau \deq \sum_{x \in \cal V_{2 + o(1)}} \sum_{\sigma = \pm} \f v^\tau_\sigma(x) \f v^\tau_\sigma(x)^*\,, \qquad \ol \Pi^\tau  \deq I - \Pi^\tau\,.
\end{equation*}
The range of $\Pi$ from \eqref{def_Pi_intro} is a subspace of the range of $\Pi^\tau$, i.e.\ $\Pi \Pi^\tau = \Pi$. The interpretation of $\Pi^\tau$ is the orthogonal projection onto all localization profiles around vertices $x$ with normalized degree at least $2 + o(1)$, which is precisely the set of vertices around which one can define an exponentially decaying localization profile. Now we define the \emph{block diagonal approximation of the pruned graph} as
\begin{equation}
\wh H^\tau \deq \sum_{x \in \cal V_{2 + o(1)}} \sum_{\sigma = \pm} \sigma \Lambda(\alpha_x) \f v^\tau_\sigma(x) \f v^\tau_\sigma(x)^* + \ol \Pi^\tau H^\tau \ol \Pi^\tau\,;
\end{equation}
here we defined the centred and scaled adjacency matrix $H^\tau \deq A^{\bb G_\tau} / \sqrt{d} - E^\tau$, where $E^\tau$ is a suitably chosen matrix that is close to $\E A^{\bb G} / \sqrt{d}$ and preserves the locality of $A^{\bb G_\tau}$ in balls around the vertices of $\cal V_\tau$. 
In the subspace spanned by the localization profiles $\h{\f v^\tau_\sigma(x) \col \sigma = \pm, x \in \cal V_{2 + o(1)}}$, $\wh H^\tau$ is diagonal with eigenvalues $\sigma \Lambda(\alpha_x)$. In the orthogonal complement, it is equal to $H^\tau$. The off-diagonal blocks are zero. The main work of our proof consists in an analysis of $\wh H^\tau$.

In terms of $\wh H^\tau$, abbreviating $H \deq (A^{\bb G} - \E A^{\bb G}) / \sqrt{d}$, the problem of showing \ref{item:pf_a} and \ref{item:pf_b} reduces to showing
\begin{enumerate}[label=(\alph*)]
\setcounter{enumi}{2}
\item \label{item:pf_c}
$\norm{H - \wh H^\tau} = o(1)$,
\item\label{item:pf_d}
$\norm{\ol \Pi^\tau H^\tau \ol \Pi^\tau} \leq 2 + o(1)$.
\end{enumerate}
Indeed, ignoring minor issues pertaining to the centring $\E A^{\bb G}$, we replace $M = A^{\bb G} / \sqrt{d}$ with $H$ in \ref{item:pf_a} and \ref{item:pf_b}. Then \ref{item:pf_a} follows immediately from \ref{item:pf_c}, since $\ol \Pi H \Pi = \norm{\ol \Pi \wh H^\tau \Pi} + o(1) = o(1)$, as $\ol \Pi \wh H^\tau \Pi = 0$ by the block structure of $\wh H^\tau$ and the relation $\Pi^\tau \Pi = \Pi$. To show \ref{item:pf_b}, we note that the $\Pi^\tau$-block of $\wh H^\tau$, $\Pi^\tau \wh H^\tau \Pi^\tau = \sum_{x \in \cal V_{2 + o(1)}} \sum_{\sigma = \pm} \sigma \Lambda(\alpha_x) \f v^\tau_\sigma(x) \f v^\tau_\sigma(x)^*$, trivially has a spectral gap: $\ol \Pi \Pi^\tau H^\tau \Pi^\tau \ol \Pi$ has no eigenvalues in the $\delta$-neighbourhood of $\lambda$, simply because the projection $\ol \Pi$ removes the projections $\f v^\tau_\sigma(x) \f v^\tau_\sigma(x)^*$ with eigenvalues $\sigma \Lambda(\alpha_x)$ in the $\delta$-neighbourhood of $\lambda$. Moreover, the $\ol \Pi^\tau$-block also has such a spectral gap by \ref{item:pf_d} and $\lambda > 2 + o(1)$. Hence, by \ref{item:pf_c}, we deduce the desired spectral gap \ref{item:pf_b}.

Thus, what remains is the proof of \ref{item:pf_c} and \ref{item:pf_d}. To prove \ref{item:pf_c}, we prove $\norm{H - H^\tau} = o(1)$ and $\norm{H^\tau - \wh H^\tau} = o(1)$. The bound $\norm{H - H^\tau} = o(1)$ follows from a detailed analysis of the graph $\bb G \setminus \bb G_\tau$ removed from $\bb G$ to obtain the pruned graph $\bb G_\tau$, which we decompose as a union of a graph of small maximal degree and a forest, to which standard estimates of adjacency matrices of graphs
can be applied (see Lemma \ref{lem:estimate_cut_graph} below). 
To prove $\norm{H^\tau - \wh H^\tau} = o(1)$, we first prove that $\f v^\tau_\sigma(x)$ is an approximate eigenvector of $H^\tau$ with approximate eigenvalue $\sigma \Lambda(\alpha_x)$ (see Proposition \ref{prop:proof_lower_bound_approximate_eigenvector} below).
Then we deduce $\norm{H^\tau - \wh H^\tau} = o(1)$ using that the balls $B_{2r_\star}(x)$, $x \in \cal V_{2 + o(1)}$, are disjoint and the locality of the operator $H^\tau$ (see Lemma \ref{lem:estime_block_matrix} below). Thus we obtain \ref{item:pf_c}.

Finally, we sketch the proof of \ref{item:pf_d}. The starting point is an observation going back to \cite{BBK1,ADK19}: from an estimate on the spectral radius of the nonbacktracking matrix associated with $H$ from \cite{BBK1} and an Ihara--Bass-type formula relating the spectra of $H$ and its nonbacktracking matrix from \cite{BBK1}, we obtain the quadratic form inequality $\abs{H} \leq I + Q + o(1)$
with very high probability, where $Q = \diag(\alpha_x \col x \in [N])$, $\abs{H}$ is the absolute value of the Hermitian matrix $H$, and $o(1)$ is in the sense of operator norm (see Proposition \ref{prop:operator_upper_bound} below). Using \ref{item:pf_c}, we deduce the inequality
\begin{equation} \label{IB_intro}
\abs{\wh H^\tau} \leq I + Q + o(1)\,.
\end{equation}
To estimate $\norm{\ol \Pi^\tau H^\tau \ol \Pi^\tau}$, we take a normalized eigenvector $\f w$ of $\ol \Pi^\tau H^\tau \ol \Pi^\tau$ with maximal eigenvalue $\lambda > 0$. Thus, $\f w \perp \f v^\tau_\pm(x)$ for all $x \in \cal V_{2 + o(1)}$. We estimate $\ol \Pi^\tau H^\tau \ol \Pi^\tau$ from above (an analogous argument yields an estimate from below) using \eqref{IB_intro} to get
\begin{equation} \label{lambda_est_intro}
\lambda \leq 1 + o(1) + \sum_x \alpha_x w_x^2 \leq 1 + \tau + o(1) + \max_x \alpha_x \sum_{x \in \cal V_\tau} w_x^2\,.
\end{equation}
Choosing $\tau = 1 + o(1)$, we see that \ref{item:pf_d} follows provided that we can show that
\begin{equation} \label{deloc_intro}
\sum_{x\in \cal V_\tau}w_{x}^{2} = o(1 / \log N)\,,
\end{equation}
since $\max_x \alpha_x \leq \cal C \log N$ with very high probability.

The estimate \eqref{deloc_intro} is a \emph{delocalization bound}, in the vertex set $\cal V_\tau$, for any eigenvector $\f w$ of $\wh H^\tau$ that is orthogonal to $\f v_\pm^\tau(x)$ for all $x \in \cal V_{2 + o(1)}$ and whose associated eigenvalue is larger than $2 \tau + o(1)$. 
It crucially relies on the assumption that $\f w \perp \f v_\pm^\tau(x)$ for all $x \in \cal V_{2 + o(1)}$, without which it is false (see Proposition \ref{thm:weak_delocalisation} below).
The underlying principle behind its proof is the same as that of the Combes--Thomas estimate \cite{combes1973asymptotic}: the Green function $((\lambda - Z)^{-1})_{ij}$ of a local operator $Z$ at a spectral parameter $\lambda$ separated from the spectrum of $Z$ decays exponentially in the distance between $i$ and $j$, at a rate inversely proportional to the distance from $\lambda$ to the spectrum of $Z$. We in fact use a radial form of a Combes--Thomas estimate, where $Z$ is the tridiagonalization of a local restriction of $\wh H^\tau$ around a vertex $x \in \cal V_\tau$ (see Appendix \ref{sec:mu}) and $i,j$ index radii of concentric spheres. The key observation is that, by the orthogonality assumption on $\f w$, the Green function $((\lambda - Z)^{-1})_{i r_\star}$, $0 \leq i < r_\star$, and the eigenvector components in the radial basis $u_i$, $0 \leq i < r_\star$, satisfy the same linear difference equation. Thus we obtain exponential decay for the components $u_i$, which yields $u_0^2 \leq o(1/\log N) \sum_{i = 0}^{r_*} u_i^2$. Going back to the original vertex basis, this implies that $w_x^2 \leq o(1/\log N) \|\f w|_{B_{2r_\star}^{\bb G_\tau}(x)}\|^2$ for all $x \in \cal V_\tau$, from which \eqref{deloc_intro} follows since the balls $B_{2r_\star}^{\bb G_\tau}(x)$, $x \in \cal V_\tau$, are disjoint.

\subsection{Overview of proof in delocalized phase}

The delocalization result of Theorem \ref{thm:delocalization} is an immediate consequence of a \emph{local law} for the matrix $A / \sqrt{d}$, which controls the entries of the Green function
\begin{equation*}
G \equiv G(z) \deq \pb{A / \sqrt{d} - z}^{-1}
\end{equation*}
in the form of high-probability estimates, for \emph{spectral scales} $\im z$ down to the optimal scale $1/N$, which is the typical eigenvalue spacing. Such a local law was first established for $d \gg (\log N)^6$ in \cite{EKYY1} and extended down to $d \geq \cal C \log N$ in \cite{HeKnowlesMarcozzi2018}. In both of these works, the diagonal entries of $G$ are close to the Stieltjes transform of the semicircle law. In contrast, in the regime \eqref{optimal_regime} 
the diagonal entry $G_{xx}$ is close to the Stieltjes transform of the spectral measure at the root of an infinite $(D_x,d)$-regular tree. Hence, $G_{xx}$ does not concentrate around a deterministic quantity. 

The basic approach of the proof is the same as for any local law: derive an approximate self-consistent equation with very high probability, solve it using a stability analysis, and perform a bootstrapping from large to small values of $\im z$. For a set $T \subset [N]$ denote by $A^{(T)}$ the adjacency matrix of the graph $\bb G$ where the vertices of $T$ (and all incident edges) have been removed, and denote by $G^{(T)} = \pb{A^{(T)} / \sqrt{d} - z}^{-1}$ the associated Green function. In order to understand the emergence of the self-consistent equation, it is instructive to consider the toy situation where, for a given vertex $x$, all neighbours $S_1(x)$ are in different connected components of $A^{(x)}$. This is for instance the case if $\bb G$ is a tree. On the \emph{global scale}, where $\im z$ is large enough, this assumption is in fact valid to a good approximation, since the neighbourhood of $x$ is with high probability a tree. Then a simple application of Schur's complement formula and the resolvent identity yield
\begin{equation} \label{simple_approx}
\frac{1}{G_{xx}} = -z - \frac{1}{d} \sum_{y \in S_1(x)} G_{yy}^{(x)} \,, \qquad
G_{yy}^{(x)} - G_{yy} = (G_{yy}^{(x)})^2 \frac{1}{d} G_{xx}\,.
\end{equation}
Thus, on the global scale, using that $G$ is bounded, we obtain the self-consistent equation
\begin{equation} \label{sc_equation_intro}
\frac{1}{G_{xx}} = -z - \frac{1}{d} \sum_{y \in S_1(x)} G_{yy} + o(1)
\end{equation}
with very high probability.

It is instructive to solve the self-consistent equation \eqref{sc_equation_intro} in the family $(G_{xx})_{x \in [N]}$ on the global scale. To that end, we introduce the notion of \emph{typical vertices},
which is roughly the set $\cal T = \h{x \in [N] \col \alpha_x = 1 + o(1)}$. (In fact, as explained below, the actual definition for local scales has to be different; see \eqref{def_calT_intro} below.) A simple argument shows that with very high probability most neighbours of any vertex are typical.
With this definition, we can try to solve \eqref{sc_equation_intro} on the global scale as follows. From the boundedness of $G$ we obtain a self-consistent equation for the vector $(G_{xx})_{x \in \cal T}$ that reads
\begin{equation} \label{sc_intro_naive}
\frac{1}{G_{xx}} = -z - \sum_{y \in \cal T} \frac{1}{d} A_{xy} G_{yy} + \zeta_x\,, \qquad \zeta_x = o(1)\,.
\end{equation}
It is not hard to see that the equation \eqref{sc_intro_naive} has a unique solution, which satisfies $G_{xx} = m + o(1)$ for all $x \in \cal T$. Here $m$ is the Stieltjes transform of the semicircle law, which satisfies $m = \frac{1}{-z - m}$. Plugging this solution back into \eqref{sc_equation_intro} and using that most neighbours of any vertex are typical shows that for $x \notin \cal T$ we have $G_{xx} = m_{\alpha_x} + o(1)$, where $m_\alpha \deq \frac{1}{-z - \alpha m}$. One readily finds (see Appendix \ref{sec:mu} below) that $m_{\alpha_x}$ is Stieltjes transform of the spectral measure of the infinite $(D_x,d)$-regular tree at the root.

The first main difficulty of the proof is to provide a derivation of identities of the form \eqref{simple_approx} (and hence a self-consistent equation of the form \eqref{sc_equation_intro}) on the local scale $\im z \ll 1$. We emphasize that the above derivation of \eqref{simple_approx} is completely wrong on the local scale. Unlike on the global scale, on the local scale the behaviour of the Green function is not governed by the local geometry of the graph, and long cycles contribute to $G$ in an essential way. In particular, eigenvector delocalization, which follows from the local law, is a global property of the graph and cannot be addressed using local arguments; it is in fact wrong outside of the region $\cal S_\kappa$, although the above derivation is insensitive to the real part of $z$.

We address this difficulty by replacing the identities \eqref{simple_approx} with the following argument, which ultimately provides an a posteriori justification of approximate versions of \eqref{simple_approx} with very high probability, provided we are in the region $\cal S_\kappa$. We make an a priori assumption that the entries of $G$ are bounded with very high probability; we propagate this assumption from large to small scales using a standard bootstrapping argument and the uniform boundedness of the density of the spectral measure associated with $m_\alpha$.
It is precisely this uniform boundedness requirement that imposes the restriction to $\cal S_\kappa$ in our local law (as explained in Remark \ref{rem:optimality_ass}, this restriction is necessary).
The key tool that replaces the simpleminded approximation \eqref{simple_approx} is a series of large deviation estimates for sparse random vectors proved in \cite{HeKnowlesMarcozzi2018}, which, as it turns out, are effective for the full optimal regime \eqref{optimal_regime}. Thus, under the bootstrapping assumption that the entries of $G$ are bounded, we obtain \eqref{simple_approx} (and hence also \eqref{sc_equation_intro}), with some additional error terms, with very high probability.

The second main difficulty of the proof is that, on the local scale and for sparse graphs, the self-consistent equation \eqref{sc_intro_naive}, which can be derived from \eqref{sc_equation_intro} as explained above, is not stable enough to be solved in $(G_{xx})_{x \in \cal T}$. This problem stems from the sparseness of the graphs that we are considering, and does not appear in random matrix theory for denser (or even heavy-tailed) matrices. Indeed, the stability estimates of \eqref{sc_intro_naive} carry a logarithmic factor, which is usually of no concern in random matrix theory but is deadly for the sparse regime of this paper. This is a major obstacle and in fact ultimately dooms the self-consistent equation \eqref{sc_intro_naive}. To explain the issue, write the sum in \eqref{sc_intro_naive} as $\sum_y S_{xy} G_{yy}$, where $S$ is the $\cal T \times \cal T$ matrix $S_{xy} = \frac{1}{d} A_{xy}$. Writing $G_{xx} = m + \epsilon_x$, plugging it into \eqref{sc_intro_naive}, and expanding to first order in $\epsilon_x$, we obtain, using the definition of $m$, 
that $\epsilon_x = -m^2 ((I - m^2 S)^{-1} \zeta)_x$. Thus, in order to deduce smallness of $\epsilon_x$ from the smallness of $\zeta_x$, we need an estimate on the norm\footnote{We write $\norm{\cdot}_{p \to p}$ for the operator norm on $\ell^p$.} $\norm{(I - m^2 S)^{-1}}_{\infty \to \infty}$. In Appendix \ref{sec:instability} below we show that for typical $S$, $\re z \in \cal S_\kappa$, and small enough $\im z$, we have 
\begin{equation} \label{instability_intro}
\frac{\log N}{C (\log \log N)^2} \leq \norm{(I - m^2 S)^{-1}}_{\infty \to \infty} \leq C_\kappa \log N
\end{equation}
for some universal constant $C$ and some constant $C_\kappa$ depending on $\kappa$. 
In our context, where $\zeta_x$ is small but much larger than the reciprocal of the lower bound of \eqref{instability_intro}, such a logarithmic factor is not affordable.

To address this difficulty, we avoid passing by the form \eqref{sc_intro_naive} altogether, as it is doomed by \eqref{instability_intro}. The underlying cause for the instability of \eqref{sc_intro_naive} is the inhomogeneous local structure of the matrix $S$, which is a multiple of the adjacency matrix of a sparse graph. Thus, the solution is to derive a self-consistent equation of the form \eqref{sc_intro_naive} but with an \emph{unstructured} $S$, which has constant entries. The basic intuition is to replace the \emph{local average} $\frac{1}{d} \sum_{y \in S_1(x)} G_{yy}^{(x)}$ in the first identity of \eqref{simple_approx} with the \emph{global average} $\frac{1}{N} \sum_{y \neq x} G_{yy}^{(x)}$. Of course, in general these two are not close, but we can include their closeness into the definition of a typical vertex. Thus, we define the set of typical vertices as
\begin{equation} \label{def_calT_intro}
\cal T \deq \hbb{x \in [N] \col \alpha_x = 1 + o(1) \,,\, \frac{1}{d} \sum_{y \in S_1(x)} G_{yy}^{(x)} = \frac{1}{N} \sum_{y \neq x} G_{yy}^{(x)} + o(1)}\,.
\end{equation}
The main work of the proof is then to prove the following facts with very high probability.
\begin{enumerate}[label=(\alph*)]
\item \label{item:pf2_a}
Most vertices are typical.
\item \label{item:pf2_b}
Most neighbours of any vertex are typical.
\end{enumerate}
With \ref{item:pf2_a} and \ref{item:pf2_b} at hand, we explain how to conclude the proof. Using \ref{item:pf2_a} and the approximate version of \eqref{simple_approx} established above, we deduce the self-consistent equation for typical vertices,
\begin{equation*}
\frac{1}{G_{xx}} = -z - \frac{1}{\abs{\cal T}} \sum_{y \in \cal T} G_{yy} + o(1)\,, \qquad x \in \cal T\,,
\end{equation*}
which, unlike \eqref{sc_intro_naive}, is stable (see Lemma \ref{lem:stability} below) and can be easily solved to show that $G_{xx} = m + o(1) = m_{\alpha_x} + o(1)$ for all $x \in \cal T$. Moreover, if $x \notin \cal T$ then we obtain from \eqref{simple_approx} and \ref{item:pf2_b} that
\begin{equation*}
\frac{1}{G_{xx}} = -z - \frac{1}{d} \sum_{y \in S_1(x) \cap \cal T} G_{yy}^{(x)} + o(1) = -z - \alpha_x m + o(1)\,,
\end{equation*}
where we used that $G_{yy} = m + o(1)$ for $y \in \cal T$. This shows that $G_{xx} = m_{\alpha_x} + o(1)$ for all $x \in [N]$ with very high probability, and hence concludes the proof.

What remains, therefore, is the proof of \ref{item:pf2_a} and \ref{item:pf2_b}; see  Proposition \ref{pro:typcial_vertices} below for a precise statement. Using the bootstrapping assumption of boundedness of the entries of $G$, it is not hard to estimate the probability $\P(x \in \cal T)$, which we prove to be $1 - o(1)$, although $\{x \in \cal T\}$ does not hold with very high probability (this characterizes the critical and subcritical regimes). Now if the events $\{x \in \cal T\}$, $x \in [N]$, were all independent, it would then be a simple matter to deduce \ref{item:pf2_a} and~\ref{item:pf2_b}.

The most troublesome source of dependence among the events $\{x \in \cal T\}$, $x \in [N]$, is the Green function $G_{yy}^{(x)}$ in the definition of $\cal T$. Thus, the main difficulty of the proof is a decoupling argument that allows us to obtain good decay for the probability $\P(T \subset \cal T)$ in the size of $T$. This decay can only work up to a threshold in the size of $T$, beyond which the correlations among the different events kick in. In fact, we essentially prove that
\begin{equation} \label{T_estimate_intro}
\P(T \subset \cal T) \leq \ee^{- o(1) d \abs{T}} + \cal C N^{-\nu} \qquad \text{for} \quad \abs{T} = o(d)\,;
\end{equation}
see Lemma \ref{lem:decoupling}. Choosing the largest possible $T$, $T = o(d)$, we find that the first term on the right-hand side of \eqref{T_estimate_intro} is bounded by $N^{-\nu}$ provided that $o(1) d^2 \geq \nu \log N$, which corresponds precisely to the optimal lower bound in \eqref{d_condition_deloc}. Using \eqref{T_estimate_intro}, we may deduce \ref{item:pf2_a} and \ref{item:pf2_b}.

To prove \eqref{T_estimate_intro}, we need to decouple the events $\{x \in \cal T\}$, $x \in T$. We do so by replacing the Green functions $G^{(x)}$ in the definition of $\cal T$ by $G^{(T)}$, after which the corresponding events are essentially independent. The error that we incur depends on the difference $G^{(T)}_{yy} - G_{yy}$, which we have to show is small with very high probability under the bootstrapping assumption that the entries of $G$ are bounded. For $T$ of fixed size, this follows easily from standard resolvent identities. However, for our purposes it is crucial that $T$ can have size up to $o(d)$, which requires a more careful quantitative analysis. As it turns out, $G^{(T)}_{yy} - G_{yy}$ is small only up to $\abs{T} = o(d)$, which is precisely what we need to reach the optimal scale $d \gg \sqrt{\log N}$ from \eqref{optimal_regime}. 

\section{The semilocalized phase}\label{sec:localization}

In this section we prove the results of Section \ref{sec:localized_results} -- Theorems \ref{thm:localisation} and \ref{thm:one_to_one}.

\subsection{The pruned graph and proof of Theorem \ref{thm:localisation}}

The balls $(B_r(x))_{x \in \cal W_{\lambda, \delta}}$ in Theorem \ref{thm:localisation} are in general not disjoint. For its proof, and in order to give a precise definition of the vector $\f v(x)$ in Theorem \ref{thm:localisation}, we need to make these balls disjoint by \emph{pruning} the graph $\bb G$. This is an important ingredient of the proof, and will also allow us to state a more precise version of Theorem \ref{thm:localisation}, which is Theorem \ref{thm:general} below. This pruning was previously introduced in \cite{ADK19}; it is performed by cutting edges from $\bb G$ in such a way that the balls $(B_r(x))_{x \in \cal W_{\lambda, \delta}}$ are disjoint for appropriate radii, $r = 2 r_\star$, by carefully cutting in the right places, thus reducing the number of cut edges. This ensures that the pruned graph is close to the original graph in an appropriate sense. The pruned graph, $\bb G_\tau$, depends on a parameter $\tau > 1$, and its construction is the subject of the following proposition.

To state it, we introduce the following notations. For a subgraph $\bb G_\tau$ of $\bb G$ we abbreviate
\begin{equation*}
B^\tau_i(x) \deq B^{\bb G_\tau}_i(x)\,, \qquad S^\tau_i(x) \deq S^{\bb G_\tau}_i(x)\,.
\end{equation*}
Moreover, we define the set of vertices with large degrees
\begin{equation*}
\cal V_\tau \deq \h{x \in [N] \col \alpha_x \geq \tau}\,.
\end{equation*}

\begin{proposition}[Existence of pruned graph]
\label{prop:subgraph_separating_large_degrees} 
Let $1 + \xi^{1/2} \leq \tau \leq 2$ and $d \leq 3 \log N$. There exists a subgraph $\bb G_\tau$ of $\bb G$ with the following properties. 
\begin{enumerate}[label=(\roman*)]
\item \label{item:subgraph_paths} Any path in $\bb G_\tau$ connecting two different vertices in $\cal V_\tau$ has length at least $4 r_{\star} +1$. In particular, the balls $(B_{2 r_{\star}}^{\tau}(x))_{x \in \cal V_\tau}$ are disjoint.
\item \label{item:subgraph_tree} The induced subgraph $\bb G_\tau|_{B_{2 r_{\star}}^{\tau}(x)}$ is a tree for each $x \in \cal V_\tau$. 
\item \label{item:subgraph_cut_only_in_S_1} For each edge in $\bb G\setminus \bb G_\tau$, there is at least one vertex in $\cal{V}_\tau$ incident to it. 
\item \label{item:subgraph_inclusion_S_i} For each $x \in \cal V_\tau$ and each $i \in \N$ satisfying $1 \leq i \leq 2 r_{\star}$ we have $S_i^{\tau}(x) \subset S_i(x)$. 
\item \label{item:subgraph_degrees} The degrees induced on $[N]$ by $\bb G\setminus \bb G_\tau$ are bounded according to
\begin{equation}\label{eq:subgraph_degrees}
\max_{x \in [N]} D_x^{\bb G \setminus \bb G_\tau} \leq  \cal C \frac{\log N}{(\tau-1)^2d}
\end{equation}
with very high probability. 
\item \label{item:subgraph_S_i} 
Suppose that $\sqrt{\log N} \leq d$. 
For each $x \in \cal V_\tau$ and all $2 \leq i \leq 2 r_\star$, the bound 
\begin{equation}
|S_{i}(x)\setminus S_{i}^{\tau}(x)|\leq \mathcal{C}\frac{\log N}{(\tau-1)^2}d^{i-2}\label{eq:sum_degree}
\end{equation}
holds with very high probability.  
\end{enumerate} 
\end{proposition} 

The proof of Proposition \ref{prop:subgraph_separating_large_degrees} is postponed to the end of this section, in Section \ref{sec:pf_pruning} below. It is essentially \cite[Lemma 7.2]{ADK19}, the main difference being that \ref{item:subgraph_S_i} is considerably sharper than its counterpart, \cite[Lemma 7.2 (vii)]{ADK19}; this stronger bound is essential to cover the full optimal regime \eqref{optimal_regime} (see Section \ref{sec:forks_intro}). As a guide for the reader's intuition, we recall the main idea of the pruning. First, for every $x \in \cal V_\tau$, we make the $2 r_\star$-neighbourhood of $x$ a tree by removing appropriate edges incident to $x$. Second, we take all paths of length less than $4 r_\star + 1$ connecting different vertices in $\cal V_\tau$, and remove all of their edges incident to any vertex in $\cal V_\tau$. Note that only edges incident to vertices in $\cal V_\tau$ are removed. This informal description already explains properties \ref{item:subgraph_paths}--\ref{item:subgraph_inclusion_S_i}. Properties \ref{item:subgraph_degrees} and \ref{item:subgraph_S_i} are probabilistic in nature, and express that with very high probability the pruning has a small impact on the graph. See also Lemma \ref{lem:estimate_cut_graph} below for a statement in terms of operator norms of the adjacency matrices. For the detailed algorithm, we refer to the proof of \cite[Lemma 7.2]{ADK19}.

Using the pruned graph $\bb G_\tau$, we can give a more precise formulation of Theorem \ref{thm:localisation}, where the localization profile vector $\f v(x)$ from Theorem \ref{thm:localisation} is explicit. 
For its statement, we introduce the set of vertices
\begin{equation} \label{def_calV}
\cal V \deq \cal V_{2 + \xi^{1/4}}
\end{equation}
around which a localization profile can be defined.

\begin{definition}[Localization profile] \label{def:bv}
Let $1 + \xi^{1/2} \leq \tau \leq 2$ and $\bb G_\tau$ be the pruned graph from Proposition \ref{prop:subgraph_separating_large_degrees}.
For $x \in \cal V$ we introduce positive weights $u_0(x), u_1(x), \dots, u_{r_\star}(x)$ as follows.
Set $u_0(x) > 0$ and define, for $i = 1, \dots, r_\star - 1$,
\begin{equation} \label{ui_definition}
u_i(x) \deq \frac{\sqrt{\alpha_x}}{(\alpha_x - 1)^{i/2}} \, u_0(x)\,, \qquad
u_{r_\star}(x) \deq \frac{1}{(\alpha_x - 1)^{(r_\star - 1)/2}} \, u_0(x)\,.
\end{equation}
For $\sigma = \pm$ we define the radial vector
\begin{equation} \label{eq:def_vtau}
\f v^\tau_\sigma(x) \deq \sum_{i = 0}^{r_\star} \sigma^i u_i(x) \frac{\f 1_{S_i^\tau(x)}}{\norm{\f 1_{S_i^\tau(x)}}}\,, \qquad 
\end{equation}
and choose $u_0(x) > 0$ such that $\f v^\tau_\sigma(x)$ is normalized.
\end{definition}

\begin{remark} \label{rem:orthog}
The family $(\f v_\sigma^\tau(x) \col x \in \cal V, \,\sigma = \pm)$ is orthonormal. Indeed, if $x,y \in \cal V$ are distinct, then by Proposition \ref{prop:subgraph_separating_large_degrees} \ref{item:subgraph_paths} the vectors $\f v^\tau_{\sigma}(x)$ and $\f v^\tau_{\tilde \sigma}(y)$ are orthogonal for any $\sigma, \tilde \sigma = \pm$ because they are supported on disjoint sets of vertices. Moreover, $\f v^\tau_+(x)$ and $\f v^\tau_-(x)$ are orthogonal by the choice of $u_{r_\star}(x)$ from \eqref{ui_definition}, as can be seen by a simple computation.
\end{remark}

The following result restates Theorem \ref{thm:localisation} by identifying $\f v(x)$ there as $\f v_+^\tau(x)$ given in \eqref{eq:def_vtau}. It easily implies Theorem \ref{thm:localisation}, and the rest of this section is devoted to its proof.

\begin{theorem} \label{thm:general}
The following holds with very high probability. Suppose that $d$ satisfies \eqref{d_assumption_localization}.
Let $\f w$ be a normalized eigenvector of $A/\sqrt{d}$ with nontrivial eigenvalue $\lambda \geq 2+ \cal C \xi^{1/2}$. Choose $0<\delta\leq (\lambda-2)/2$ and set $\tau \deq 1 + (\lambda-2)/8\wedge 1$. Then
\begin{equation} \label{w_v_leq}
\sum_{x \in \cal W_{\lambda,\delta}} \scalar{\f v^\tau_+(x)}{\f w}^2 \geq  1 -\cal C  \left(\frac{ \xi+\xi_{\tau-1}}{\delta}\right)^2\,.
\end{equation}
\end{theorem}

\begin{remark} \label{rem:general_negative}
An analogous result holds for negative eigenvalues $-\lambda$, where $\lambda$ is as in Theorem \ref{thm:general} and $\f v_+^\tau(x)$ in \eqref{w_v_leq} is replaced with $\f v_-^\tau(x)$.
\end{remark}

For the motivation behind Definition \ref{def:bv}, we refer to the discussion in Section \ref{sec:sketch_localized} and Appendix \ref{sec:mu}. As explained there, if $\bb G_\tau$ is sufficiently close to the infinite tree $\bb T_{D_x, d}$ in a ball of radius $r_\star$ around $x$, and if $r_\star$ is large enough for $u_{r_\star}(x)$ to be very small, we expect \eqref{eq:def_vtau} to be an approximate eigenvector of $A$. This will in fact turn out to be true; see Proposition \ref{prop:proof_lower_bound_approximate_eigenvector} below. That $r_\star$ is in fact large enough is easy to see: the definition of $r_\star$ in \eqref{def_r_star} and the bound $\xi \geq 1/d$ imply that, for $\alpha_x\geq 2+ C (\log d)^2 / \sqrt{\log N}$, we have 
\begin{equation} \label{eq:smallness_alpha_x_minus_r_star} 
(\alpha_x - 1)^{-(r_\star -2)/2} \leq \xi\,. 
\end{equation}
This means that the last element of the sequence $(u_i(x))_{i=0}^{r_\star}$ is bounded by $\xi$. Note that the lower bound on $\alpha_x$ imposed above always holds for $x \in \cal V$, since, by \eqref{d_assumption_localization},
\begin{equation}\label{eq:lower_bound_xi_1_4} 
\frac{C (\log d)^2}{\sqrt{\log N}} \leq \xi^{1/4}\,.
\end{equation}

\begin{figure}[!ht]
\begin{center}
{\small 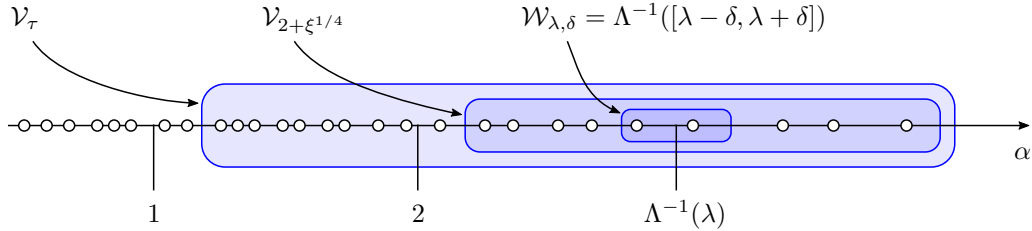}
\end{center}
\caption{An illustration of the three sets of vertices of increasing size that enter into the proof of Theorem \ref{thm:general}. Each vertex $x$ is plotted as a dot at its normalized degree $\alpha_x$. The largest set is $\cal V_{\tau}$ from Proposition \ref{prop:subgraph_separating_large_degrees}, where $1 + \xi^{1/2} \leq \tau \leq 2$. It is used to define the pruned graph $\bb G_\tau$. The intermediate set is $\cal V \equiv \cal V_{2 + \xi^{1/4}}$ from \eqref{def_calV}. It is the set of vertices for which we can define the localization profile vector $\f v(x)$ that decays exponentially around $x$. The smallest set $\cal W_{\lambda,\delta} = \Lambda^{-1}([\lambda - \delta, \lambda + \delta])$ is the set of resonant vertices at energy $\lambda$.
 \label{fig:vertices}}
\end{figure}

As a guide to the reader, in Figure \ref{fig:vertices}, we summarize the three main sets of vertices that are used in the proof of Theorem \ref{thm:general}.
We conclude this subsection by proving Theorem \ref{thm:localisation} and Corollary~\ref{cor:loc_centres} using Theorem \ref{thm:general}.

\begin{proof}[Proof of Theorem \ref{thm:localisation}]
The first claim follows immediately from Theorem \ref{thm:general}, with $\f v(x) = \f v^\tau_+(x)$.  To verify the claim about the exponential decay of $\f v$, we note that the graph distance in $\bb G$ is bounded by the graph distance in $\bb G_\tau$, which implies
\begin{equation*}
\sum_{y \in B_r(x)^c} (\f v^\tau_+(x))_y^2 \leq \sum_{y \in B_r^\tau(x)^c} (\f v^\tau_+(x))_y^2  = \sum_{i = r+1}^{r_\star} u_i(x)^2\,,
\end{equation*}
from which the claim easily follows using the definition \eqref{ui_definition}.
\end{proof}

\begin{proof}[Proof of Corollary \ref{cor:loc_centres}]
We decompose $\f w = \sum_{x \in \cal W_{\lambda,\delta}} \gamma_x \f v^\tau_+(x) + \f e$, where $\gamma_x \deq \scalar{\f v_+^\tau(x)}{\f w}$ and $\f e$ is orthogonal to $\Span \h{\f v_+^\tau(x) \col x \in \cal W_{\lambda, \delta}}$.
By Theorem \ref{thm:general} we have $\norm{\f e} \leq \frac{\cal C (\xi+\xi_{\tau-1})}{\delta}$ and
\begin{equation} \label{gamma_sum}
\sum_{x \in \cal W_{\lambda,\delta}} \gamma_x^2\geq 1-\frac{\cal C (\xi+\xi_{\tau-1})}{\delta}\,.
\end{equation}
Moreover, since $\lambda - \delta \geq 2 \geq \tau$, we have $\cal W_{\lambda,\delta} \subset \cal V_\tau$, so that Proposition \ref{prop:subgraph_separating_large_degrees} \ref{item:subgraph_paths} implies $(\f v^\tau_+(x))_y = \delta_{xy} u_0(x)$ for $x,y \in \cal W_{\lambda, \delta}$. Thus we have
\begin{equation} \label{eq:form_sum_center}
\sum_{y \in \cal W_{\lambda,\delta}} w_y^2 = \|\f w|_{\cal W_{\lambda,\delta}}\|^2 = \normBB{\sum_{x \in \cal W_{\lambda,\delta}} \gamma_x \f v^\tau_+(x) \vert_{\cal W_{\lambda,\delta}}}^2 + O(\norm{\f e})
= \sum_{y \in \cal W_{\lambda,\delta}} \gamma_y^2 u_0(y)^2+\mathcal{O}\left(\frac{\xi+\xi_{\tau-1}}{\delta}\right)\,.
\end{equation}
Since $u_0(y)$ was chosen such that $\f v_+^\tau(y) $ is normalized, we find
\begin{equation*}
u_0(y)^2=\left( 1 + \sum_{i=1}^{r_\star-1} \frac{\alpha_y}{(\alpha_y-1)^i} +\frac{1}{(\alpha_y-1)^{r_\star - 1}}\right)^{-1} =  \frac{\alpha_y-2}{2(\alpha_y-1)} + O \pbb{\frac{1}{(\alpha_y - 1)^{r_\star - 1}}}\,.
\end{equation*}
Define $\alpha \deq \Lambda^{-1}(\lambda)$ for $\alpha \geq 2$. 
Since $|\Lambda(\alpha_y)-\lambda|\leq \delta$ for $y\in \cal W_{\lambda, \delta}$, we obtain
\begin{equation*}
|\alpha_y - \alpha| \leq \delta \max_{t\in [\lambda-\delta,\lambda+\delta]}( \Lambda^{-1})'(t) = O\left(\delta \lambda^{3/2} (\lambda-2)^{-1/2}\right)\,,
\end{equation*}
where we used that $\lambda \pm \delta - 2 \asymp \lambda - 2$. Since $\frac{\dd}{\dd \alpha} \frac{\alpha - 2}{2 (\alpha - 1)} = \frac{1}{2(\alpha - 1)^2} \asymp \lambda^{-4}$, we find
\begin{equation} \label{u_0_estimate}
u_0(y)^2 = \frac{\alpha - 2}{2 (\alpha - 1)} +  O\left(\frac{\delta}{\lambda^{5/2} \sqrt{\lambda-2}} + \frac{1}{(\alpha_y - 1)^{r_\star - 1}}\right)= \frac{\alpha - 2}{2 (\alpha - 1)} + O \left( \frac{\delta}{\lambda^{5/2} \sqrt{\lambda-2}} + \frac{\xi}{\delta} 
\right) 
 \,,
\end{equation}
where we used \eqref{eq:smallness_alpha_x_minus_r_star} and the upper bound on $\delta$ in the last step.  
By an elementary computation,
\begin{equation*}
\frac{\alpha - 2}{2 (\alpha - 1)} = \frac{\sqrt{\lambda^2 - 4}}{\lambda + \sqrt{\lambda^2 - 4}}\,,
\end{equation*}
and the claim hence follows by recalling \eqref{eq:smallness_alpha_x_minus_r_star} and plugging \eqref{gamma_sum} and \eqref{u_0_estimate} into \eqref{eq:form_sum_center}.
\end{proof}

\subsection{Block diagonal approximation of pruned graph and proof of Theorems \ref{thm:general} and \ref{thm:one_to_one}}
We now introduce the adjacency matrix of $\bb G_\tau$ and a suitably defined centred version. Then we define a block diagonal approximation of this matrix, called $\wh H^\tau$ in \eqref{def_block_diagonal} below, which is the central construction of our proof.

\begin{definition} \label{def:underline_A_underline_A_tau_chi_tau} 
Let $A^\tau$ be the adjacency matrix of $\bb G_\tau$.
Let $H \deq \ul A / \sqrt{d}$ and $H^\tau \deq \ul A^\tau / \sqrt{d}$, where
\begin{equation}\label{eq:defAoperator}
\ul A \deq A - \E A \,, \qquad \underline{A}^\tau \deq A^\tau - \chi^\tau(\mathbb{E}A)\chi^\tau
\end{equation}
and $\chi^\tau$ is the orthogonal projection onto $\Span \{ \f 1_y \col y \notin \bigcup_{x \in \cal V_\tau} B_{2 r_\star}^\tau(x)\}$.
\end{definition}

The definition of $\ul A^\tau$ is chosen so that (i) $\ul A^\tau$ is close to $\ul A$ provided that $A^\tau$ is close to $A$, since the kernel of $\chi^\tau$ has a relatively low dimension, and (ii) when restricted to vertices at distance at most $2 r_\star$ from $\cal V_\tau$, the matrix $\ul A^\tau$ coincides with $A^\tau$. In fact, property (i) is made precise by the simple estimate
\begin{equation} \label{EA_estimate}
\norm{\E A - \chi^\tau (\E A) \chi^\tau} \leq 2
\end{equation}
with very high probability (see \cite[Eq.~(8.17)]{ADK19} for details). 
Property (ii) means that $\ul A^\tau$ inherits the locality of the matrix $A$, meaning that applying $\ul A^\tau$ to a vector localized in space to a small enough neighbourhood of $\cal V_\tau$ yields again a vector localized in space. This property will play a crucial role in the proof, and it can be formalized as follows.

\begin{remark} \label{rem:v_tau_supp}
Let $i + j \leq 2 r_\star$. Then for any $x \in \cal V_\tau$ and vector $\f v$ we have
\begin{equation*}
\supp \f v \subset B_i^\tau(x) \quad \Longrightarrow \quad \supp \qb{(H^\tau)^j \f v} \subset B_{i+j}^\tau(x)\,.
\end{equation*}
\end{remark}

The next result states that $H^\tau$ is a small perturbation of $H$.
\begin{lemma} \label{lem:estimate_cut_graph}
Suppose that $d \leq 3 \log N$. For any $1 + \xi^{1/2} \leq \tau \leq 2$ we have $\norm{H - H^\tau} \leq \cal C \xi_{\tau-1}$ with very high probability.
\end{lemma}

The next result states that $\f v_\sigma^\tau(x)$ is an approximate eigenvector of $H^\tau$.
\begin{proposition}\label{prop:proof_lower_bound_approximate_eigenvector}
Let $d$ satisfy \eqref{d_assumption_localization}.
Let $x \in [N]$ and suppose that $1 + \xi^{1/2} \leq \tau \leq 2$.
If $\alpha_x\geq 2+ C (\log d)^2 / \sqrt{\log N}$ then for $\sigma = \pm$ we have
\begin{equation}\label{eq:proof_lower_bound_approximate_eigenvector} 
\|(H^\tau - \sigma \Lambda(\alpha_x)) \f v^{\tau}_\sigma(x) \|
 \leq \mathcal{C} \xi
\end{equation}
with very high probability.
\end{proposition}

The proofs of Lemma \ref{lem:estimate_cut_graph} and Proposition \ref{prop:proof_lower_bound_approximate_eigenvector} are deferred to Section \ref{Subsec:TechnicalProofA}.
The following object is the central construction in our proof.
\begin{definition}[Block diagonal approximation of pruned graph]
Define the orthogonal projections
\begin{equation} \label{def_Pi}
\Pi^\tau \deq \sum_{x \in \cal V} \sum_{\sigma = \pm} \f v^\tau_\sigma(x) \f v^\tau_\sigma(x)^*\,, \qquad \ol \Pi^\tau  \deq I - \Pi^\tau\,,
\end{equation}
and the matrix
\begin{equation} \label{def_block_diagonal}
\wh H^\tau \deq \sum_{x \in \cal V} \sum_{\sigma = \pm} \sigma \Lambda(\alpha_x) \f v^\tau_\sigma(x) \f v^\tau_\sigma(x)^* + \ol \Pi^\tau H^\tau \ol \Pi^\tau\,.
\end{equation}
\end{definition}
That $\Pi^\tau$ and $\ol \Pi^\tau$ are indeed orthogonal projections follows from Remark \ref{rem:orthog}.
Note that $\wh H^\tau$ may be interpreted as a \emph{block diagonal approximation of $H^\tau$}. Indeed, completing the orthonormal family $(\f v^\tau_\sigma(x))_{x \in \cal V, \sigma = \pm}$ to an orthonormal basis of $\R^N$, which we write as the columns of the orthogonal matrix $R$, we have
\begin{equation*}
R^* \wh H^\tau R =
\begin{bmatrix}
\diag(\sigma \Lambda(\alpha_x))_{x \in \cal V, \sigma = \pm}  & 0
\\
0 & [*]
\end{bmatrix}\,.
\end{equation*}

The following estimate states that $\wh H^\tau$ is a small perturbation of $H^\tau$.
\begin{lemma}\label{lem:estime_block_matrix}
Let $d$ satisfy \eqref{d_assumption_localization}.
If $1 + \xi^{1/2} \leq \tau \leq 2$ then
$\norm{H^\tau - \wh H^\tau} \leq \mathcal{C}\xi $ with very high probability.
\end{lemma}

The proof of Lemma \ref{lem:estime_block_matrix} is deferred to Section \ref{Subsec:TechnicalProofA}.
The following result is the key estimate of our proof; it states that on the range of $\ol \Pi^\tau$ the matrix $H^\tau$ is bounded by $2\tau + o(1)$.

\begin{proposition}\label{pro:one-one_exact_correspondance}
Let $d$ satisfy \eqref{d_assumption_localization}.
If $1 + \xi^{1/2} \leq \tau \leq 2$ then $\norm{\ol \Pi^\tau H^\tau \ol \Pi^\tau} \leq  2\tau  + \cal C (\xi+\xi_{\tau-1})$ with very high probability.
\end{proposition}

The proof of Proposition \ref{pro:one-one_exact_correspondance} is deferred to Section \ref{sec:estimatesB}.
We now use Lemma \ref{lem:estime_block_matrix} and Proposition \ref{pro:one-one_exact_correspondance} to conclude Theorems \ref{thm:general} and \ref{thm:one_to_one}.

\begin{proof}[Proof of Theorem \ref{thm:general}]
Define the orthogonal projections 
\begin{equation*}
\Pi^\tau_{\lambda,\delta} \deq \sum_{x \in \cal W_{\lambda,\delta}} \f v^\tau_+(x) \, \f v^\tau_+(x)^* \,, 
\qquad \qquad \ol \Pi^\tau_{\lambda,\delta} \deq I - \Pi^\tau_{\lambda,\delta}
\,.
\end{equation*}
By definition, the orthogonal projections $\Pi^\tau$ and $\Pi^\tau_{\lambda,\delta}$ commute. Moreover, under the assumptions of Theorem \ref{thm:general} we have the inclusion property
\begin{equation} \label{incl_Pi}
\Pi^\tau \Pi^\tau_{\lambda,\delta} = \Pi^\tau_{\lambda,\delta}\,.
\end{equation}
See also Figure \ref{fig:vertices}.
To show \eqref{incl_Pi}, we note that the condition on $\delta$ and the lower bound on $\lambda$ in Theorem \ref{thm:general} imply $\lambda - \delta \geq 2 + \cal C \xi^{1/2}$. Using $\Lambda(2 + x) - 2 \asymp x^2 \wedge x^{1/2}$ for $x \geq 0$ we conclude that for any $\alpha \geq 2$ we have the implication $\Lambda(\alpha) \geq \lambda - \delta \; \Rightarrow \; \alpha \geq 2 + \xi^{1/4}$, which implies \eqref{incl_Pi}.

Next, we abbreviate $E^\tau \deq \chi^\tau (\E A / \sqrt{d}) \chi^\tau$ and note that $\Pi^\tau E^\tau = 0$ because $\Pi^\tau \chi^\tau = 0$ by construction of $\f v_\sigma^\tau(x)$. From \eqref{incl_Pi} we obtain $\ol \Pi^\tau_{\lambda,\delta} = \ol \Pi^\tau_{\lambda,\delta} \Pi^\tau + \ol \Pi^\tau$, which yields
\begin{equation} \label{H_tau_decomp}
\ol \Pi^\tau_{\lambda,\delta} (\wh H^\tau + E^\tau) \ol \Pi^\tau_{\lambda,\delta} = \ol \Pi^\tau_{\lambda,\delta} \Pi^\tau \wh H^\tau \Pi^\tau \ol \Pi^\tau_{\lambda,\delta} + \pb{\ol \Pi^\tau \wh H^\tau \ol \Pi^\tau + E^\tau}\,,
\end{equation}
where we used that the cross terms vanish because of the block diagonal structure of $\wh H^\tau$.

The core of our proof is the \emph{spectral gap}
\begin{equation} \label{spectral_gap}
\spec \pB{\ol \Pi^\tau_{\lambda,\delta} (\wh H^\tau + E^\tau) \ol \Pi^\tau_{\lambda,\delta}} \subset \R \setminus [\lambda - \delta, \lambda + \delta]\,. 
\end{equation}
To establish \eqref{spectral_gap}, it suffices to establish the same spectral gap for each term on the right-hand side of \eqref{H_tau_decomp} separately, since the right-hand side of \eqref{H_tau_decomp} is a block decomposition of its left-hand side. 
The first term on the right-hand side of \eqref{H_tau_decomp} is explicit:
\begin{equation*}
\ol \Pi^\tau_{\lambda,\delta} \Pi^\tau \wh H^\tau \Pi^\tau \ol \Pi^\tau_{\lambda,\delta} = \sum_{x \in \cal V} \sum_{\sigma = \pm} \sigma \Lambda(\alpha_x) \, \ind{\abs{\sigma \Lambda(\alpha_x) - \lambda} > \delta} \, \f v^\tau_\sigma(x) \f v^\tau_\sigma(x)^*\,,
\end{equation*}
which trivially has no eigenvalues in $[\lambda - \delta, \lambda + \delta]$.

In order to establish the spectral gap for the second term of \eqref{H_tau_decomp}, we begin by remarking that $E^\tau$ has rank one and, by \eqref{EA_estimate}, its unique nonzero eigenvalue is $\sqrt{d} + O(1/\sqrt{d})$.
Hence, by rank-one interlacing and Proposition \ref{pro:one-one_exact_correspondance}, we find
\begin{equation} \label{eq:spec_wh_H_tau2} 
\spec \pb{\ol \Pi^\tau (H^\tau + E^\tau) \ol \Pi^\tau} \subset
\qb{-2\tau -\cal C (\xi+\xi_{\tau-1}) \,, 2\tau +\cal C (\xi+\xi_{\tau-1})} \cup \hb{\mu}
\end{equation}
for some simple eigenvalue $\mu = \sqrt{d} + O(1)$. Thus, to conclude the proof of the spectral gap for the second term of \eqref{H_tau_decomp}, it suffices to show that
\begin{align} \label{lambda_delta_1}
\lambda - \delta &> 2\tau +\cal C (\xi+\xi_{\tau-1})
\\ \label{lambda_delta_2}
\lambda + \delta &< \mu\,.
\end{align}
To prove \eqref{lambda_delta_1}, we suppose that $\lambda \geq 2 + 8 \cal C \xi^{1/2}$ and, recalling the condition on $\delta$ and the choice of $\tau$ in Theorem \ref{thm:general}, obtain
\begin{equation} \label{xi_zeta_est}
\lambda - \delta \geq 2 + \frac{\lambda - 2}{2} \geq 2\tau + 2 \cal C \xi^{1/2} > 2\tau + \cal C(\xi + \xi_{\tau-1})\,,
\end{equation}
where in the last step we used that $\xi_{\tau-1} < \xi^{1/2}$ by our choice of $\tau$ and the lower bound on $\lambda$. This is \eqref{lambda_delta_1}.

For the following arguments, we compare $A / \sqrt{d}$ with $\wh H^\tau + E^\tau$ using 
the estimate 
\begin{equation} \label{tel_estimate}
\norm{A / \sqrt{d} - (\wh H^\tau + E^\tau)}  \leq 
\|(H^\tau - \wh H^\tau) + (H - H^\tau) + (\E A / \sqrt{d} - E^\tau)\|\leq \cal C (\xi+\xi_{\tau-1})
\end{equation}
with very high probability, which follows from 
Lemma \ref{lem:estimate_cut_graph}, Lemma  \ref{lem:estime_block_matrix}, \eqref{EA_estimate} 
and $d^{-1/2} \leq \cal C \xi$.

Next, we use \eqref{tel_estimate} to conclude the proof of \eqref{lambda_delta_2}. The only nonzero eigenvalue of $E^\tau$ is $\sqrt{d}(1 + O(1/d))$, and from Proposition~\ref{pro:one-one_exact_correspondance} and Remark \ref{rem:xi} we have $\norm{\wh{H}^\tau} \leq \Lambda(\max_{x \in \cal V} \alpha_x) + O(1)$ with very high probability, so that Lemma \ref{lem:upper_bound_degrees} and the assumption \eqref{d_assumption_localization} yield $\norm{\wh{H}^\tau} \leq \cal C \sqrt{\frac{\log N}{d}}$ with very high probability. 
Hence, by first order perturbation theory (e.g.\ Weyl's inequality), \eqref{d_assumption_localization} and \eqref{tel_estimate} imply that $A/\sqrt{d}$ has one eigenvalue bigger than $\sqrt{d} - O(1)$ and all other 
eigenvalues are at most $\cal C \sqrt{\frac{\log N}{d}}$. 
Since $\lambda$ is nontrivial, we conclude that 
$\lambda \leq \cal C \sqrt{\frac{\log N}{d}}$. 
By the upper bound $\delta \leq (\lambda - 2)/2$ and the lower bound on $d$ in \eqref{d_assumption_localization}, this concludes the proof of \eqref{lambda_delta_2} and, thus, the one of the spectral gap \eqref{spectral_gap}.

Next, from \eqref{spectral_gap}, and \eqref{tel_estimate}, we conclude the \emph{spectral gap for the full adjacency matrix}
\begin{equation} \label{spectral_gap_A}
\spec \pB{\ol \Pi^\tau_{\lambda,\delta} (A / \sqrt{d}) \ol \Pi^\tau_{\lambda,\delta}} \subset \R \setminus \qb{\lambda-\delta+\cal C (\xi+\xi_{\tau-1}),\lambda+\delta-\cal C (\xi+\xi_{\tau-1})}\,. 
\end{equation}
Using \eqref{spectral_gap_A} we may conclude the proof.
The eigenvalue-eigenvector equation $(A/ \sqrt{d} - \lambda) \f w = 0$ yields
\begin{equation} \label{w_identity}
\ol \Pi_{\lambda,\delta}^\tau \f w = - \pB{\ol \Pi_{\lambda,\delta}^\tau (A/\sqrt{d}) \ol \Pi_{\lambda,\delta}^\tau -\lambda}^{-1} \ol \Pi_{\lambda,\delta}^\tau  (A/\sqrt{d})  \Pi_{\lambda,\delta}^\tau \f w\,.
\end{equation}
Assuming that $\delta > \cal C (\xi + \xi_{\tau - 1})$, from \eqref{spectral_gap_A} we get
\begin{equation} \label{res_bound}
\normB{\pB{\ol \Pi_{\lambda,\delta}^\tau (A/\sqrt{d}) \ol \Pi_{\lambda,\delta}^\tau -\lambda}^{-1}}  \leq \frac{1}{\delta-\cal C (\xi+\xi_{\tau-1})}\,.
\end{equation}
Moreover, since $\ol \Pi_{\lambda,\delta}^\tau \wh H^\tau \Pi_{\lambda,\delta}^\tau = 0$ and $E^\tau\Pi_{\lambda,\delta}^\tau=0$, we deduce from \eqref{tel_estimate} that
\begin{equation}\label{offdiag_bound}
\|\ol \Pi_{\lambda,\delta}^\tau (A/\sqrt{d}) \Pi_{\lambda,\delta}^\tau\| \leq \cal C (\xi + \xi_{\tau-1})\,.
\end{equation}
Plugging \eqref{res_bound} and \eqref{offdiag_bound} into \eqref{w_identity} yields
\[
\|\ol \Pi_{\lambda,\delta}^\tau \f w \| \leq \frac{\cal C (\xi+\xi_{\tau-1})}{\delta - \cal C(\xi + \xi_{\tau-1})}\wedge 1 \leq \frac{2\cal C (\xi+\xi_{\tau-1})}{\delta}\,,
\] 
since $\f w$ is normalized. This concludes the proof if $\delta > \cal C (\xi + \xi_{\tau - 1})$ (after a renaming of the constant $\cal C$), and otherwise the claim is trivial.
\end{proof}

Proposition \ref{pro:one-one_exact_correspondance} is also the main tool to prove Theorem \ref{thm:one_to_one}.
\begin{proof}[Proof of Theorem \ref{thm:one_to_one}]
The proof uses Proposition~\ref{pro:one-one_exact_correspondance}, Lemma~\ref{lem:estimate_cut_graph}, and Lemma~\ref{lem:estime_block_matrix} 
for $\tau \in [1 + \xi^{1/2}/3,2]$. Note that the lower bound $1 + \xi^{1/2}/3$ is smaller than the lower bound $1 + \xi^{1/2}$ imposed in these results, but their proofs hold verbatim also in this regime of $\tau$. 

We set $E^\tau \deq \chi^\tau (\E A/\sqrt{d})\chi^\tau$ with $\chi^\tau$ from Definition~\ref{def:underline_A_underline_A_tau_chi_tau}. 
We now compare $A/\sqrt{d}$ and $\wh{H}^\tau + E^\tau$, as in the proof of Theorem~\ref{thm:general}, 
and use some estimates from its proof. 
For any $ \tau \in [1 + \xi^{1/2}/3,2]$, we have 
\begin{equation} \label{eq:spec_wh_H_tau1} 
\spec(\wh H^\tau + E^\tau ) = \{\pm \Lambda(\alpha_x)\col x \in \cal U \} \cup \spec\big(\ol \Pi^\tau ( H^\tau + E^\tau) \ol \Pi^\tau \big)\,,
\end{equation}
since $\Pi^\tau \chi^\tau =0$.  
By first order perturbation theory and the choice $\tau =2$, we get from \eqref{eq:spec_wh_H_tau1}, \eqref{eq:spec_wh_H_tau2} and \eqref{tel_estimate} that 
$\lambda_1(A/\sqrt{d}) = \mu + O(\xi) = \sqrt{d} + O(1)$ and $\lambda_1(A/\sqrt{d})$ 
is well separated from the other eigenvalues of $A/\sqrt{d}$ (see the proof of Theorem~\ref{thm:general}).
Combining \eqref{eq:spec_wh_H_tau1}, \eqref{eq:spec_wh_H_tau2}, and \eqref{tel_estimate},  
choosing $\tau = 1 + \xi^{1/2}/3$ as well as using $\cal C(\xi + \xi_{\tau - 1}) \leq \xi^{1/2}/3$ for this choice of $\tau$ imply \eqref{lambda_bulk_estimate}. 

Moreover, we apply first order perturbation theory to \eqref{eq:spec_wh_H_tau1} 
using \eqref{eq:spec_wh_H_tau2} and \eqref{tel_estimate}, and obtain  
\begin{equation}\label{eq:eigen_error_bound}
|\lambda_{i + 1}(A/\sqrt{d}) - \Lambda(\alpha_{\sigma(i)})|  + |\lambda_{N-i+1}(A/\sqrt{d}) + \Lambda(\alpha_{\sigma(i)})| \leq \mathcal{C}(\xi + \xi_{\tau - 1}) 
\end{equation}
with very high probability for all $\tau \in [1+\xi^{1/2}/3,2]$ and all $i \in \q{\abs{\cal{U}}}$ 
satisfying 
\begin{equation} \label{eq:condition_eigenvalue_estimates} 
2 (\tau -1) + \cal C ( \xi + \xi_{\tau - 1}) < \Lambda(\alpha_{\sigma(i)})-2. 
\end{equation} 

What remains is choosing $\tau \equiv \tau_i$, depending on $i \in [\abs{\cal U}]$, such that the condition \eqref{eq:condition_eigenvalue_estimates} is satisfied and the error estimate from \eqref{eq:eigen_error_bound} transforms into the form of \eqref{lambda_estimate}. Both are achieved by setting 
\begin{equation} \label{eq:choice_tau} 
\tau = 1 + \frac{1}{3}\big[ (\Lambda(\alpha_{\sigma(i)}) -2 ) \wedge 3 \big]. 
\end{equation} 
Note that $\tau \in [1 + \xi^{1/2}/3,2]$ as $\sigma(i) \in \cal U$. 
From $\Lambda(\alpha_{\sigma(i)})-2 \geq 3(\tau -1) $ due to \eqref{eq:choice_tau} and $\Lambda(\alpha_{\sigma(i)})-2 \geq \xi^{1/2}$ by the definition of $\cal U$, we conclude that 
\[  \Lambda(\alpha_{\sigma(i)})-2 \geq \frac{5}{2} (\tau - 1) + \frac{1}{6} \xi^{1/2} \geq 2 (\tau - 1) + \cal C (\xi_{\tau - 1} + \xi), \] 
where we used $\tau - 1 \geq 3 \xi_{\tau - 1} \log d$ as $\tau -1 \geq \xi^{1/2}/3$. 
This proves \eqref{eq:condition_eigenvalue_estimates} and, thus, \eqref{eq:eigen_error_bound} for any $\sigma(i) \in \cal U$ with the choice of $\tau$ from \eqref{eq:choice_tau}. 

In order to show that the right-hand side of \eqref{eq:eigen_error_bound} is controlled by the one in \eqref{lambda_estimate},  
we now distinguish the two cases, $\Lambda(\alpha_{\sigma(i)}) -2 \leq 3$ and $\Lambda(\alpha_{\sigma(i)}) -2  > 3$. 
In the latter case, $\tau = 2$ by \eqref{eq:choice_tau} and \eqref{lambda_estimate} follows immediately 
from \eqref{eq:eigen_error_bound} as $\xi_1 \leq \xi$. 
If $\Lambda(\alpha_{\sigma(i)}) -2 \leq 3$ then $\tau - 1 = (\Lambda(\alpha_{\sigma(i)}) -2)/3$ and, thus, $\xi_{\tau - 1} = 3 \xi_{\Lambda(\alpha_{\sigma(i)}) -2}$. 
Hence, \eqref{eq:eigen_error_bound} implies \eqref{lambda_estimate}. This concludes the proof of Theorem~\ref{thm:one_to_one}. 
\end{proof}

\subsection{Proof of Lemma \ref{lem:estimate_cut_graph}, Proposition \ref{prop:proof_lower_bound_approximate_eigenvector}, and Lemma \ref{lem:estime_block_matrix}}
\label{Subsec:TechnicalProofA}

\begin{proof}[Proof of Lemma \ref{lem:estimate_cut_graph}]
To begin with, we reduce the problem to the adjacency matrices by using the estimate \eqref{EA_estimate}.
Hence, with very high probability,
\begin{equation*}
\sqrt{d} \norm{H - H^\tau} \leq  \norm{\E A - \chi^\tau (\E A) \chi^\tau} +  \norm{A - A^\tau} \leq 2  + \norm{A^{\bb D_\tau}} \,,
\end{equation*}
where $A^{\bb D_\tau}$ is the adjacency matrix of the graph $\bb D_\tau \deq \bb G \setminus \bb G_\tau$. Hence, since $d^{-1/2} \leq C \xi_{\tau-1} $ by $d \leq 3 \log N$ and the definition \eqref{eq:main_error_term}, it suffices to show that 
$\norm{A^{\bb D_\tau}} \leq \cal C \xi_{\tau-1} \sqrt{d}$.

We know from Proposition \ref{prop:subgraph_separating_large_degrees} \ref{item:subgraph_cut_only_in_S_1} and \ref{item:subgraph_degrees} that with very high probability $\bb D_\tau$ consists of 
 (possibly overlapping)  stars\footnote{A \emph{star} around a vertex $x$ is a set of edges incident to $x$.} around vertices $x \in \cal V_\tau$ of central degree $D_x^{\bb D_\tau} \leq \cal C d \xi_{\tau-1}^2$. 
Moreover, with very high probability,
\begin{enumerate}[label=(\roman*)]
\item \label{item:ballone}
any ball  $B_{2 r_\star}(x)$ around $x \in \cal V_\tau$ has at most $\cal C$ cycles;
\item \label{item:balltwo}
any ball  $B_{2 r_\star}(x)$ around $x \in \cal V_\tau$ contains at most $\cal C d \xi_{\tau-1}^2$ vertices in $\cal V_\tau$.
\end{enumerate}
Claim \ref{item:ballone} follows from \cite[Corollary 5.6]{ADK19}, the definition \eqref{def_r_star}, and Lemma \ref{lem:upper_bound_degrees}. Claim \ref{item:balltwo} follows from \cite[Lemma 7.3]{ADK19} and $h((\tau -1)/2) \asymp (\tau-1)^2$ for $1 \leq \tau \leq 2$.

Let $x \in \cal V_\tau$. We claim that we can remove at most $\cal C$ edges of $\bb D_\tau$ incident to $x$ so that no cycle passes through $x$. Indeed, if there were more than $\cal C$ cycles in $\bb D_\tau$ passing through $x$, then at least one such cycle would  have to leave $B_{2 r_\star}(x)$ (by \ref{item:ballone}), which would imply that $B_{2 r_\star}(x)$ has at least $r_\star$ vertices in $\cal V_\tau$, which, by \ref{item:balltwo}, is impossible since $r_\star \geq 2 \cal C d \xi_{\tau-1}^2$ by $\tau \geq 1 + \xi^{1/2}$. See Figure \ref{fig:D} for an illustration of $\bb D_\tau$.

\begin{figure}[!ht]
\begin{center}
{\small %% Creator: Inkscape 1.0beta2 (2b71d25, 2019-12-03), www.inkscape.org
%% Accompanies image file 'fig1.pdf' (pdf, eps, ps)
\begingroup%
  \makeatletter%
  \providecommand\color[2][]{%
    \errmessage{(Inkscape) Color is used for the text in Inkscape, but the package 'color.sty' is not loaded}%
    \renewcommand\color[2][]{}%
  }%
  \providecommand\transparent[1]{%
    \errmessage{(Inkscape) Transparency is used (non-zero) for the text in Inkscape, but the package 'transparent.sty' is not loaded}%
    \renewcommand\transparent[1]{}%
  }%
  \providecommand\rotatebox[2]{#2}%
  \newcommand*\fsize{\dimexpr\f@size pt\relax}%
  \newcommand*\lineheight[1]{\fontsize{\fsize}{#1\fsize}\selectfont}%
  \ifx\svgwidth\undefined%
    \setlength{\unitlength}{260.41972479bp}%
    \ifx\svgscale\undefined%
      \relax%
    \else%
      \setlength{\unitlength}{\unitlength * \real{\svgscale}}%
    \fi%
  \else%
    \setlength{\unitlength}{\svgwidth}%
  \fi%
  \global\let\svgwidth\undefined%
  \global\let\svgscale\undefined%
  \makeatother%
  \begin{picture}(1,0.779667)%
    \lineheight{1}%
    \setlength\tabcolsep{0pt}%
    \put(0,0){\includegraphics[width=\unitlength,page=1]{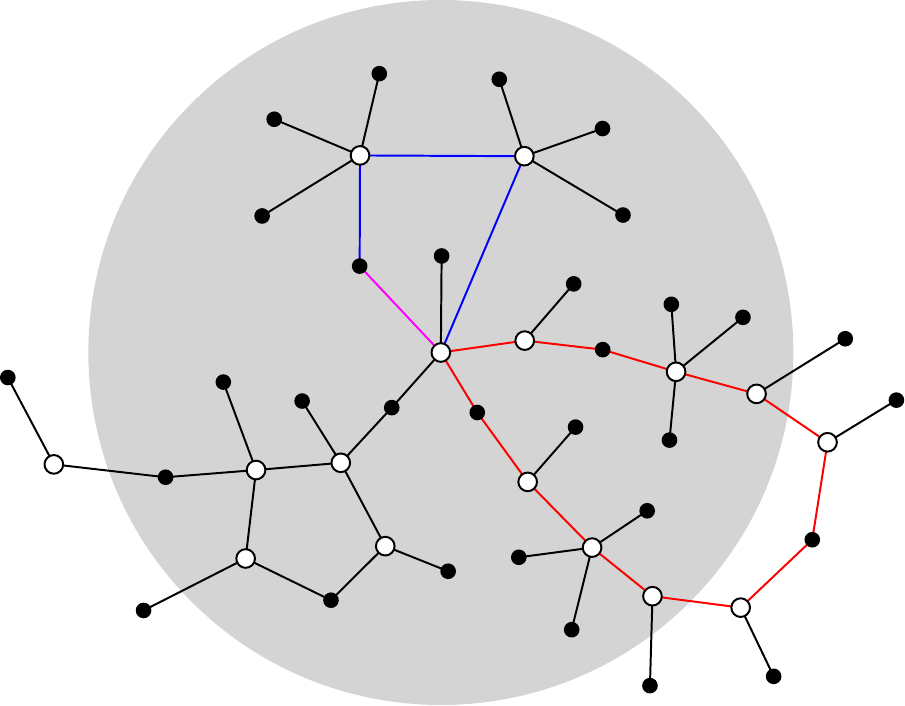}}%
    \put(0.43490939,0.38140106){\makebox(0,0)[lt]{\lineheight{1.25}\smash{\begin{tabular}[t]{l}$x$\end{tabular}}}}%
  \end{picture}%
\endgroup%
}
\end{center}
\caption{An illustration of a connected component of $\bb D_\tau$. Vertices of $\cal V_\tau$ are drawn in white and the other vertices in black. The ball $B_{2 r_\star}(x)$ around a chosen white vertex $x$ is drawn in grey, where $2 r_\star = 4$. The illustrated component of $\bb D_\tau$ has three cycles, two of which are in $B_{2 r_\star}(x)$. The blue and red cycles pass through $x$. The purple edge is removed from the blue cycle, i.e.\ it is put into the graph $\bb U_\tau$. With very high probability, the red cycle cannot appear, because it leaves the ball $B_{2 r_\star}(x)$ and therefore contains more white vertices in $B_{2 r_\star}(x)$ than allowed by property \ref{item:balltwo}. \label{fig:D}}
\end{figure}

Thus, we can remove a graph $\bb U_\tau$ from $\bb D_\tau$ such that $\bb U_\tau$ has maximal degree $\cal C$ and $\bb D_\tau \setminus \bb U_\tau$ is a forest of maximal degree $\cal Cd  \xi_{\tau-1}^2$ (by \ref{item:balltwo}). The claim now follows from Lemma \ref{lem:forest_bound}.
\end{proof}

\begin{proof}[Proof of Proposition \ref{prop:proof_lower_bound_approximate_eigenvector}]
We focus on the case $\sigma = +$; trivial modifications yield \eqref{eq:proof_lower_bound_approximate_eigenvector} for $\sigma = -$.
The basic strategy is to decompose $(H^\tau - \Lambda(\alpha_x))\f v_+^\tau(x)$ into several error terms that are estimated separately.
A similar argument was applied in \cite[Proposition~5.1]{ADK19} to the original graph $\mathbb{G}$ instead of $\bb G^\tau$, 
which however does not yield sharp enough estimates to reach the optimal scale $d \gg \sqrt{\log N}$ (see Section \ref{sec:forks_intro}).

We omit $x$ from the notation in this proof and write $u_i$, $\f v_+^\tau$ and $S_i^\tau$ instead of $u_i(x)$, $\f v^\tau_+(x)$ and $S_i^\tau(x)$. 
We define
\[ \f s^\tau_i \deq \frac{\f 1_{S_i^\tau}}{\norm{\f 1_{S_i^\tau}}}, \qquad \qquad N_i^\tau(y) \deq \abs{S_1^\tau(y) \cap S_i^\tau}\,. \]
Note that $(\f s^\tau_i)_{i=0}^{2r_\star}$ form an orthonormal system. Defining the vectors
\begin{equation} \label{eq:def_w_k} 
\begin{aligned}
 \f w_2 & \defeq \sum_{i=2}^{{r_\star}} \frac{u_i}{\sqrt{d \abs{S^\tau_i}}} \sum_{y \in S^\tau_{i-1}} \bigg( N_i^\tau(y) - \frac{\abs{S^\tau_i}}{\abs{S^\tau_{i-1}}} \bigg) \f 1_y,  \\ 
\f w_3 & \defeq u_2 \left( \frac{\sqrt{\abs{S^\tau_2}}}{\sqrt{d\abs{S^\tau_1}}} - 1 \right) \f s^\tau_1 + \sum_{i=2}^{{r_\star} - 1} \left[ u_{i+1} \left(\frac{\sqrt{\abs{S^\tau_{i+1}}}}{\sqrt{d\abs{S^\tau_i}}} -1 \right) 
 + u_{i-1} \left( \frac{\sqrt{\abs{S^\tau_i}}}{\sqrt{d \abs{S^\tau_{i-1}}}} - 1 \right)  \right] \f s^\tau_i , \\ 
 \f w_4 & \defeq u_{r_\star} \bigg(1 - \frac{1}{\sqrt{\alpha_x}} \bigg) \f s_{r_\star -1}^\tau + u_{{r_\star}-1} \Bigg( \frac{\sqrt{\abs{S^\tau_{r_\star}}}}{\sqrt{d \abs{S^\tau_{{r_\star}-1}}}}  - \frac{1}{\sqrt{\alpha_x - 1}} \Bigg) \f s^\tau_{r_\star}
 + u_{r_\star} \frac{\sqrt{\abs{S^\tau_{{r_\star}+1}}}}{\sqrt{d\abs{S^\tau_{r_\star}}}} \, \f s^\tau_{{r_\star}+1}\,,
\end{aligned}
\end{equation}
a straightforward computation using the definition of $\f v^\tau_+$ yields 
\begin{equation} \label{eq:decomposition_approximate_eigenvector}
 (H^\tau  - \Lambda(\alpha_x)) \f v_+^\tau = \f w_2 + \f w_3 + \f w_4. 
\end{equation}
For a detailed proof of \eqref{eq:decomposition_approximate_eigenvector} in a similar setup, we refer the reader to \cite[Lemma~5.2]{ADK19} (note that in the analogous calculation of \cite{ADK19} the left-hand side of \eqref{eq:decomposition_approximate_eigenvector} is multiplied by $\sqrt{d}$). 
The terms in \eqref{eq:decomposition_approximate_eigenvector} analogous to $\f w_0$ and $\f w_1$ in \cite{ADK19} vanish, respectively, because the projection $\chi^\tau$ is included in \eqref{eq:defAoperator} and because $\bb G_\tau |_{B_{2r_\star}^\tau}$ is a tree by Proposition \ref{prop:subgraph_separating_large_degrees} \ref{item:subgraph_tree}. 
The vector $\f w_4$ from \eqref{eq:def_w_k} differs from the one in \cite{ADK19} due to the special choice of $u_{r_\star}$ in \eqref{ui_definition}.

We now complete the proof of \eqref{eq:proof_lower_bound_approximate_eigenvector} by showing that each term on the right-hand side of \eqref{eq:decomposition_approximate_eigenvector} is bounded in norm
by $\mathcal{C} \xi$ with very high probability. 
We start with $\f w_3$ by first proving the concentration bound
\begin{equation} \label{S_tau_S_1}
\left|\frac{|S_{i+1}^\tau|}{d|S_i^\tau|}-1\right|=\mathcal{O}\left({\frac{\sqrt{\log N}}{d}}\right)
\end{equation}
with very high probability,
for $i = 1, \ldots, r_\star$.
To prove this, we use Proposition \ref{prop:subgraph_separating_large_degrees} \ref{item:subgraph_inclusion_S_i} and \ref{item:subgraph_S_i}, as well as \cite[Lemma 5.4]{ADK19}, to obtain  
\begin{equation} \label{eq:lower_bound_S_i_tau_over_S_i} 
\frac{|S_i^\tau|}{|S_i|}=1-\frac{|S_{i}\setminus S_{i}^\tau|}{|S_i|}\geq 1-\mathcal{C}\frac{\log N}{(\tau - 1)^2 d^2} 
\end{equation}
with very high probability,
where we used that $\alpha_x \geq 1$, and the assumption \cite[Eq.\ (5.13)]{ADK19} is satisfied by the definition \eqref{def_r_star}.
Therefore, invoking \cite[Lemma 5.4]{ADK19} in the following expansion yields 
\begin{equation}\label{eq:ratio_S_i_tau_S_i} 
\frac{\abs{S_{i+1}^\tau}}{d \abs{S_i^\tau}} = \frac{\abs{S_{i+1}}}{d \abs{S_i}}\, \frac{\abs{S_{i}}}{\abs{S_i^\tau}}\, \frac{\abs{S_{i+1}^\tau}}{\abs{S_{i+1}}} = \pbb{1 + \cal O \pbb{{\frac{\sqrt{\log N}}{d}}}} \pbb{1 + \cal O \pbb{\frac{\log N}{d^2 (\tau - 1)^2}}}
\end{equation}
with very high probability.
Hence, recalling the lower bound $\tau \geq 1 + \xi^{1/2}$, we obtain \eqref{S_tau_S_1}.

We take the norm in the definition of $\f w_3$, use the orthonormality of $(\f s_i^\tau)_{i=0}^{r_\star}$, and end up with 
\[
  \norm{\f w_3}^2  \leq  \left[ \left( \frac{\sqrt{\abs{S^\tau_2}}}{\sqrt{d \abs{S^\tau_1}}} -1 \right)^2 u_2^2 + 2 \sum_{i=2}^{{r_\star}-1} \left( \left(\frac{\sqrt{\abs{S^\tau_{i+1}}}}{\sqrt{d \abs{S^\tau_i}}} - 1\right)^2 u_{i+1}^2 + \left( \frac{\sqrt{\abs{S^\tau_i}}}{
\sqrt{d \abs{S^\tau_{i-1}}}} - 1 \right)^2 u_{i-1}^2 \right) \right]. 
\] 
Consequently, \eqref{S_tau_S_1} and $\sum_{i=0}^{r_\star} u_i^2 =1$ yield the desired bound on $\norm{\f w_3}$. 

In order to estimate $\norm{\f w_2}$, we use the definitions
\[ N_i(y) \deq \abs{S_1(y) \cap S_i}, \qquad Y_i \deq \frac{1}{\abs{S_{i-1}^{\tau}}}\sum_{y \in S_{i-1}^\tau} \big( N_i(y) - \E[N_i(y) | B_{i-1}] \big)^2\] 
and the Pythagorean theorem to obtain 
\begin{align}
\|\f w_2\|^2 & = \sum_{i=2}^{{r_\star}} \frac{u_i^2}{d \abs{S^\tau_i}} \sum_{y \in S^\tau_{i-1}} \bigg( N^\tau_i(y) - \frac{\abs{S^\tau_i}}{\abs{S^\tau_{i-1}}} \bigg)^2 \nonumber
\\ & \leq 4 \sum_{i=2}^{{r_\star}} \frac{u_i^2}{d \abs{S^\tau_i}} \sum_{y \in S^\tau_{i-1}} \bigg[ \big( N_i(y) - \E[N_i(y) | B_{i-1}]\big)^2 + \big(\E[N_i(y) | B_{i-1}] - d \big)^2 
\nonumber \\ & \qquad \qquad \qquad \qquad 
+\bigg(d- \frac{\abs{S^\tau_i}}{\abs{S^\tau_{i-1}}} \bigg)^2+(N_i^{\tau}(y)-N_i(y))^2 \bigg] \nonumber
\\ & \leq4 \max_{2\leq i\leq r_\star} \frac{\abs{S^{\tau}_{i-1}}}{d \abs{S^\tau_i}} \Big[ Y_i 
+ \cal C \log N+\big(\max_y D_y^{\bb G \setminus \bb G_\tau}\big)^2 \Big] \label{eq:w2_sumEstimate}
\end{align}
with very high probability.
Here, in the last step, 
we used \eqref{S_tau_S_1}, $\sum_{i=0}^{r_\star} u_i^2 =  1$  and $\abs{d - \E[N_i(y) | B_{i-1}]} = d \abs{B_{i-1}}/N \leq \mathcal{C}$ with very high probability due to 
\cite[Eq.~(5.12b)]{ADK19} and Lemma~\ref{lem:upper_bound_degrees}.

Next, we claim that
\begin{equation} \label{eq:bound_Z_i} 
 Y_i \leq \cal C \log N \log d 
\end{equation}
with very high probability, for $i = 2, \ldots, r_\star$. The proof of \eqref{eq:bound_Z_i} is based on a dyadic decomposition analogous to the one used in the proof of \cite[Eq.~(5.26)]{ADK19}.
We distinguish two regimes and estimate 
\begin{align} 
Y_i &\leq d +  \frac{1}{\abs{S_{i-1}^{\tau}}}\sum_{y \in S_{i-1}^\tau} \ind{|N_i(y) - \E[N_i(y) | B_{i-1}] |>d^{1/2}} \left( N_i(y) - \E[N_i(y) | B_{i-1}] \right)^2
\notag \\ \label{eq:decomposition_Z_i}
&\leq d + \frac{1}{\abs{S_{i-1}^{\tau}}}
\sum_{k=k_{\min}}^{0} d^2 \ee^{k+1} \abs{\cal N^\tau_{i,k} }
\end{align}
with very high probability, where we introduced
\[
k_{\min}\deq \lfloor -\log d \rfloor\,,\qquad \cal N^\tau_{i,k} \deq \Big\{y\in S_{i-1}^\tau \col d^2 \ee^k< \pb{N_i(y)-\E[N_i(y) | B_{i-1}]}^2 \leq d^2 \ee^{k+1}\Big\}\,.
\]
In \eqref{eq:decomposition_Z_i}, we used that, with very high probability, $\pb{N_i(y)-\E[N_i(y) | B_{i-1}]}^2  \leq d^2 \pb{(\tau - 1/2)^2 \vee 1} \leq d^2 \ee$, because
$y\in S_{i-1}^\tau$ implies the conditions $0\leq N_i(y)\leq D_y\leq \tau d $
due to Proposition \ref{prop:subgraph_separating_large_degrees} \ref{item:subgraph_paths} and $d/2 \leq \E[N_i(y) |B_{i-1}] \leq d$ 
with very high probability.
By Proposition~\ref{prop:subgraph_separating_large_degrees} \ref{item:subgraph_inclusion_S_i}, we have $\cal N_{i,k}^\tau \subset \cal N^{i-1}_k$, 
where $\cal N^{i-1}_k$ is defined as in the proof of \cite[Eq.~(5.26)]{ADK19}.  (Note that, in the notation of \cite{ADK19}, there is a one-to-one mapping between $A_{( B_{i-1})}$ and $B_i$.)
In this proof it is shown that, with very high probability,  
\[|\cal N^{i-1}_k| \leq \ell_k, \qquad \qquad   \ell_k \deq \frac{\cal C}{d}(|S_{i-1}|+\log N) \ee^{-k}. 
\]
Using \eqref{eq:lower_bound_S_i_tau_over_S_i} and \eqref{eq:ratio_S_i_tau_S_i}, and then plugging the resulting bound into \eqref{eq:decomposition_Z_i} concludes the proof of \eqref{eq:bound_Z_i}. 

Thus, we obtain $\norm{\f w_2} \leq \cal C \xi$ with very high probability, by starting from \eqref{eq:w2_sumEstimate} and using \eqref{S_tau_S_1}, \eqref{eq:bound_Z_i} and Proposition \ref{prop:subgraph_separating_large_degrees} \ref{item:subgraph_degrees} as well as the assumption $1 + \xi^{1/2} \leq \tau \leq 2$.

Finally, we estimate $\f w_4$. Since $\alpha_x \geq 2$ and $u_0 \leq 1$ we have that $u_{r_\star} + u_{r_\star-1} \leq 3 (\alpha_x - 1)^{-(r_\star -2)/2}$. 
The other coefficients of $\f s_{r_\star-1}^\tau$, $\f s_{r_\star}^\tau$ and $\f s_{r_\star+1}^\tau$ are bounded by $\cal C$ with very high probability, due to $\alpha_x \geq 2$ and \eqref{S_tau_S_1}, respectively. 
Therefore, \eqref{eq:smallness_alpha_x_minus_r_star} implies $\norm{\f w_4} \leq \cal{C} \xi$. This concludes the proof of Proposition~\ref{prop:proof_lower_bound_approximate_eigenvector}. 
\end{proof}

\begin{proof}[Proof of Lemma \ref{lem:estime_block_matrix}]
We have to estimate the norm of 
\begin{equation} \label{eq:rep_H_tau_minus_wh_H_tau} 
H^\tau -\wh H^\tau = \Pi^\tau H^\tau \Pi^\tau-\sum_{x \in \cal V} \sum_{\sigma = \pm} \sigma \Lambda(\alpha_x) \f v^\tau_\sigma(x) \f v^\tau_\sigma(x)^* +\ol \Pi^\tau H^\tau  \Pi^\tau + (\ol \Pi^\tau H^\tau \Pi^\tau)^*. 
\end{equation}
Each $x \in \cal V$ satisfies the condition of Proposition~\ref{prop:proof_lower_bound_approximate_eigenvector} since $\xi^{1/4} \geq C (\log d)^2 / \sqrt{\log N}$ (see \eqref{eq:lower_bound_xi_1_4}).
Hence, for any $x \in \cal V$ and $\sigma = \pm$, Proposition~\ref{prop:proof_lower_bound_approximate_eigenvector} yields
\begin{equation*}
H^\tau \f v_\sigma^\tau(x)= \sigma\Lambda(\alpha_x) \f v_\sigma^\tau(x) + \f e_\sigma^\tau(x)\,,
\qquad \supp \f e_\sigma^\tau(x) \subset B_{r_{\star}+1}^{\tau}(x) \,, \qquad \norm{\f e_\sigma^\tau(x)} \leq \cal C \xi
\end{equation*}
with very high probability, where the second statement follows from the first together with the definition \eqref{eq:def_vtau} of $\f v_\sigma^\tau(x)$ and Remark \ref{rem:v_tau_supp}.
By Proposition \ref{prop:subgraph_separating_large_degrees} \ref{item:subgraph_paths}, the balls $B_{2r_\star}^\tau(x)$ and $B_{2r_\star}^\tau(y)$ are disjoint for 
$x, y \in \cal V_\tau$ with $x \neq y$. Hence, in this case,  
$\f v_\sigma^\tau (x),\f e_\sigma^\tau (x) \perp \f v_{\sigma'}^\tau (y),\f e_{\sigma'}^\tau (y)$. 
For any $\f a = \sum_{x\in \cal V} \sum_{\sigma = \pm} a_{x,\sigma} \f v_\sigma^\tau(x)$, we obtain  
\begin{equation*}
\ol \Pi^\tau H^\tau \Pi^\tau \f a = \sum_{x \in \cal V} \sum_{\sigma = \pm} a_{x,\sigma} \ol \Pi^\tau H^\tau \f v_\sigma^\tau(x) = \ol \Pi^\tau \sum_{x \in \cal V} \sum_{\sigma = \pm} a_{x,\sigma} \f e_\sigma^\tau(x)\,.
\end{equation*}
Thus, with very high probability, $\norm{\ol \Pi^\tau H^\tau \Pi^\tau \f a}^2 \leq \sum_{x \in \cal V} \|\sum_{\sigma = \pm} a_{x,\sigma} \f e_\sigma^\tau(x)\|^2 \leq 4 \cal C^2\sum_{x\in \cal V} \sum_{\sigma = \pm} a_{x,\sigma}^2 \xi^2 = 4 \cal C^2 \xi^2 \norm{\f a}^2$ by orthogonality. 
Therefore, $\norm{\ol \Pi^\tau H^\tau \Pi^\tau} \leq  \cal C \xi $ with very high probability. Similarly, 
the representation 
\[\left(\Pi^\tau H^\tau \Pi^\tau-\sum_{x \in \cal V} \sum_{\sigma = \pm} \sigma \Lambda(\alpha_{x}) \f v^\tau_\sigma(x) \f v^\tau_\sigma(x)^*\right)\f a = \Pi^\tau
\sum_{x \in \cal V} \sum_{\sigma = \pm} a_{x,\sigma} \f e_\sigma^\tau(x)\]
yields the desired estimate on the sum of the two first terms on the right-hand side of \eqref{eq:rep_H_tau_minus_wh_H_tau}. 
\end{proof}

\subsection{Proof of Proposition \ref{pro:one-one_exact_correspondance}} \label{sec:estimatesB}

In this section we prove Proposition \ref{pro:one-one_exact_correspondance}. Its proof relies on two fundamental tools.

The first tool is a quadratic form estimate, which estimates $H$ in terms of the diagonal matrix of the vertex degrees. It is an improvement of \cite[Proposition 6.1]{ADK19}. To state it, for two Hermitian matrices $X$ and $Y$ we use the notation $X \leq Y$ to mean that $Y - X$ is a nonnegative matrix, and $\abs{X}$ is the absolute value function applied to the matrix $X$.

\begin{proposition} \label{prop:operator_upper_bound}
Let $4 \leq d \leq 3 \log N$. Then, with very high probability, we have 
\[
\abs{H}\leq I+(1+ 2d^{-1/2}) Q +\mathcal{C}\frac{\log N}{d^{2}}\vee d^{-1/2},
\] where $Q$ is the diagonal matrix with diagonal $(\alpha_x)_{x \in [N]}$.
\end{proposition}

The second tool is a delocalization estimate for an eigenvector $\f w$ of $\wh H^\tau$ associated with an eigenvalue $\lambda > 2$. Essentially, it says that $w_x$ is small at any $x \in \cal V_\tau$ unless $\f w$ happens to be the specific eigenvector $\f v^\tau_\pm(x)$ of $\wh H^\tau$, which is by definition localized 
around $x$. Thus, in any ball $B_{2 r_\star}^\tau(x)$ around $x \in \cal V_\tau$, all eigenvectors except  $\f v^\tau_\pm(x)$ are locally delocalized in the sense that their magnitudes at $x$ are small. Using that the balls $(B_{2 r_\star}^\tau(x))_{x \in \cal V_\tau}$ are disjoint, this implies that eigenvectors of $\ol \Pi^\tau H^\tau \ol \Pi^\tau$ have negligible mass on the set $\cal V$.

\begin{proposition}\label{thm:weak_delocalisation}
Let $d$ satisfy \eqref{d_assumption_localization}.
If $1 + \xi^{1/2} \leq \tau \leq 2$ then the following holds with very high probability.
Let $\lambda$ be an eigenvalue of $\wh{H}^{\tau}$ with $\lambda>2\tau+\cal C \xi$
and $\f w=(w_x)_{x \in [N]}$ its corresponding eigenvector.
\begin{enumerate}[label=(\roman*)]
\item \label{item:deloc1}
If $x \in \cal V$ and $\f v_\pm^{\tau}(x)\perp \f w$ or
if $x \in \cal V_\tau \setminus \cal V$ then 
\[
\frac{|w_{x}|}{\|\f w|_{B_{2r_\star}^\tau(x)}\|}\leq\frac{\lambda^2}{(\lambda-2 \tau- \cal C\xi)^{2}}\bigg(\frac{2 \tau+\cal C \xi}{\lambda}\bigg)^{r_{\star}}\,.
\]
\item \label{item:deloc2} Let $\f w$ be normalized. If $\f v_\pm^{\tau}(x)\perp \f w$ for all $x \in \cal V$
then 
\[
\sum_{x\in \cal V_\tau}w_{x}^{2}\leq\frac{\lambda^4}{(\lambda-2 \tau -\cal C \xi)^{4}}\bigg(\frac{2 \tau+\cal C \xi}{\lambda}\bigg)^{2r_{\star}}\,.
\]
\end{enumerate}
Analogous results hold for $\lambda < -2 \tau - \cal C \xi$. 
\end{proposition}

We may now conclude the proof of Proposition~\ref{pro:one-one_exact_correspondance}.

\begin{proof}[Proof of Proposition \ref{pro:one-one_exact_correspondance}]
By Proposition \ref{prop:operator_upper_bound}, Lemma \ref{lem:estime_block_matrix}, and Lemma~\ref{lem:estimate_cut_graph} we have 
\begin{equation}\label{eq:ihara_bass_estimate} 
\begin{aligned}
\wh H^\tau & \leq I+(1+2 d^{-1/2}) Q +\mathcal{C}\frac{\log N}{d^{2}}\vee d^{-1/2}+\|H-H^\tau\|+\|H^\tau - \wh H^\tau \|
\\ & \leq I+(1+2 d^{-1/2}) Q +\cal C (\xi+\xi_{\tau-1})
\end{aligned}
\end{equation}
with very high probability, where we used $\frac{\log N}{d^{2}}\vee d^{-1/2} \leq (\xi+\xi_{\tau-1})$. 

Arguing by contradiction, we assume that there exists an eigenvalue $\lambda > 2\tau  + \cal C' (\xi+\xi_{\tau-1})$ of $\ol \Pi^\tau H^\tau \ol \Pi^\tau$ for some $\cal C' \geq 2 \cal C$ to be chosen later. By the lower bound in \eqref{d_assumption_localization}, we may assume that $\cal C' \xi \leq 1$.
Thus, by the definition of $\wh{H}^\tau$, there is an eigenvector $\f w$ 
of $\wh{H}^\tau$ corresponding to $\lambda$, which is orthogonal to $\f v_\pm^\tau(x)$ for all $x \in \cal V$.
From \eqref{eq:ihara_bass_estimate}, we conclude 
\begin{equation} \label{eq:upper_bound_wh_lambda}  
\lambda = \langle \f w, \wh H^\tau \f w \rangle 
 \leq1+(1+2 d^{-1/2}) \sum_{x \notin \cal V_\tau}w_x^{2} \tau +(1+2 d^{-1/2}) \sum_{x \in \cal V_\tau}w_x^{2} \max_{y\in[N]} \alpha_y +\mathcal{C} (\xi+\xi_{\tau-1}). 
\end{equation} 
It remains to estimate the two sums on right-hand side of \eqref{eq:upper_bound_wh_lambda}.

Since $\f w \perp \f v^\tau_\pm(x)$ for all $x \in \cal V$, we can 
apply Proposition~\ref{thm:weak_delocalisation} \ref{item:deloc2}. 
We find
\begin{equation} \label{eq:upper_bound_exponential_factor_delocalization_bound} 
2r_\star \log \left( \frac{2 \tau +\cal C \xi}{\lambda}\right)\leq 2r_\star \log \left( \frac{2\tau+\cal C \xi}{2\tau+\cal C' \xi}\right)
\leq - 2 r_\star \frac{(\cal C' - \cal C) \xi}{2\tau + \cal C' \xi}
\leq - \frac{c (\cal C' - \cal C)}{3} \sqrt{\log N} \, \xi\,,
\end{equation} 
where in the last step we recalled the definition \eqref{def_r_star} and used that $\tau \leq 2$ and $\cal C' \xi \leq 1$. Using the estimate
\begin{equation*}
\frac{\lambda^4}{(\lambda-2\tau-\cal C \xi)^{4}} \leq \frac{C}{(\cal C' - \cal C)^4 \xi^4}\,,
\end{equation*}
combined with Proposition~\ref{thm:weak_delocalisation} \ref{item:deloc2}, \eqref{eq:upper_bound_exponential_factor_delocalization_bound} and Lemma \ref{lem:upper_bound_degrees}, yields
\begin{align*}
\frac{1}{\xi} \sum_{x \in \cal V_\tau}w_x^{2} \max_{y\in[N]} \alpha_y &\leq \frac{C \log N}{(\cal C' - \cal C)^4 \xi^5}
\exp \pbb{- \frac{c (\cal C' - \cal C)}{3} \sqrt{\log N} \, \xi}
\\
&\leq 
\frac{C d^5 \log N}{(\cal C' - \cal C)^4}
\exp \pbb{- \frac{c (\cal C' - \cal C)}{3} \frac{\log N}{d} \log d}
\\
&\leq 
\frac{C d^5 \log N}{(\cal C' - \cal C)^4} \frac{1}{d^8} \leq 1\,,
\end{align*}
where the third step follows by choosing $\cal C'$ large enough, depending on $\cal C$.

Plugging this estimate into \eqref{eq:upper_bound_wh_lambda} and using 
$\sum_x w_x^2\leq 1$ to estimate the first sum in \eqref{eq:upper_bound_wh_lambda}, we obtain $\lambda \leq 2\tau +  2 \mathcal{C} (\xi+\xi_{\tau-1})$. 
This is a contradiction to the assumption $\lambda > 2\tau  + \cal C' (\xi+\xi_{\tau-1})$. The proof of Proposition~\ref{pro:one-one_exact_correspondance} is therefore complete.
\end{proof}

\begin{proof}[Proof of Proposition \ref{prop:operator_upper_bound}]
We only establish an upper bound on $H$. The proof of the same upper bound on $-H$ is identical and, therefore, omitted. 

We introduce the matrices $H(t) = (H_{xy}(t))_{x, y \in [N]}$ and $M(t) = (\delta_{xy} m_x(t))_{x,y \in [N]}$ 
with entries  
\[
H_{xy}(t) \deq \frac{tH_{xy}}{t^2 -H_{xy}^2},\quad m_x(t)\deq 1+\sum_y \frac{H_{xy}^2}{t^2-H_{xy}^2}
\]

By the estimate on the spectral radius of the nonbacktracking matrix associated with $H$ in \cite[Theorem~2.5]{BBK1} and the Ihara--Bass-type formula in \cite[Lemma~4.1]{BBK1} 
we have, with very high probability, $\det (M(t)-H(t))\neq 0$ for all $t\geq 1+\cal C d^{-1/2}$. 
Because $(M(t)-H(t))\rightarrow I$ as $t\rightarrow \infty$, the matrix $M(t)-H(t)$ is positive definite for large enough $t$. 
By continuity of the eigenvalues, we conclude that all eigenvalues of $M(t)-H(t)$ stay positive for $t\geq 1+\cal C d^{-1/2}$, and hence
\begin{equation}
H(t)\leq M(t) \label{eq:IharaBassInequality}
\end{equation}
for all $t\geq 1+\cal C d^{-1/2}$ with very high probability. We now define the matrix $\Delta = (\Delta_{xy})_{x,y \in [N]}$ with  
\[ \Delta_{xy}\deq \begin{cases} H_{xy}(t)-t^{-1}H_{xy} & \text{ if }x\neq y \\ 
\sum_{y'} |H_{xy'}(t)-t^{-1}H_{xy'}| & \text{ if }x=y\,. \end{cases}\]
It is easy to check that $\Delta$ is a nonnegative matrix.
We also have
\[\sum_{y'} |H_{xy'}(t)-t^{-1}H_{xy'}|\leq \sum_{y'} \frac{|H_{xy'}|^3}{t(t^2 -H_{xy'}^2)}\leq \frac{2}{t^3d^{1/2}}\bigg(\alpha_x+\frac{1}{d} \bigg)\,,
 \]
where we used that $\abs{H_{xy}} \leq d^{-1/2}$ and $\sum_{y'} H_{x y'}^2 \leq \alpha_x + \frac{d}{N}$ by definition of $H$.
We use this to estimate the diagonal entries of $\Delta$ and obtain 
\begin{equation} \label{eq:errorAtA}
0\leq \Delta \leq H(t)-t^{-1}H+\frac{2}{t^3\sqrt{d}}Q+\frac{2}{t^3 d^{3/2}}. 
\end{equation}
On the other hand, for the diagonal matrix $M(t)$, we have the trivial upper bound 
\begin{equation}\label{eq:errorDM}
M(t)\leq I +t^{-2}Q+\mathcal{C}\frac{\log N}{d^2}
\end{equation}
since $\alpha_x \leq \mathcal{C} (\log N)/d$ with very high probability due to Lemma~\ref{lem:upper_bound_degrees}. 
Finally, combining \eqref{eq:IharaBassInequality}, \eqref{eq:errorAtA} and \eqref{eq:errorDM} yields 
\[
t^{-1} H\leq I +\pbb{t^{-2}+\frac{2}{t^3\sqrt{d}}}Q+\mathcal{C}\frac{\log N}{d^2}
\]
and Proposition \ref{prop:operator_upper_bound} follows by choosing $t= 1+\cal C d^{-1/2}$. 
\end{proof}

What remains is the proof of Proposition \ref{thm:weak_delocalisation}. The underlying principle behind the proof is the same as that of the Combes--Thomas estimate \cite{combes1973asymptotic}: the Green function $((\lambda - Z)^{-1})_{ij}$ of a local operator $Z$ at a spectral parameter $\lambda$ separated from the spectrum of $Z$ decays exponentially in the distance between $i$ and $j$, at a rate inversely proportional to the distance from $\lambda$ to the spectrum of $Z$. Here \emph{local} means that $Z_{ij}$ vanishes if the distance between $i$ and $j$ is larger than 1. Since a graph is equipped with a natural notion of distance and the adjacency matrix is a local operator, a Combes--Thomas estimate would be applicable directly on the level of the graph, at least for the matrix $H^\tau$. For our purposes, however, we need a \emph{radial} version of a Combes--Thomas estimate, obtained by first tridiagonalizing (a modification of) $\wh H^\tau$ around a vertex $x \in \cal V_\tau$ (see Appendix \ref{sec:mu}). In this formulation, the indices $i$ and $j$ have the interpretation of radii around the vertex $x$, and the notion of distance is simply that of $\N$ on the set of radii. Since $Z$ is tridiagonal, the locality of $Z$ is trivial, although the matrix $\wh H^\tau$ (or its appropriate modification) is not a local operator on the graph $\bb G_\tau$.

To ensure the separation of $\lambda > 2\tau + o(1)$ and the spectrum of $Z$, we cannot choose $Z$ to be the tridiagonalization of $\wh H^\tau$, since $\lambda$ is an eigenvalue of $\wh H^\tau$. In fact, $Z$ is the tridiagonalization of a new matrix $\wh H^{\tau,x}$, obtained by restricting $\wh H^\tau$ to the ball $B^\tau_{2 r_\star}(x)$ and possibly subtracting a suitably chosen rank-two matrix, which allows us to show $\norm{\wh H^{\tau, x}} \leq 2 \tau + o(1)$. By the orthogonality assumption on $\f w$, we then find that the Green function $((\lambda - Z)^{-1})_{i r_\star}$, $0 \leq i < r_\star$, and the eigenvector components in the radial basis $u_i$, $0 \leq i < r_\star$, satisfy the same linear difference equation. The exponential decay of $((\lambda - Z)^{-1})_{i r_\star}$ in $r_\star - i$ then implies that, for each $x \in \cal V_\tau$, $u_0^2 \leq o(1/\log N) \sum_{i = 0}^{r_*} u_i^2$. Going back to the original vertex basis, this implies that $w_x^2 \leq o(1/\log N) \|\f w|_{B_{2r_\star}^\tau(x)}\|^2$ for all $x \in \cal V_\tau$, from which Proposition \ref{thm:weak_delocalisation} follows since the balls $B_{2r_\star}^\tau(x)$, $x \in \cal V_\tau$, are disjoint.

\begin{proof}[Proof of Proposition \ref{thm:weak_delocalisation}]
For a matrix $M \in \R^{N \times N}$ and a set $V \subset [N]$, we use the notation $(M \vert_V)_{xy} \deq \ind{x,y \in V} M_{xy}$.

We begin with part \ref{item:deloc1}. We first treat the case $x\in \cal V$. To that end, we introduce the matrix
\begin{equation} \label{def_H-whx}
\wh H^{\tau,x} \deq \wh H^{\tau}|_{B_{2r_{\star}}^{\tau}(x)}-\Lambda(\alpha_{x}) \f v_+^{\tau}(x)\f v_+^{\tau}(x)^{*}+\Lambda(\alpha_{x}) \f v_{-}^{\tau}(x)\f v_{-}^{\tau}(x)^{*}\,.
\end{equation}
We claim that, with very high probability,
\begin{equation} \label{eq:norm_wh_H_tau_x} 
 \norm{\wh H^{\tau,x}} \leq 2 \tau+\cal C\xi \,.
\end{equation}

To show \eqref{eq:norm_wh_H_tau_x}, we begin by noting that, by Proposition \ref{prop:subgraph_separating_large_degrees} \ref{item:subgraph_paths} and \ref{item:subgraph_tree}, $\bb G_{\tau}$ restricted to $B_{2r_{\star}}^{\tau}(x)$ is a tree whose root $x$ has $\alpha_x d$ children and all other vertices have at most $\tau d$ children. Hence, Lemma \ref{lem:normTree} yields $\normb{H^{\tau}|_{B_{2r_{\star}}^{\tau}(x)}} \leq \sqrt{\tau} \Lambda(\alpha_x /\tau \vee 2)$.
Using Lemma~\ref{lem:estime_block_matrix} we find
\begin{equation} \label{H_B_est}
\|\wh H^{\tau}|_{B_{2r_{\star}}^\tau(x)}-H^{\tau}|_{B_{2r_{\star}}(x)}\|\leq\cal C \xi
\end{equation}
with very high probability, and since $\f v_\pm^\tau(x)$ is an eigenvector of $\wh H^\tau |_{B_{2r_{\star}}^{\tau}(x)}$ with eigenvalue $\pm \Lambda(\alpha_x)$, we conclude
\begin{equation} \label{basic_estimate_wh_H}
\norm{\wh H^{\tau,x}} \leq  \sqrt{\tau} \Lambda(\alpha_x /\tau \vee 2) + \cal C \xi
\end{equation}
with very high probability. The estimate \eqref{basic_estimate_wh_H} is rough in the sense that the subtraction of the two last terms of \eqref{def_H-whx} is not needed for its validity (since $\Lambda(\alpha_x) \leq \sqrt{\tau} \Lambda(\alpha_x/\tau \vee 2)$). Nevertheless, it is sufficient to establish \eqref{eq:norm_wh_H_tau_x} in the following cases, which may be considered degenerate.

If $\alpha_x \leq 2 \tau$ then \eqref{basic_estimate_wh_H} immediately implies \eqref{eq:norm_wh_H_tau_x}, since $\sqrt{\tau} \leq \tau$. Moreover, if $\alpha_x > 2 \tau$ and $\Lambda(\alpha_x) \leq 2 \sqrt{\tau} + \cal C \xi$, then \eqref{basic_estimate_wh_H} implies
\begin{equation*}
\norm{\wh H^{\tau,x}} \leq  \sqrt{\tau} \Lambda(\alpha_x /\tau) + \cal C \xi \leq \sqrt{\tau} \Lambda(\alpha_x) + \cal C \xi \leq 2 \tau + 3 \cal C \xi\,,
\end{equation*}
which is \eqref{eq:norm_wh_H_tau_x} after renaming the constant $\cal C$.

Hence, to prove \eqref{eq:norm_wh_H_tau_x}, it suffices to consider the case $\Lambda(\alpha_x) > 2 \sqrt{\tau} + \cal C \xi$. By Proposition \ref{prop:subgraph_separating_large_degrees} \ref{item:subgraph_paths} and \ref{item:subgraph_tree}, $\bb G_{\tau}$ restricted to $B_{2r_{\star}}^{\tau}(x) \setminus \{x\}$ is a forest of maximal degree at most $\tau d$. Lemma~\ref{lem:forest_bound} therefore yields $\|H^{\tau}|_{B_{2r_{\star}}^\tau(x) \setminus \{x\}}\|\leq2\sqrt{\tau}$.
Moreover, the adjacency matrix of the star graph consisting of all edges of $\bb G_\tau$ incident to $x$ has precisely two nonzero eigenvalues, $\pm \sqrt{d \alpha_x}$. By first order perturbation theory, we therefore conclude that $H^{\tau}|_{B_{2r_{\star}}^\tau(x)}$ has at
most one eigenvalue strictly larger than $2\sqrt{\tau}$ and at most
one strictly smaller than $-2\sqrt{\tau}$. Using \eqref{H_B_est} we conclude that $\wh H^{\tau}|_{B_{2r_{\star}}^\tau(x)}$ has at most one eigenvalue strictly
larger than $2\sqrt{\tau}+\cal C \xi$ and at most one strictly smaller than
$-2\sqrt{\tau}-\cal C \xi$. Since $\f v_+^{\tau}(x)$ (respectively $\f v_-^{\tau}(x)$) is an eigenvector
of $\wh H^{\tau}|_{B_{2r_{\star}}^\tau(x)}$ with eigenvalue $\Lambda(\alpha_{x})$ (respectively $-\Lambda(\alpha_{x})$), and since $\Lambda(\alpha_x) > 2 \sqrt{\tau} + \cal C \xi$, we conclude \eqref{eq:norm_wh_H_tau_x}. 

Next, let $(\f g_i)_{i=0}^{r_\star}$ be the Gram--Schmidt orthonormalization of the vectors $((\wh H^{\tau,x})^i \f 1_x)_{i=0}^{r_\star}$. 
We claim that 
\begin{equation} \label{eq:b_g_i_support} 
\supp \f g_i \subset B_{r_\star+i}^\tau(x)\,. 
\end{equation}  
for $i = 0, \ldots, r_\star$. The proof proceeds by induction.
The base case for $i =0$ holds trivially. For the induction step, it suffices to prove for $0 \leq i < r_\star$ that if $\supp \f g_i \subset B_{r_\star+i}^\tau(x)$ then
\begin{equation} \label{gi_supp}
\supp (\wh H^{\tau,x} \f g_i) \subset B_{r_\star+i+1}^\tau(x)
\end{equation}
To that end, we note that by Proposition \ref{prop:subgraph_separating_large_degrees} \ref{item:subgraph_paths} we have  $\wh H^{\tau,x} = \pb{\ol \Pi^{\tau} H^\tau \ol \Pi^\tau} |_{B_{2r_{\star}}^{\tau}(x)}$.
Hence, by induction assumption, Proposition \ref{prop:subgraph_separating_large_degrees} \ref{item:subgraph_paths}, and Remark \ref{rem:v_tau_supp}, 
\begin{equation*}
\wh H^{\tau,x} \f g_i = \pbb{I - \sum_{\sigma = \pm} \f v^\tau_\sigma(x) \f v^\tau_\sigma(x)^*} H^\tau
\pbb{I - \sum_{\sigma = \pm} \f v^\tau_\sigma(x) \f v^\tau_\sigma(x)^*} \f g_i\,,
\end{equation*}
and we conclude \eqref{gi_supp}, as $\supp \f v_\sigma^\tau(x) \subset B_{r_\star}^\tau(x)$.

Let $Z = (Z_{ij})_{i,j=0}^{r_\star}$, $Z_{ij} \deq \scalar{\f g_i}{\wh H^{\tau,x} \f g_j}$, be the tridiagonal representation of $\wh H^{\tau,x}$ up to radius $r_\star$ (see Appendix \ref{sec:mu} below).
Owing to \eqref{eq:norm_wh_H_tau_x}, we have 
\begin{equation} \label{eq:norm_wh_M} 
\|Z\|\leq 2 \tau+\cal C \xi.  
\end{equation}

We set $u_{i}\deq\langle \f g_{i},\f w\rangle$ for any $0 \leq i \leq r_\star$. Because $\f w$ is an eigenvector of $\wh H^{\tau}$ that is orthogonal to $\f v_\pm^\tau(x)$, for any $i<r_\star$, 
\eqref{eq:b_g_i_support} implies 
\begin{equation} \label{eq:wh_lambda_u_i} 
\begin{aligned}
\lambda u_i & = \left\langle \f g_i,\left(\wh H^\tau-\Lambda(\alpha_{x}) \f v_+^{\tau}(x)\f v_+^{\tau}(x)^{*}+\Lambda(\alpha_{x}) \f v_{-}^{\tau}(x)\f v_{-}^{\tau}(x)^{*}\right)\f w\right\rangle \\ &
=\left\langle\wh H^{\tau,x}\f g_i, \f w \right\rangle \\ 
& = \langle Z_{ii}\f g_{i}+Z_{i\,i+1}\f g_{i+1}+Z_{i\,i-1}\f g_{i-1}, \f w\rangle \\ 
& =Z_{ii}u_{i}+Z_{i\,i+1}u_{i+1}+Z_{i\,i-1}u_{i-1}\, 
\end{aligned}
\end{equation}
with the conventions $u_{-1}=0$ and $Z_{0,-1}=0$. 
Let $G(\lambda) \deq (\lambda - Z)^{-1}$
be the resolvent of $Z$ at $\lambda$. Note that $\lambda-Z$ is invertible since $\lambda > \norm{Z}$ by assumption and \eqref{eq:norm_wh_M}.
Since $\pb{(\lambda - Z) G(\lambda)}_{i \, r_{\star}} = 0$ for $i<r_\star$, we find
\[
\lambda G_{i r_{\star}}(\lambda)=Z _{ii}G_{i r_{\star}}(\lambda)+Z _{i\,i+1}G_{i +1 \, r_{\star}}(\lambda)+Z _{i\,i-1}G_{i -1\, r_{\star}}(\lambda).
\]
Therefore $(G_{i r_{\star}}(\lambda))_{i\leq r_\star}$ and $\left(u_{i}\right)_{i\leq r_\star}$ satisfy the same linear
recursive equation (cf.\ \eqref{eq:wh_lambda_u_i}); solving them recursively from $i = 0$ to $i = r_\star$ yields
\begin{equation} \label{eq:G_entries_u_i_s}
\frac{G_{i r_{\star}}(\lambda)}{G_{r_\star r_{\star}}(\lambda)}=\frac{u_{i}}{u_{r_{\star}}}
\end{equation}
for all $i\leq r_{\star}$. Moreover, as $\lambda>\|Z\|$ by assumption and \eqref{eq:norm_wh_M}, we have the convergent Neumann series $G(\lambda)= \frac{1}{\lambda}\sum_{k\geq0}(Z / \lambda)^{k}$.
Thus, the offdiagonal entries of the resolvent satisfy 
\begin{align*}
G_{0 r_{\star}}(\lambda) =\frac{1}{\lambda}\sum_{k \geq 0} \pb{(Z / \lambda)^{k}}_{0 r_\star}\,.
\end{align*}
Since $Z$ is tridiagonal, we deduce that $\pb{(Z / \lambda)^k}_{0 r_\star} = 0$ if $k < r_\star$, so that, by \eqref{eq:norm_wh_M},
\begin{equation} \label{eq:G_diagonal} 
|G_{0 r_{\star}}(\lambda)| \leq \bigg(\frac{2 \tau+\cal C \xi}{\lambda}\bigg)^{r_{\star}}\frac{1}{\lambda-2\tau-\cal C \xi}\,.
\end{equation}
On the other hand, for the diagonal entries of the resolvent, we get, by splitting the summation over $k$ into even and odd values,
\begin{multline} \label{eq:G_offdigonal} 
G_{r_\star r_\star}(\lambda) = \frac{1}{\lambda} \sum_{k \geq 0} \pb{(Z/\lambda)^k}_{r_\star r_\star}
= \frac{1}{\lambda} \sum_{k \geq 0} \pB{(Z/\lambda)^{k} (I + Z/\lambda) (Z/\lambda)^{k}}_{r_\star r_\star}
\\
\geq 
\frac{1}{\lambda} (I + Z/\lambda)_{r_\star r_\star}
\geq
\frac{1}{\lambda}\bigg(1-\frac{2\tau+\cal C \xi}{\lambda}\bigg)\,,
\end{multline}
where in the thid step we discarded the terms $k > 0$ to obtain a lower bound using that $I + Z/\lambda \geq 0$ by \eqref{eq:norm_wh_M}, and in the last step we used \eqref{eq:norm_wh_M}.
Hence, the definition of $u_i$ and \eqref{eq:b_g_i_support} imply 
\[
\frac{|w_{x}|}{\|\f w|_{B^\tau_{2r_{\star}}(x)}\|}\leq\frac{|u_{0}|}{\left(\sum_{i=0}^{r_\star} u_{i}^2\right)^{1/2}}\leq\frac{|u_{0}|}{|u_{r_{\star}}|}=\frac{\abs{G_{0r_{\star}}(\lambda)}}{G_{r_{\star} r_{\star}}(\lambda)}\leq\frac{\lambda^2}{(\lambda-2\tau-\cal C \xi)^{2}}\bigg(\frac{2\tau+\cal C \xi}{\lambda}\bigg)^{r_{\star}}.
\]
Here, we used \eqref{eq:G_entries_u_i_s} in third step and \eqref{eq:G_diagonal} as well as \eqref{eq:G_offdigonal}  in the last step. 
This concludes the proof of \ref{item:deloc1} for $x \in \cal V$.

In the case $x\in \cal V_\tau \setminus \cal V$, we set $\wh H^{\tau,x} \deq \wh H^{\tau} |_{B_{2r_{\star}}^{\tau}(x)}$. We claim that \eqref{eq:norm_wh_H_tau_x} holds. To see that, we use Proposition \ref{prop:subgraph_separating_large_degrees} \ref{item:subgraph_paths} and \ref{item:subgraph_tree} as well as Lemma \ref{lem:normTree} with $p = d(2 + \xi^{1/4})$ and $q = d \tau$ to obtain 
\[ \norm{H^\tau |_{B_{2r_{\star}}^{\tau}(x)}} \leq \sqrt{\tau} \Lambda((2 + \xi^{1/4})/\tau \vee 2) \leq 2 \tau. \] 
Here, the last step is trivial if $\tau \geq 1 + \xi^{1/4}/2$ and, if $\tau \in [1 + \xi^{1/2}, 1 + \xi^{1/4}/2]$, we used that $f(\tau) \deq \sqrt{\tau}\Lambda((2 + \xi^{1/4}) / \tau)/(2\tau)$ is 
monotonically decreasing on this interval and $f(1 + \xi^{1/2}) \leq 1$, as can be seen by an explicit analysis of the function $f$.
Now we may take over the previous argument verbatim to prove \ref{item:deloc1} for $x \in \cal V_\tau \setminus \cal V$.

Finally, we prove \ref{item:deloc2}. By \ref{item:deloc1} we have 
\begin{align*}
\sum_{x\in \cal V_\tau}w_{x}^2
\leq \sum_{x \in \cal V_\tau} \|\f w|_{B^\tau_{2r_{\star}}(x)}\|^2 \frac{\lambda^4}{(\lambda-2\tau-\cal C \xi)^{4}}\bigg(\frac{2\tau+\cal C \xi}{\lambda}\bigg)^{2 r_{\star}}
\leq \frac{\lambda^4}{(\lambda-2\tau-\cal C \xi)^{4}}\bigg(\frac{2\tau+\cal C \xi}{\lambda}\bigg)^{2 r_{\star}}\,,
\end{align*}
where we used that the the balls $\{B^\tau_{2r_\star}(x) \col x \in \cal V_\tau\}$ are disjoint, which implies $1=\|\f w\|^2 \geq \sum_{x\in \cal V_\tau} \|\f w|_{B^\tau_{2r_{\star}}(x)}\|^2$.
\end{proof}

\subsection{Proof of Proposition \ref{prop:subgraph_separating_large_degrees}} \label{sec:pf_pruning}

We conclude this section with the proof of Proposition \ref{prop:subgraph_separating_large_degrees}.

\begin{proof}[Proof of Proposition \ref{prop:subgraph_separating_large_degrees}] 
Parts \ref{item:subgraph_paths}--\ref{item:subgraph_degrees} follow immediately from parts (i)--(iv) and (vi) of \cite[Lemma~7.2]{ADK19}. To see this, we remark that the function $h$ from \cite{ADK19} satisfies $h((\tau - 1)/2) \asymp (\tau - 1)^2$ for $1 < \tau \leq 2$. Moreover, by Lemma \ref{lem:upper_bound_degrees} and the upper bound on $d$, we have $\max_x D_x \leq \cal C \log N$ with very high probability. Hence, choosing the universal constant $c$ small enough in \eqref{def_r_star} and recalling the lower bound on $\tau - 1$, in the notation of \cite[Equations (5.1) and (7.2)]{ADK19} we obtain for any $x \in \cal V_\tau$ the inequality $2 r_\star \leq (\frac{1}{4} r_x ) \wedge (\frac{1}{2} r(\tau))$ with very high probability. This yields parts \ref{item:subgraph_paths}--\ref{item:subgraph_degrees}.

It remains to prove \ref{item:subgraph_S_i}, which is the content of the rest of this proof. From now on we systematically omit the argument $x$ from our notation.
Part \ref{item:subgraph_degrees} already implies the bound
\begin{equation} \label{S1_est1}
|S_{1} \setminus S_{1}^{\tau}|=D_x^{\bb G \setminus \bb G_\tau}\leq\mathcal{C}\frac{\log N}{(\tau-1)^2 d}
\end{equation}
with very high probability,
which is \eqref{eq:sum_degree} for $i=1$.

From \cite[Eq.\ (7.13)]{ADK19} we find
\begin{equation*}
\abs{S_i \setminus S_i^\tau} \leq \sum_{y \in S_1 \setminus S_1^\tau} \abs{S_{i-1}(y)}\,.
\end{equation*}
(As a guide to the reader, this estimate follows from the construction of $\bb G_\tau$ given in \cite[Proof of Lemma 7.2]{ADK19}, which ensures that if a vertex $z \in S_i$ is not in $S_i^\tau$ then any path in $\bb G$ of length $i$ connecting $z$ to $x$ is cut in $\bb G_\tau$ at its edge incident to $x$.) Hence, in order to show \ref{item:subgraph_S_i} for $i \geq 2$, it suffices to prove
\begin{equation}\label{eq:sumN_i}
\sum_{y \in S_1 \setminus S_1^\tau} \abs{S_{i-1}(y)} \leq \mathcal{C}\frac{\log N}{(\tau-1)^2}d^{i-2}
\end{equation}
with very high probability, for all $2 \leq i \leq 2 r_\star$.

We start with the case $i=2$. We shall use the relation 
\begin{equation} \label{eq:relation_S_1_and_N_2} 
 \sum_{y \in S_1 \setminus S_1^\tau} \abs{S_{1}(y)} =  \sum_{y \in S_1 \setminus S_1^\tau} N_2(y) + \sum_{y \in S_1\setminus S_1^\tau} \abs{S_1(y) \cap S_1} + \abs{S_1 \setminus S_1^\tau}\,,
\end{equation}
where, for $y \in S_1$, we introduced  $N_2(y) \deq \abs{S_{1}(y) \cap S_2}$.
Note that $N_2(y)$ is the number of vertices in $S_2$ connected to $x$ via a path of minimal length passing through $y$. 
The identity \eqref{eq:relation_S_1_and_N_2} is a direct consequence of $\abs{S_1(y)} = \abs{S_1(y) \cap S_2} + \abs{S_1(y) \cap S_1} + \abs{S_1(y) \cap S_0}$ using the definition of $N_2$ and 
$\abs{S_1(y) \cap S_0} = \abs{S_1(y) \cap \{x \}} = 1$. 

The second and third terms of \eqref{eq:relation_S_1_and_N_2} are smaller than the right-hand side of \eqref{eq:sumN_i} for $i=2$ due to \cite[Eq.~(5.23)]{ADK19} and \eqref{S1_est1}, respectively. 
Hence, it remains to estimate the first term on the right-hand side of \eqref{eq:relation_S_1_and_N_2} in order to prove \eqref{eq:sumN_i} for $i =2$.

To that end, we condition on the ball $B_1$ and abbreviate $\P_{B_1}(\cdot ) \deq \P(\, \cdot \mid B_1)$. Since
\begin{equation}
N_2(y) = \sum_{z \in [N] \setminus B_1} A_{yz}\,,
\end{equation}
we find that conditioned on $B_1$ the random variables $(N_2(y))_{y \in S_1}$ are independent $\op{Binom}(N - \abs{B_1}, d/N)$ random variables. We abbreviate $\Gamma \deq \frac{\log N}{(\tau-1)^2}$. For given $\cal C, \cal C'$, we set $\cal C'' \deq \cal C' + 2 \cal C$ and estimate
\begin{align}
& \mathbb{P}_{B_1}\pBB{\sum_{y\in S_1\setminus S_1^\tau} N_2(y)
\geq \cal C'' \Gamma} 
\notag \\ & \leq \mathbb{P}_{B_1}\pBB{\sum_{y\in S_1\setminus S_1^\tau} \ind{N_2(y) \geq 2 d}N_2(y)
\geq (\cal C'' - 2 \cal C  )\Gamma}+\mathbb{P}_{B_1}\pBB{\sum_{y\in S_1\setminus S_1^\tau} \ind{N_2(y)< 2 d}N_2(y)
\geq 2 \cal C \Gamma}
\notag \\ \label{P_B_1}
& \leq \mathbb{P}_{B_1}\pBB{\sum_{y\in S_1} \ind{2d \leq N_2(y) \leq N^{1/4}}N_2(y)
\geq \cal C'\Gamma}
+ \sum_{y \in S_1} \P_{B_1}\pb{N_2(y) \geq N^{1/4}}
+\mathbb{P}_{B_1}\pb{ |S_1\setminus S_1^\tau|
\geq \cal C \Gamma d^{-1}}.
\end{align}
In order to estimate the first term on the right-hand side of \eqref{P_B_1}, we shall prove that if $\abs{B_1} \leq N^{1/4}$ then
\begin{equation}\label{eq:bound_markov_geqtau}
\mathbb{E}_{B_1} \qB{\exp \pB{\ind{2d \leq N_2(y) \leq N^{1/4}} N_2(y)t}} \leq 2
\end{equation}
for all $y\in S_1$ and $t \leq 1/8$. To that end, we estimate
\begin{equation*}
\mathbb{E}_{B_1} \qB{\exp \pB{\ind{2d \leq N_2(y) \leq N^{1/4}} N_2(y)t}} \leq
1 + \mathbb{E}_{B_1} \qB{\ind{2d \leq N_2(y) \leq N^{1/4}} \ee^{N_2(y)t}}\,.
\end{equation*}
With Poisson approximation, Lemma \ref{lem:binomial_estimate} below, we obtain (assuming that $2d$ is an integer to simplify notation)
\begin{align*}
\mathbb{E}_{B_1} \qB{\ind{2d \leq N_2(y) \leq N^{1/4}} \ee^{N_2(y)t}}
&= \sum_{2d \leq k \leq N^{1/4}}\frac{(d - d \abs{B_1}/N)^k \ee^{tk}}{k!}\ee^{-d + d \abs{B_1}/N} \pb{1+O(N^{-1/2})}
\\
&\mspace{-100mu}\leq \sum_{k\geq 2 d}\frac{d^k \ee^{tk}}{k!}\ee^{-d} \pb{1+O(N^{-1/2})}
 =\frac{d^{2d} \ee^{2td}}{(2d)!}\ee^{-d}\sum_{i\geq 0}\frac{d^i \ee^{t i}}{\prod_{j=2 d+1}^{2d + i}j} \pb{1+O(N^{-1/2})}
 \\
&\mspace{-100mu} \leq \frac{d^{2d} \ee^{2td}}{(2d)!}\ee^{-d} \sum_{i\geq 0}\frac{2d^i \ee^{t i}}{(2d)^i}
 = \frac{d^{2d} \ee^{2td}}{(2d)!}\ee^{-d}\frac{2}{(1-e^t/2)}.
\end{align*}
By Stirling's approximation we get
\begin{align*}
\log\left(\frac{d^{2d} \ee^{2td}}{(2d)!}\ee^{-d}\right)
&= d\left(2t- 2 \log 2+1 \right)-\frac{1}{2}\log(4 \pi  d)+ \oo  (1).
\end{align*}
The term in the parentheses on the right-hand side is negative for $t \leq 1/8$, and hence
\begin{equation*}
\mathbb{E}_{B_1} \qB{\ind{2d \leq N_2(y) \leq N^{1/4}} \ee^{N_2(y)t}} \leq 1
\end{equation*}
for large enough $d$, which gives \eqref{eq:bound_markov_geqtau}. Since the family $(N_2(y))_{y \in S_1}$ is independent conditioned on $B_1$, we can now use Chebyshev's inequality to obtain, for $0 \leq t \leq 1/8$,
\begin{align*}
\mathbb{P}_{B_1}\pBB{\sum_{y\in S_1} \ind{2d \leq N_2(y) \leq N^{1/4}} N_2(y)
\geq \cal C'\Gamma}& \leq \frac{\max_{y \in S_1}\left( \mathbb{E}_{B_1} \exp \pB{\ind{2d \leq N_2(y) \leq N^{1/4}} N_2(y)t}\right)^{|S_1|}}{\ee^{t \cal C'\Gamma }} \\ & \leq \exp\left( \abs{S_1} \log 2 - \cal C'\frac{t}{(\tau-1)^2}\log N \right)\,.
\end{align*}
Now we set $t = 1/8$, recall the bound $\tau \leq 2$, plug this estimate back into \eqref{P_B_1}, and take the expectation. We use Lemma \ref{lem:upper_bound_degrees} to estimate $\abs{S_1}$, which in particular implies that $\abs{B_1} \leq N^{1/4}$ with very high probability; this concludes the estimate of the expectation of the first term of \eqref{P_B_1} by choosing $\cal C'$ large enough. Next, the expectation of the second term is easily estimated by Lemma \ref{lem:upper_bound_degrees} since $N_2(y)$ has law $\op{Binom}(N - \abs{B_1}, d/N)$ when conditioned on $B_1$. Finally, the expectation of the last term of \eqref{P_B_1} is estimated by \eqref{S1_est1} by choosing $\cal C$ large enough. This concludes the proof of \eqref{eq:sumN_i} for $i = 2$.

We now prove \eqref{eq:sumN_i} for $i + 1$ with $i\geq 2$ by induction. 
Using \cite[Lemma 5.4 (ii)]{ADK19} combined with Lemma \ref{lem:upper_bound_degrees}, we deduce that
\[\abs{S_{i}(y)}\leq d \abs{S_{i-1}(y)}+\cal C \sqrt{d \abs{S_{i-1}(y)}\log N}\]
with very high probability for all $y\in S_1\setminus S_1^{\tau}$ and all $i \leq r_\star$. Therefore, using the induction assumption, i.e.\ \eqref{eq:sumN_i} for $i$, we obtain 
\begin{align*}
\sum_{y\in S_1\setminus S_1^\tau} \abs{S_{i}(y)}& \leq\mathcal{C}\frac{\log N}{(\tau-1)^2}d^{i-1}+\cal C \sqrt{d \log N}\sum_{y\in S_1\setminus S_1^\tau}\sqrt{\abs{S_{i-1}(y)}}
\\
& \leq\mathcal{C}\frac{\log N}{(\tau-1)^2}d^{i-1}+\cal C \sqrt{d \log N} |S_1\setminus S_1^\tau | \pBB{ \sum_{y\in S_1\setminus S_1^\tau}\frac{\abs{S_{i-1}(y)}}{|S_1\setminus S_1^\tau|}}^{1/2}
\\ 
& \leq\mathcal{C}\frac{\log N}{(\tau-1)^2}d^{i-1}+\cal C \sqrt{d \log N} \frac{\log N}{d (\tau-1)^2} \sqrt{d^{i-1}}
\end{align*}
with very high probability, 
where we used the concavity of $\sqrt{\,\cdot\,}$ in the second step, \eqref{S1_est1} and \eqref{eq:sumN_i} for $i$ in the last step. Since $\sqrt{d^i \log N}\leq d^{i/2+1}\leq d^i$ for $i\geq 2 $ and the 
sequence $(d^{1-i/2})_{i \in \N}$ is summable, this proves \eqref{eq:sumN_i} for $i+1$ with a constant $\cal C$ independent of $i$. This concludes the proof of Proposition \ref{prop:subgraph_separating_large_degrees}. 
\end{proof}

\section{The delocalized phase} \label{sec:delocalization} 

In this section we prove Theorem \ref{thm:delocalization}. In fact, we state and prove a more general result, Theorem~\ref{thm:local_law} below, which immediately implies Theorem~\ref{thm:delocalization}.

\subsection{Local law}
Theorem~\ref{thm:local_law} is a \emph{local law} for a general class of sparse random matrices of the form
\begin{equation} \label{eq:def_M} 
M = H + f \f e \f e^*\,,
\end{equation}
where $f \geq 0$ and $\f e \deq N^{-1/2}(1,1,\dots,1)^*$. Here $H$ is a Hermitian random matrix satisfying the following definition.

\begin{definition} \label{def:sparse} Let $0 < d < N$. A \emph{sparse matrix} is a complex Hermitian $N\times N$ matrix $H=H^* \in \bb C^{N \times N}$ whose entries $H_{ij}$ satisfy the following conditions.
	\begin{enumerate}
		\item[(i)] The upper-triangular entries ($H_{ij}\col 1 \leq i \leq j\leq N$) are independent.
		\item[(ii)] We have $\bb E H_{ij}=0$ and $ \bb E \abs{H_{ij}}^2=(1 + O(\delta_{ij}))/N$ for all $i,j$.
		\item[(iii)] Almost surely, $\abs{H_{ij}} \leq K d^{-1/2}$ for all $i,j$ and some constant $K$.
	\end{enumerate}
\end{definition}

It is easy to check that the set of matrices $M$ defined as in \eqref{eq:def_M} and Definition \ref{def:sparse} contains those from Theorem~\ref{thm:delocalization} (see the proof of Theorem~\ref{thm:delocalization} below). From now on we suppose that $K = 1$ to simplify notation.

The local law for the matrix $M$ established in Theorem \ref{thm:local_law} below provides control of the entries of the \emph{Green function} 
\begin{equation} \label{eq:def_G} 
G(z) \deq \big( M - z\big)^{-1}  
\end{equation}
for $z$ in the spectral domain
\begin{equation} \label{eq:def_S_spectral_domain} 
\mathbf{S} \equiv \f S_{\kappa, L, N} = \cal S_\kappa \times [N^{-1 + \kappa}, L]
\end{equation}
for some constant $L \geq 1$.
We also define the Stieltjes transform $g$ of the \emph{empirical spectral measure of $M$} given by 
\begin{equation} \label{def_g}
g(z) \deq \frac{1}{N} \sum_{i =1}^N \frac{1}{\lambda_i(M) - z} = \frac{1}{N} \tr G(z)\,.
\end{equation}

The limiting behaviour of $G$ and $g$ is governed by the following deterministic quantities. Denote by $\C_+ \deq \{z \in \C \col \im z > 0\}$ the complex upper half-plane.
For $z \in \C_+$ we define $m(z)$ as the Stieltjes transform of the semicircle law $\mu_1$,
\begin{equation} \label{def_m}
m(z) \deq \int \frac{\mu_1(\dd u)}{u - z} \,, \qquad \mu_1(\dd u) \deq \frac{1}{2 \pi} \sqrt{(4 - u^2)_+} \, \dd u\,.
\end{equation}
An elementary argument shows that $m(z)$ can be characterized as the unique solution $m$ in $\C_+$ of the equation
\begin{equation} \label{m_quadr}
\frac{1}{m(z)} = -z - m(z)\,. 
\end{equation}
For $\alpha \geq 0$ and $z \in \C_+$ we define
\begin{equation} \label{eq:def_m_alpha} 
m_\alpha(z) \deq - \frac{1}{ z + \alpha m(z)}\,,
\end{equation}
so that $m_1 = m$ by \eqref{m_quadr}.
In Lemma \ref{lem:properties_m_alpha} below we show that $m_\alpha$ is bounded in the domain $\f S$, with a bound depending only on $\kappa$.

For $x \in [N]$ we denote the square Euclidean norm of the $x$th row of $H$ by
\begin{equation} \label{eq:def_beta_x} 
\beta_x \deq \sum_{y} \abs{H_{xy}}^2\,,
\end{equation}
which should be thought of as the normalized degree of $x$; see Remark \ref{rem:alpha_beta} below.

\begin{theorem}[Local law for $M$] \label{thm:local_law} 
Fix $0 < \kappa \leq 1/2$ and $L \geq 1$. 
Let $H$ be a sparse matrix as in Definition~\ref{def:sparse}, define $M$ as in \eqref{eq:def_M} for some $0 \leq f \leq  N^{\kappa/6}$, and define $G$ and $g$ as in \eqref{eq:def_G} and \eqref{def_g} respectively. Then with very high probability, for $d$ satisfying \eqref{d_condition_deloc}, for all $z \in \f S$ we have
\begin{align} \label{eq:local_law_entrywise}
\max_{x,y \in [N]} \absb{G_{xy}(z) - \delta_{xy} m_{\beta_x}(z) } &\leq \Cnu \bigg( \frac{\log N}{d^2} \bigg)^{1/3}\,,
\\
\label{eq:local_law_averaged} 
\absb{g(z) - m(z)} &\leq \Cnu \Bigg( \frac{\log N}{d^2} \bigg)^{1/3}\,.
\end{align}
\end{theorem} 

\begin{proof}[Proof of Theorem~\ref{thm:delocalization}] 
Under the assumptions of Theorem~\ref{thm:delocalization} we find that $M \deq A / \sqrt{d}$ is of the form \eqref{eq:def_M} for some $H$ and $f$ satisfying the assumptions of Theorem \ref{thm:local_law}.
Now Theorem~\ref{thm:delocalization} is a well-known consequence of Theorem \ref{thm:local_law} and the boundedness of $m_{\alpha}(z)$ in \eqref{eq:m_alpha_bounded} below. For the reader's convenience, we give the short proof. Denoting the eigenvalues of $M$ by $(\lambda_i(M))_{i \in [N]}$ and the associated eigenvectors by $(\f w_i(M))_{i \in [N]}$, setting $z = \lambda + \ii \eta$ with $\eta = N^{-1 + \kappa}$, by \eqref{eq:local_law_entrywise} and \eqref{eq:m_alpha_bounded} we have with very high probability
\begin{equation*}
\cal C \geq \im G_{xx}(z) = \sum_{i \in [N]} \frac{\eta \abs{\scalar{\f 1_x}{\f w_i(M)}}^2}{\eta^2 + (\lambda - \lambda_i(M))^2} \geq \frac{1}{\eta} \, \abs{\scalar{\f 1_x}{\f w}}^2\,,
\end{equation*}
where in the last step we omitted all terms except $i$ satisfying $\lambda_i(M) = \lambda$. The claim follows by renaming $\kappa \to \kappa / 2$. (Here we used that Theorem \ref{thm:local_law} holds also for random $z \in \f S$, as follows form a standard net argument; see e.g.\ \cite[Remark 2.7]{BenyachKnowles2017}.)
\end{proof}

\begin{remark}[Relation between $\alpha_x$ and $\beta_x$] \label{rem:alpha_beta}
In the special case $M = d^{-1/2} A$ with $A$ the adjacency matrix of $\bb G(N,d/N)$, we have
\begin{equation*}
\beta_x = \frac{1}{d} \sum_{y} \bigg(A_{xy} - \frac{d}{N}\bigg)^2 = \alpha_x + O \pbb{\frac{d (1 + \alpha_x)}{N}} = \alpha_x + \cal O \pbb{\frac{d + \log N}{N}}
\end{equation*}
with very high probability, by Lemma \ref{lem:upper_bound_degrees}.
\end{remark}

By definition, $m_\alpha(z) \in \C_+$ for $z \in \C_+$, i.e.\ $m_\alpha$ is a Nevanlinna function, and $\lim_{z \to \infty} z m_\alpha(z) = -1$. By the integral representation theorem for Nevanlinna functions, we conclude that $m_\alpha$ is the Stieltjes transform of a Borel probability measure $\mu_\alpha$ on $\R$,
\begin{equation} \label{m_mu_st}
m_\alpha(z) = \int \frac{\mu_\alpha(\dd u)}{u - z}\,.
\end{equation}
Theorem \ref{thm:local_law} implies that the spectral measure of $M$ at a vertex $x$ is approximately $\mu_{\beta_x}$ with very high probability.

Inverting the Stieltjes transform \eqref{m_mu_st} and using the definitions \eqref{def_m} and \eqref{eq:def_m_alpha}, we find after a short calculation
\begin{equation} \label{mu_alpha}
\mu_\alpha(\dd u) = g_\alpha(u) \, \dd u + h_\alpha \delta_{s_\alpha}(\dd u)+ h_\alpha \delta_{-s_\alpha}(\dd u)\,,
\end{equation}
where
\begin{equation*}
g_\alpha(u) \deq \frac{\alpha \ind{\abs{u} < 2}}{2\pi} \frac{\sqrt{4-u^2}}{(1-\alpha)u^2 + \alpha^2}\,,
\qquad
h_\alpha \deq \ind{\alpha > 2} \frac{\alpha - 2}{2 \alpha - 2} + \frac{\ind{\alpha = 0}}{2}\,, \qquad
s_\alpha \deq \ind{\alpha > 2} \Lambda(\alpha)\,.
\end{equation*}
The family $(\mu_\alpha)_{\alpha \geq 0}$ contains the semicircle law ($\alpha = 1$), the Kesten-McKay law of parameter $d$ ($\alpha = d / (d - 1)$), and the arcsine law ($\alpha = 2$). For rational $\alpha = p/q$, the measure $\mu_{p/q}$ can be interpreted as the spectral measure at the root of the infinite rooted $(p,q)$-regular tree, whose root has $p$ children and all other vertices have $q$ children. We refer to Appendix \ref{sec:mu} for more details.
See Figure \ref{fig:mu} for an illustration of the measure $\mu_\alpha$. 

\begin{remark} \label{rem:spectral measure}
Using a standard application the Helffer-Sjöstrand formula (see e.g.\ \cite[Section 8 and Appendix C]{BenyachKnowles2017}), we deduce from Theorem \ref{thm:local_law} the following local law for the spectral measure. Denote by $\varrho_x$ the spectral measure of $M$ at vertex $x$. Under the assumptions of Theorem \ref{thm:local_law}, with very high probability, for any inverval $I \subset \cal S_\kappa$, we have
\begin{equation*}
\varrho_x(I) = \mu_{\beta_x}(I) + \cal O \pbb{\abs{I} \bigg( \frac{\log N}{d^2} \bigg)^{1/3} + N^{\kappa - 1}}\,.
\end{equation*}
The error is smaller than the left-hand side provided that $\abs{I} \geq \cal C N^{\kappa - 1}$.
\end{remark}

\begin{figure}[!ht]
\begin{center}
{\footnotesize 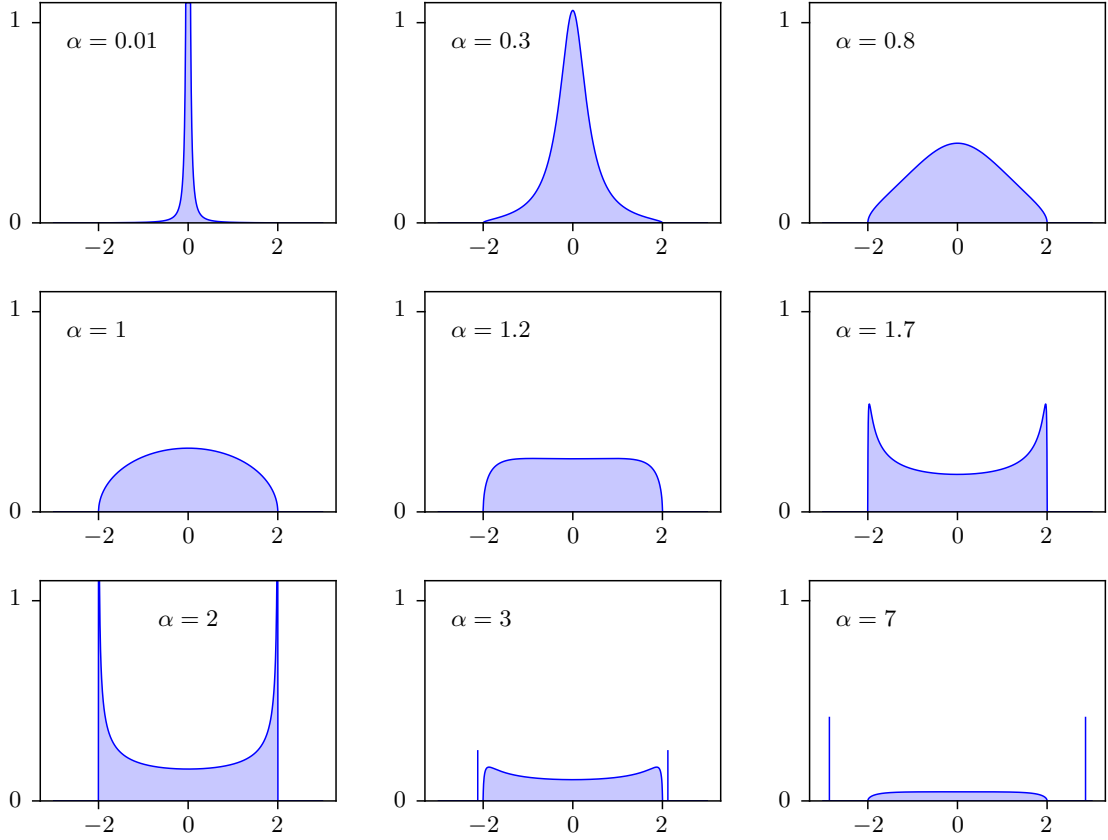}
\end{center}
\caption{An illustration of the probability measure $\mu_\alpha$ for various values of $\alpha$. For $\alpha > 2$, $\mu_\alpha$ has two atoms which we draw using vertical lines. The measure $\mu_\alpha$ is the semicircle law for $\alpha = 1$, the arcsine law for $\alpha = 2$, and the Kesten-McKay law with $d = \frac{\alpha}{\alpha - 1}$ for $1 < \alpha < 2$. Note that the density of $\mu_\alpha$ is bounded in $\cal S_\kappa$, uniformly in $\alpha$. The divergence of the density near $0$ is caused by values of $\alpha$ close to $0$, and the divergence of the density near $\pm 2$ by values of $\alpha$ close to $2$.
 \label{fig:mu}}
\end{figure}

The remainder of this section is devoted to the proof of Theorem~\ref{thm:local_law}. For the rest of this section, we assume that $M$ is as in Theorem \ref{thm:local_law}.
To simplify notation, we consistently omit the $z$-dependence from our notation in quantities that depend on $z \in \f S$.
Unless mentioned otherwise, from now on all statements are uniform in $z \in \f S$.

For the proof of Theorem \ref{thm:local_law}, it will be convenient to single out the generic constant $\cal C$ from \eqref{d_condition_deloc} by introducing a new constant $\cal D$ and replacing \eqref{d_condition_deloc} with
\begin{equation} \label{d_condition_deloc2}
\cal D \sqrt{\log N} \leq d \leq (\log N)^{3/2}\,.
\end{equation}
Our proof will always assume that $\cal C \equiv \cal C_\nu$ and $\cal D \equiv \cal D_\nu$ are large enough, and the constant $\cal C$ in \eqref{d_condition_deloc} can be taken to be $\cal C \vee \cal D$. For the rest of this section we assume that $d$ satisfies \eqref{d_condition_deloc2} for some large enough $\cal D$, depending on $\kappa$ and $\nu$.
To guide the reader through the proof, in Figure \ref{fig:dependencies} we include a diagram of the dependencies of the various quantities appearing throughout this section.

\begin{figure}[!ht]
\begin{center}
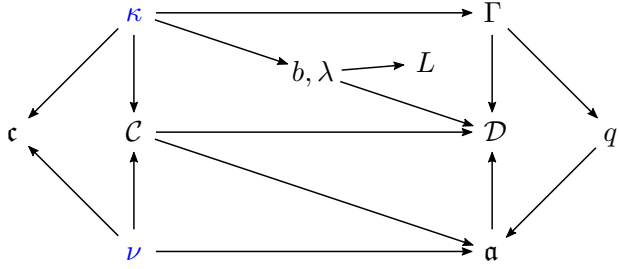
\end{center}
\caption{The dependency graph of the various quantities appearing in the proof of Theorem \ref{thm:local_law}. An arrow from $x$ to $y$ means that $y$ is chosen as a function of $x$. The independent parameters, $\kappa$ and $\nu$, are highlighted in blue.
 \label{fig:dependencies}}
\end{figure}

\subsection{Typical vertices} \label{subsec:typical_vertices}

We start by introducing the key tool in the proof of Theorem \ref{thm:local_law}, a decomposition of vertices into \emph{typical vertices} and the complementary \emph{atypical vertices}. Heuristically, a typical vertex $x$ has close to $d$ neighbours and the spectral measure of $M$ at $x$ is well approximated by the semicircle law. In fact, in order to be applicable to the proof of Proposition \ref{lem:bootstrapping_step} below, the notion of a typical vertex is somewhat more complicated, and when counting the number of neighbours of a vertex $x$ we also need to weight the neighbours with diagonal entries of a Green function, so that the notion of typical vertex also depends on the spectral parameter $z$, which in this subsection we allow to be any complex number $z$ with $\im z \geq N^{-1 + \kappa}$. This notion is defined precisely using the parameters $\Phi_x$ and $\Psi_x$ from  \eqref{eq:def_Phi_Psi} below. The main result of this subsection is Proposition~\ref{pro:typcial_vertices} below, which states, in the language of graphs when $M = d^{-1/2} A$ with $A$ the adjacency matrix of $\bb G(N,d/N)$, that most vertices are typical and most neighbours of any vertex are typical. To state it, we introduce some notation.

\begin{definition} \label{def:minors}
For any subset $T \subset [N]$, we define the minor $M^{(T)}$ with indices in $T$ as the $(N-\abs{T}) \times (N-\abs{T})$-matrix 
\begin{equation} \label{eq:def_M_T} 
 M^{(T)} \deq (M_{xy})_{x,y \in [N] \setminus T}. 
\end{equation}
If $T$ consists only of one or two elements, $T = \{x\}$ or $T=\{x,y\}$, then we abbreviate 
$M^{(x)}$ and $M^{(xy)}$ for $M^{(\{x\})}$ and $M^{(\{x,y\})}$. 
We also abbreviate $M^{(Tx)}$ for $M^{(T \cup\{ x\})}$. 
The Green function of $M^{(T)}$ is denoted by 
\begin{equation} \label{eq:def_G_T} 
 G^{(T)}(z) \deq (M^{(T)} - z)^{-1}. 
\end{equation}
We use the notation
\begin{equation} \label{eq:convention_summation} 
 \sum_{x}^{(T)} \deq \sum_{x \in [N]\setminus T} . 
\end{equation} 
\end{definition}

\begin{definition}[Typical vertices] \label{def_typical_vert}
Let $\fa > 0$ be a constant, and define the set of \emph{typical vertices}
\begin{equation} \label{eq:def_cal_T} 
\cal T_\fa \deq \h{x \in [N] \col \abs{\Phi_x} \vee \abs{\Psi_x} \leq \varphi_\fa }\,, \qquad \varphi_\fa \deq \fa \bigg( \frac{\log N}{d^2} \bigg)^{1/3}\,,
\end{equation}
where
\begin{equation} \label{eq:def_Phi_Psi} 
\Phi_x \deq \sum_{y}^{(x)} \pbb{\abs{H_{xy}}^2 - \frac{1}{N}}\,, \qquad
\Psi_x \deq \sum_{y}^{(x)} \pbb{\abs{H_{xy}}^2 - \frac{1}{N}} G_{yy}^{(x)}\,.
\end{equation}
\end{definition}

Note that this notion depends on the spectral parameter $z$, i.e.\ $\cal T_\fa \equiv \cal T_\fa(z)$.
The constant $\fa$ will depend only on $\nu$ and $\kappa$. It will be fixed in \eqref{def_a} below. The constant $\cal D \geq \fa^{3/2}$ from \eqref{d_condition_deloc2} is always chosen large enough so that $\varphi_{\fa} \leq 1$.

The following proposition holds on the event $\{\theta = 1\}$, where we introduce the indicator function 
\begin{equation} \label{eq:def_xi} 
\theta \deq \ind{\max_{x,y} \abs{G_{xy}} \leq \Gamma}
\end{equation}
depending on some deterministic constant $\Gamma \geq 1$. In \eqref{eq:G_xy_bounded_Gamma} below, we shall choose a constant $\Gamma \equiv \Gamma_\kappa$, depending 
only on $\kappa$, such that the condition $\theta = 1$ can be justified 
by a bootstrapping argument along the proof of Theorem~\ref{thm:local_law} in Section~\ref{subsec:proof_local_law} below.

Throughout the sequel we use the following generalization of Definition \ref{def:very_high_probability_def}.

\begin{definition}
An event $\Xi$ \emph{holds with very high probability on an event} $\Omega$ if for all $\nu > 0$ there exists $\cal C > 0$ such that $\P(\Xi \cap \Omega) \geq \P(\Omega) - \cal C N^{-\nu}$ for all $N \in \N$.
\end{definition}

We now state the main result of this subsection.

\begin{proposition} \label{pro:typcial_vertices} 
There are constants $0 < q \leq 1$, depending only on $\Gamma$, and $\fa > 0$, depending only on $\nu$ and $q$, such that, on the event $\{\theta = 1\}$, the following holds with very high probability.
\begin{enumerate}[label=(\roman*)]
\item \label{item:a1} Most vertices are typical: 
\[ \abs{\cal T_\fa^c} \leq \exp(  q \varphi_\fa^2 d ) + N \exp ( - 2 q \varphi_\fa^2 d). \] 
\item \label{item:a2}
Most neighbours of any vertex are typical:
\[ \sum_{y \in \cal T_\fa^c}^{(x)}\abs{H_{xy}}^2 \leq \cal C \varphi_\fa + \Cnu d^4 \exp(- q \varphi_\fa^2 d ) \] 
uniformly for $x \in [N]$. 
\end{enumerate}
\end{proposition} 

For the interpretation of Proposition~\ref{pro:typcial_vertices} \ref{item:a2}, one should think of the motivating example $M = d^{-1/2} A$, for which $d \sum_{y \in \cal T^c_\fa}^{(x)}\abs{H_{xy}}^2$ is the number of atypical neighbours of $x$, up to an error term $\cal O\pb{\frac{d^2 + d \log N}{N}}$ by Remark \ref{rem:alpha_beta}.

The remainder of Section~\ref{subsec:typical_vertices} is devoted to the proof of Proposition~\ref{pro:typcial_vertices}. 
We need the following version of $\cal T_\fa$ defined in terms of $H^{(T)}$ instead of $H$.

\begin{definition} \label{def:psi_T}
For any $x \in [N]$ and $T \subset [N]$, we define
\[ \Phi_x^{(T)} \deq \sum_{y}^{(Tx)} \pbb{\abs{H_{xy}}^2 - \frac{1}{N}}\,, \qquad
\Psi_x^{(T)} \deq \sum_{y}^{(Tx)} \pbb{\abs{H_{xy}}^2 - \frac{1}{N}} G_{yy}^{(Tx)}\, \] 
and
\begin{equation*}
\cal T^{(T)} _\fa\deq \h{x \in [N] \setminus T \col \abs{\Phi_x^{(T)}} \vee \abs{\Psi_x^{(T)}} \leq \varphi_\fa}\,.
\end{equation*}
\end{definition}

Note that $\Phi_x^{(\emptyset)} = \Phi_x$ and $\Psi_x^{(\emptyset)} = \Psi_x$ with the definitions from \eqref{eq:def_Phi_Psi}, and hence $\cal T_\fa^{(\emptyset)} = \cal T_\fa$.
The proof of Proposition~\ref{pro:typcial_vertices} relies on the two following lemmas. 

\begin{lemma} \label{lem:XcapTc}
There are constants $0 < q \leq 1$, depending only on $\Gamma$, and $\fa > 0$, depending only on $\nu$ and $q$, such that, for any deterministic $X \subset [N]$, the following holds with very high probability on the event $\{\theta = 1\}$.
\begin{enumerate}[label=(\roman*)]
\item \label{item:XcapTc_general} 
$\abs{X \cap \cal T_{\fa/2}^c} \leq \exp( q \varphi_\fa^2 d) + \abs{X} \exp(- 2 q \varphi_\fa^2 d)$.
\item \label{item:XcapTc_X_small} 
If $\abs{X} \leq \exp ( 2 q  \varphi_\fa^2 d)$ then $\abs{X \cap \cal T_{\fa/2}^c} \leq \cal \varphi_\fa d$. 
\end{enumerate}
For any deterministic $x \in [N]$, the same estimates hold for $\big(\cal T^{(x)}_{\fa / 2}\big)^c$ instead of $\cal T^c_{\fa/2}$ and a random set $X \subset [N] \setminus \{x\}$ that is independent of $H^{(x)}$. 
\end{lemma}

\begin{lemma} \label{lem:phi_phi_x_psi_psi_x} 
With very high probability, for any constant $\fa > 0$ we have
\[ \theta \abs{\Phi_y - \Phi_y^{(x)}} \leq \varphi_{\fa/2},  \qquad \theta \abs{\Psi_y - \Psi_y^{(x)}} \leq \varphi_{\fa / 2} \] 
for all $x,y \in [N]$.
\end{lemma} 

Before proving Lemmas~\ref{lem:XcapTc} and \ref{lem:phi_phi_x_psi_psi_x}, we use them to establish Proposition~\ref{pro:typcial_vertices}.

\begin{proof}[Proof of Proposition~\ref{pro:typcial_vertices}] 
For~\ref{item:a1}, we choose $X = [N]$ in Lemma \ref{lem:XcapTc} \ref{item:XcapTc_general}, using that $\cal T_{\fa/2} \subset \cal T_\fa$.

We now turn to the proof of \ref{item:a2}.
By Lemma~\ref{lem:phi_phi_x_psi_psi_x}, on the event $\{\theta = 1\}$ we have 
$\cal T^c_\fa \subset \big(\cal T^{(x)}_{\fa/2}\big)^c$
with very high probability and hence
\begin{equation*}
\theta \sum^{(x)}_{y \in \cal T_\fa^c} \abs{H_{xy}}^2 \leq \theta \sum^{(x)}_{y \in (\cal T_{\fa / 2}^{(x)})^c} \abs{H_{xy}}^2
\end{equation*}
with very high probability.
Since $\abs{H_{xy}}^2 \leq 1 / d$ almost surely, we obtain the decomposition
\begin{equation} \label{eq:sum_a_y_in_Tc} 
\begin{aligned}
\sum^{(x)}_{y \in (\cal T_{\fa / 2}^{(x)})^c} \abs{H_{xy}}^2 &\leq  \sum_{k = 0}^{\log N} \sum^{(x)}_{y \in (\cal T_{\fa / 2}^{(x)})^c} \abs{H_{xy}}^2 \ind{d^{-k-2} \leq \abs{H_{xy}}^2 \leq d^{-k - 1}} + \frac{1}{N}
\\
&\leq \sum_{k = 0}^{\log N} \sum^{(x)}_{y \in (\cal T_{\fa / 2}^{(x)})^c} d^{-k - 1} \ind{\abs{H_{xy}}^2 \geq d^{-k-2}} + \frac{1}{N}
\\
&= \sum_{k = 0}^{\log N} d^{-k - 1} \abs{X_k \cap (\cal T_{\fa / 2}^{(x)})^c} + \frac{1}{N}\,,
\end{aligned}
\end{equation}
where we defined
\begin{equation*}
X_k \deq \hb{y \neq x \col \abs{H_{xy}}^2 \geq d^{-k - 2}}\,.
\end{equation*}
Since $\sum^{(x)}_y \abs{H_{xy}}^2 \leq \cal C d$ with very high probability by Definition \ref{def:sparse} and Bennett's inequality, we conclude that
\begin{equation} \label{X_k_estimate}
\abs{X_k} \leq \cal C d^{k + 3}
\end{equation}
with very high probability.

We shall apply Lemma \ref{lem:XcapTc} to the sets $X = X_k$ and $(\cal T_{\fa / 2}^{(x)})^c$. To that end, note that $X_k \subset [N] \setminus \{x\}$ is a measurable function of the family $(H_{xy})_{y \in [N]}$, and hence independent of $H^{(x)}$.
Thus, we may apply Lemma \ref{lem:XcapTc}.

We define $K \deq \max \hb{k \geq 0 \col \cal C d^{k + 3} \leq \ee^{ 2q\varphi_\fa^2 d}}$ and decompose the sum on the right-hand side of \eqref{eq:sum_a_y_in_Tc} into 
\begin{align*}
\sum_{k = 0}^{\log N} d^{-k - 1} \abs{X_k \cap \big(\cal T_{\fa / 2}^{(x)}\big)^c} &= \sum_{k = 0}^{K} d^{-k - 1} \abs{X_k \cap \big(\cal T_{\fa / 2}^{(x)}\big)^c} + \sum_{k = K+1}^{\log N} d^{-k - 1} \abs{X_k \cap \big(\cal T_{\fa / 2}^{(x)}\big)^c}
\\
&\leq \sum_{k = 0}^{K} d^{-k - 1} \varphi_\fa d   + \sum_{k = K+1}^{\log N} d^{-k - 1} \pb{\ee^{q \varphi_\fa^2 d} + \cal C d^{k+3} \ee^{- 2q \varphi_\fa^2 d}}
\\
&\leq 2 \varphi_\fa + \cal C d^2 \ee^{-q \varphi_\fa^2 d}\log N\, 
\end{align*}
with very high probability.
Here, we used Lemma \ref{lem:XcapTc} \ref{item:XcapTc_X_small} to estimate the summands if $k \leq K$ and Lemma~\ref{lem:XcapTc} \ref{item:XcapTc_general} and \eqref{X_k_estimate} for the other summands. Since $\log N \leq d^2$, this concludes the proof of \ref{item:a2}.
\end{proof}

The rest of this subsection is devoted to the proofs of Lemmas~\ref{lem:XcapTc} and \ref{lem:phi_phi_x_psi_psi_x}.
Let $\theta$ be defined as in \eqref{eq:def_xi} for some constant $\Gamma \geq 1$. 
For any subset $T \subset[N]$, we define the indicator function 
\begin{equation*}
\theta^{(T)} \deq \ind{\max_{a,b \notin T} \abs{G_{ab}^{(T)}} \leq 2 \Gamma}\,.
\end{equation*}
Lemma~\ref{lem:XcapTc} is a direct consequence of the following two lemmas. 

The first one, Lemma \ref{lem:decoupling}, is mainly a decoupling argument for the random variables $(\Psi_x)_{x \in [N]}$. Indeed, the probability that any fixed vertex $x$ is atypical is only small, $o(1)$, and not very small, $N^{-\nu}$; see \eqref{eq:bound_Phi_Psi_x} below. If the events of different vertices being atypical were independent, we could deduce that the probability that a sufficiently large set of vertices are atypical is very small. However, these events are not independent. The most serious breach of independence arises from the Green function $G^{(x)}_{yy}$ in the definition of $\Psi_x$. In order to make this argument work, we have to replace the parameters $\Phi_x$ and $\Psi_x$ with their \emph{decoupled} versions $\Phi_x^{(T)}$ and $\Psi_x^{(T)}$ from Definition \ref{def:psi_T}. To that end, we have to estimate the error involved, $\abs{\Phi_x - \Phi_x^{(T)}}$ and $\abs{\Psi_x - \Psi_x^{(T)}}$. Unfortunately the error bound on the latter is proportional to $\beta_x$ (see \eqref{eq:bound_Phi_Psi_difference}), which is not affordable for vertices of large degree. The solution to this issue involves the observation that if $\beta_x$ is too large then the vertex is atypical by the condition on $\Phi_x$, which allows us to disregard the size of $\Psi_x$. The details are given in the proof of Lemma \ref{lem:decoupling} below.

The second one, Lemma \ref{lem:xi_leq_xi_T}, gives a priori bounds on the entries of the Green function $G^{(T)}$, which shows that if the entries of $G$ are bounded then so are those of $G^{(T)}$ for $\abs{T} = o(d)$. For $T$ of fixed size, this fact is a standard application of the resolvent identities from Lemma \ref{lem:resolvent_expansions}. For our purposes, it is crucial that $T$ can have size up to $o(d)$, and such a quantitative estimate requires slightly more care.

\begin{lemma} \label{lem:decoupling} 
There is a constant $0 < q \leq 1$, depending only on $\Gamma$, such that, for any $\nu>0$, there is 
$\Cnu >0$ such that the following holds for any fixed $\fa > 0$.
If $x \notin T \subset [N]$ are deterministic with $\abs{T} \leq \varphi_\fa d /\Cnu$ then 
\begin{subequations} 
\begin{align} 
\P \big( T \subset \cal T_{\fa / 2}^c,\, \theta = 1 \big) & \leq \ee^{- 4 q \varphi_\fa^2 d \abs{T}} + \Cnu N^{-\nu} , 
\label{eq:bound_decoupling} \\ 
\P \big( T \subset \big(\cal T_{\fa /2}^{(x)}\big)^c, \theta^{(x)} =1 \big) & 
\leq  \ee^{-4 q \varphi_\fa^2 d \abs{T}} + \Cnu N^{-\nu}\,.\label{eq:bound_decoupling2} 
\end{align} 
\end{subequations} 
\end{lemma}

\begin{lemma} \label{lem:xi_leq_xi_T} 
For any subset $T \subset [N]$ satisfying $\abs{T} \leq \frac{d}{\cal C \Gamma^2}$ we have $\theta \leq \theta^{(T)}$ with very high probability. 
\end{lemma}

Before proving Lemma~\ref{lem:decoupling} and Lemma~\ref{lem:xi_leq_xi_T}, we use them to show Lemma~\ref{lem:XcapTc}. 

\begin{proof}[Proof of Lemma~\ref{lem:XcapTc}] 
Throughout the proof we abbreviate $\P_\theta(\Xi) \deq \P (\Xi \cap \{ \theta = 1\})$. Let $\cal C$ be the constant from Lemma \ref{lem:decoupling}, and set
\begin{equation} \label{def_a}
\fa \deq \pbb{\frac{\cal C \nu}{4 q}}^{1/3}\,.
\end{equation}

For the proof of \ref{item:XcapTc_X_small}, we choose $k = \varphi_\fa d /\Cnu$
and estimate
\begin{multline*}
\P_\theta(\abs{X \cap \cal T_{\fa / 2}^c} \geq k) \leq \sum_{Y \subset X : \abs{Y} = k} \P_\theta(Y \subset \cal T_{\fa/2}^c)
\leq \binom{\abs{X}}{k} \Big( \ee^{-4 q \varphi_\fa^2 d k} + \Cnu N^{-\nu} \Big)
\\
 \leq  \big(\abs{X} \ee^{- 4 q \varphi_\fa^2 d}\big)^k + \Cnu \abs{X}^k N^{-\nu}
 \leq  \ee^{- 2 q \varphi_\fa^2 d k} + \Cnu \ee^{2 q \varphi_\fa^2 d k} N^{-\nu} = N^{-2q\fa^3/\cal C} + \cal C N^{2q\fa^3/\cal C - \nu}\,.
\end{multline*}
where in the second step we used \eqref{eq:bound_decoupling}. Thus, by our choice of $\fa$, we have $\P_\theta(\abs{X \cap \cal T_{\fa / 2}^c} \geq k) \leq (\cal C + 1) N^{-\nu/2}$,
from which \ref{item:XcapTc_X_small} follows after renaming $\nu$ and $\cal C$.

To prove \ref{item:XcapTc_general} we estimate, for $t>0$ and $l \in \N$,
\begin{equation*}
\P_\theta(\abs{X \cap \cal T_{\fa / 2}^c} \geq t) \leq \frac{1}{t^l} \E \pBB{\sum_{x \in X} \ind{x \in \cal T_{\fa / 2}^c}\theta }^l = \frac{1}{t^l} \sum_{x_1, \dots, x_l \in X} \P_\theta(x_1 \in \cal T_{\fa / 2}^c, \dots, x_l \in \cal T_{\fa / 2}^c)\,.
\end{equation*}
Choosing $l = \varphi_\fa d/\Cnu$, regrouping the summation according to the partition of coincidences, and using Lemma~\ref{lem:decoupling} yield
\begin{multline*}
\P_\theta (\abs{X \cap \cal T_{\fa / 2}^c} \geq t) \leq \frac{1}{t^l} \sum_{\pi \in \fra P_l} \abs{X}^{\abs{\pi}} \big( \ee^{- 4 q \varphi_\fa^2 d \abs{\pi}} + \Cnu N^{- \nu} \big) 
\\
\leq \frac{1}{t^l} \sum_{k = 0}^l \binom{l}{k} l^{l - k} \abs{X}^k \big ( \ee^{-4 q \varphi_\fa^2 dk} + \Cnu N^{- \nu} \big) 
=  \frac{(l + \abs{X} \ee^{- 4 q \varphi_\fa^2 d})^l + \Cnu N^{-\nu} (l + \abs{X})^l}{t^l}\,.
\end{multline*}
Here, $\fra P_l$ denotes the set of partitions of $[l]$, and we denote by $k = \abs{\pi}$ the number of blocks in the partition $\pi \in \fra P_l$. We also used that the number of partitions of $l$ elements consisting of $k$ blocks is bounded by $\binom{l}{k} l^{l - k}$. The last step follows from the binomial theorem. 
Therefore, using $l = \varphi_\fa d/\cal C$ and choosing $t = \ee^{q \varphi_\fa^2 d} + \abs{X} \ee^{- 2 q \varphi_\fa^2 d}$ as well as $\Cnu$ and $\nu$ 
sufficiently large imply the bound in Lemma~\ref{lem:XcapTc} \ref{item:XcapTc_general} with very high probability, after renaming $\Cnu$ and $\nu$. Here we used \eqref{d_condition_deloc2}.  

To obtain the same statements for $\cal T_{\fa / 2}^{(x)}$ instead of $\cal T_{\fa / 2}$, we estimate 
\[ \P_\theta  \Big( \abs{X \cap (\cal T_{\fa / 2}^{(x)})^c} \geq t\Big)  \leq \E \Big[ \P \pB{ \abs{X \cap (\cal T_{\fa / 2}^{(x)})^c} \geq t, \theta^{(x)} = 1 \Big\vert X } \Big] + \P \big( \theta^{(x)} = 0 , \theta = 1\big). \] 
For both parts, \ref{item:XcapTc_general} and \ref{item:XcapTc_X_small}, the conditional probability 
$\P \pb{ \abs{X \cap (\cal T_{\fa / 2}^{(x)})^c} \geq t, \theta^{(x)} = 1 \big\vert X }$ can be bounded as before 
using \eqref{eq:bound_decoupling2} instead of \eqref{eq:bound_decoupling} since, by assumption on $X$, the set $\cal T_{\fa / 2}^{(x)}$ and the indicator function $\theta^{(x)}$ are independent of $X$. 
The smallness of $\P(\theta^{(x)} = 0, \theta = 1) \leq \P(\theta^{(x)} < \theta)$ is a consequence of Lemma \ref{lem:xi_leq_xi_T}.
This concludes the proof of Lemma~\ref{lem:XcapTc}. 
\end{proof}

The rest of this subsection is devoted to the proofs of Lemmas~\ref{lem:phi_phi_x_psi_psi_x}, \ref{lem:decoupling}, and \ref{lem:xi_leq_xi_T}.

\begin{lemma} \label{lem:bound_Gamma_k} 
There is $\fc \equiv \fc_\nu>0$, depending on $\nu$ and $\kappa$, such that for any deterministic $T \subset [N]$ satisfying $\abs{T} \leq \fc d / \Gamma^2$ we have with very high probability
\begin{equation} \label{eq:boundedness_resolvent_minor} 
 \theta \max_{x,y \notin T} \absb{G_{xy}^{(T)}} \leq 2 \Gamma\,.
\end{equation}
Moreover, under the same assumptions on $T$ and for any $u \in [N] \setminus T$, we have 
\begin{equation}  \label{eq:smallness_diff_minor} 
 \theta \max_{x,y \notin T \cup \{u\}} \absb{G_{xy}^{(Tu)} - G_{xy}^{(T)}} \leq \Cnu d^{-1} 
\end{equation}
with very high probability. 
\end{lemma}

Before proving Lemma \ref{lem:bound_Gamma_k}, we use it to conclude the proof of Lemma~\ref{lem:xi_leq_xi_T}.
\begin{proof}[Proof of Lemma~\ref{lem:xi_leq_xi_T}]
The bound in \eqref{eq:boundedness_resolvent_minor} of Lemma~\ref{lem:bound_Gamma_k} implies that
$\theta = \theta \theta^{(T)}$
with very high probability. Since $\theta \leq 1$, the proof is complete.
\end{proof}

\begin{proof}[Proof of Lemma \ref{lem:bound_Gamma_k}]
Throughout the proof we work on the event $\{\theta = 1\}$ exclusively. After a relabelling of the vertices $[N]$, we can suppose that $T = [k]$ with $k \leq cd/\Gamma^2$.
For $k \in [N]$, we set 
\[ \Gamma_k \deq 1 \vee \max_{x,y \notin [k]} \abs{G_{xy}^{([k])}}\,.\]
Note that $\Gamma_0 \leq \Gamma$ by definition of $\theta$.

We now show by induction on $k$ that there is $\Cnu>0$ such that
\begin{equation} \label{eq:proof_bound_Gamma_k_induction} 
\Gamma_k \leq \Gamma_0 \bigg(1 + \frac{16 \Cnu \Gamma^2}{d} \bigg)^k  
\end{equation}
for all $k \in \N$ satisfying $k \leq \frac{d}{32 \, \Cnu \Gamma^2}$. Since $1 + x \leq \ee^x$, \eqref{eq:proof_bound_Gamma_k_induction} implies that $\Gamma_k \leq \ee^{1/2} \Gamma_0 \leq 2 \Gamma$. This directly implies \eqref{eq:boundedness_resolvent_minor} by the definition of $\theta$. 

The initial step with $k = 0$ is trivially correct.
For the induction step $k \to k+1$, we set $T = [k]$ and $u = k + 1$. The algebraic starting point for the induction step is the identities \eqref{eq:expansion_G_xy_Tu1} and \eqref{eq:expansion_G_xy_Tu2}. We shall need the following two estimates.
First, from Lemma \ref{lem:Ward} and Cauchy--Schwarz, we get
\begin{equation} \label{eq:proof_bound_Gamma_k_aux4} 
 \frac{f}{N}\absbb{G_{uy}^{(T)} \sum_a^{(Tu)} G_{xa}^{(Tu)}} \leq \frac{f}{N} \Gamma_k \sqrt{\frac{N}{\im z}} \Gamma_{k+1} \leq N^{-\kappa/3} \Gamma_k \Gamma_{k+1}\,,
\end{equation}
where we used that $\Gamma_{k+1} \geq 1$, $f \leq N^{\kappa/6}$, and $\im z  \geq N^{-1 + \kappa}$.
Second, the first estimate of \eqref{eq:LDB_wvhp} in Corollary~\ref{cor:large_deviation_very_high_probabilty} with $\psi = \Gamma_{k+1}/\sqrt{d}$ and $\gamma = \sqrt{\Gamma_{k+1}/(N \im z)}$, Lemma \ref{lem:Ward}, and $\Gamma_{k+1} \geq 1$ imply 
\begin{equation} \label{eq:proof_bound_Gamma_k_aux5} 
\absbb{ \sum_{a}^{(Tu)} G_{xa}^{(Tu)} H_{au} }  \leq \frac{\Cnu}{\sqrt{d}} \Gamma_{k+1} 
\end{equation}
with very high probability. 

Hence, owing to \eqref{eq:expansion_G_xy_Tu1} and \eqref{eq:expansion_G_xy_Tu2} with $T = [k]$ and $u = k + 1$, we get, respectively,
\begin{equation} \label{eq:proof_bound_Gamma_k_aux1} 
 \Gamma_{k+1} \leq \Gamma_k + \frac{\Cnu}{\sqrt{d}}  \Gamma_k \Gamma_{k+1}, \qquad  \qquad \Gamma_{k+1} \leq \Gamma_k + \frac{\Cnu}{d} \Gamma_k \Gamma_{k+1}^2 
\end{equation}
with very high probability.

By the induction assumption \eqref{eq:proof_bound_Gamma_k_induction} we have $\cal C \Gamma_k / \sqrt{d} \leq 2 \cal C \Gamma /  \sqrt{d} \leq 1/2$, so that the first inequality in \eqref{eq:proof_bound_Gamma_k_aux1} implies the rough a priori bound
\begin{equation} \label{eq:proof_bound_Gamma_k_aux2} 
 \Gamma_{k+1} \leq 2 \Gamma_k
\end{equation}
with very high probability.
From the second inequality in \eqref{eq:proof_bound_Gamma_k_aux1} and \eqref{eq:proof_bound_Gamma_k_aux2}, we deduce that
\begin{equation*}
\Gamma_{k+1} \leq \Gamma_k \pbb{1 + \frac{4 \cal C}{d} \Gamma_k^2} \leq \Gamma_k \pbb{1 + \frac{16 \cal C \Gamma^2}{d}}\,,
\end{equation*}
where in the second step we used $\Gamma_k \leq 2 \Gamma$, by the induction assumption \eqref{eq:proof_bound_Gamma_k_induction}. This concludes the proof of \eqref{eq:proof_bound_Gamma_k_induction}, and, hence, of \eqref{eq:boundedness_resolvent_minor}.

For the proof of \eqref{eq:smallness_diff_minor}, we start from \eqref{eq:expansion_G_xy_Tu2} and use \eqref{eq:proof_bound_Gamma_k_aux4}, \eqref{eq:proof_bound_Gamma_k_aux5} as well as \eqref{eq:boundedness_resolvent_minor}. 
This concludes the proof of Lemma~\ref{lem:bound_Gamma_k}. 
\end{proof}

The next result provides concentration estimates for the parameters $\Phi_x$ and $\Psi_x$.

\begin{lemma}
\label{lem:prob_typicality_vertex}
There is a constant $0 < q \leq 1$, depending only on $\Gamma$, such that the following holds.
Let $\fc>0$ be as in Lemma~\ref{lem:bound_Gamma_k}, and let $x \in [N]$ and $T \subset [N]$ be deterministic and satisfy $\abs{T} \leq \fc d / \Gamma^2$. Then for any $0 < \epsilon \leq 1$ we have
\begin{equation}\label{eq:bound_Phi_Psi_x}
\theta^{(T)}  \P \big( \abs{\Phi_x^{(T)}} > \eps \bigm\vert H^{(T)} \big) \leq \ee^{-  32 q \eps^2 d}\,,  \qquad
\theta^{(T)} \P \big(\abs{\Psi_x^{(T)}} > \eps \bigm\vert H^{(T)} \big)  \leq \ee^{ -  32 q \eps^2 d}\,,  
\end{equation}
and, for any $u \notin T$,
\begin{equation} \label{eq:bound_Phi_Psi_difference}
\Phi_x^{(Tu)} - \Phi_x^{(T)}  = O\pbb{\frac{1}{d}}\,,
\qquad 
\theta^{(T)} \pb{\Psi_x^{(Tu)} - \Psi_x^{(T)}}  = \cal O\pbb{\frac{1 + \beta_x}{d}}
\end{equation}
with very high probability. 
\end{lemma} 

Before proving Lemma~\ref{lem:prob_typicality_vertex}, we use it conclude the proof of Lemma~\ref{lem:phi_phi_x_psi_psi_x}. 

\begin{proof}[Proof of Lemma~\ref{lem:phi_phi_x_psi_psi_x}] 
Using \eqref{eq:LDB_linear_noncentered}, we find that $\beta_x \leq \cal C (1 + \frac{\log N}{d})$ with very high probability. The claim now follows from \eqref{eq:bound_Phi_Psi_difference} with $T = \emptyset$ and the definition of $\varphi_{\fa}$, choosing the constant $\cal D$ in \eqref{d_condition_deloc2} large enough.
\end{proof}

\begin{proof}[Proof of Lemma~\ref{lem:prob_typicality_vertex}]

Set $q \deq \frac{1}{2^{11}(\ee \Gamma)^2}$.
We get, using \eqref{eq:LDB_linear_noncentered} with $r \deq 32 q \epsilon^2 d \leq d$, $\E \abs{H_{xy}}^2 = 1/N$, and Chebyshev's inequality,
\begin{multline*}
\theta^{(T)} \P\Big( \abs{\Psi_x^{(T)}} > \eps \Bigm\vert H^{(T)} \Big)
= \P \Bigg(\theta^{(T)} \absBB{\sum_y^{(Tx)} (\abs{H_{xy}}^2 - \E \abs{H_{xy}}^2) G_{yy}^{(T)}} > \eps \biggm\vert H^{(T)} \Bigg)
\\
\leq \pbb{\frac{8 \Gamma}{\epsilon} \sqrt{\frac{r}{d}}}^r
= \ee^{ -32 q \eps^2 d}
\end{multline*}
with very high probability for any $0 < \epsilon \leq 1$. This proves the estimate on $\Psi_x^{(T)}$ in \eqref{eq:bound_Phi_Psi_x}, and the estimate for $\Phi_x^{(T)}$ is proved similarly.

We now turn to the proof of \eqref{eq:bound_Phi_Psi_difference}. If $x = u$ then the statement is trivial. Thus, we assume $x \neq u$. In this case we have
\begin{equation}
\Phi_x^{(Tu)} - \Phi_x^{(T)} = - \bigg(\abs{H_{xu}}^2 - \frac{1}{N} \bigg)
\end{equation}
and the claim for $\Phi$ follows by Definition \ref{def:sparse}. Next,
\[ \Psi_x^{(Tu)} - \Psi_x^{(T)} =  \sum_{y}^{(Tux)} \bigg( \abs{H_{xy}}^2 - \frac{1}{N} \bigg) \Big( G_{yy}^{(Tux)} - G_{yy}^{(Tx)} \Big) - \bigg(\abs{H_{xu}}^2 - \frac{1}{N} \bigg) G_{uu}^{(Tx)}\,. \]
The last term multiplied by $\theta^{(T)}$ is estimated by $O(\Gamma / d)$ since $\theta^{(T)} \abs{G_{uu}^{(Tx)}} \leq 4 \Gamma$ by \eqref{eq:proof_bound_Gamma_k_aux2}. We estimate the first term using \eqref{eq:smallness_diff_minor} in Lemma~\ref{lem:bound_Gamma_k}, which yields
\begin{equation*}
\theta^{(T)} \absb{\Psi_x^{(Tu)} - \Psi_x^{(T)}} \leq \sum_{y}^{(Tux)} \abs{H_{xy}}^2 \frac{\cal C}{d} + \frac{1}{N} \sum_{y}^{(Tux)} \frac{\cal C}{d} + O \pbb{ \frac{\Gamma}{d}} = \cal O \pbb{\frac{1 + \beta_x}{d}}
\end{equation*}
with very high probability.
This concludes the proof of Lemma~\ref{lem:prob_typicality_vertex}.  
\end{proof}

\begin{proof}[Proof of Lemma~\ref{lem:decoupling}]
Throughout the proof we abbreviate $\P_\theta(\Xi) \deq \P (\Xi \cap \{ \theta = 1\})$. We have
\begin{equation*}
\P \big( T \subset \cal T_{\fa / 2}^c,\, \theta = 1 \big) = \P_\theta \pBB{\bigcap_{x \in T}\Omega_x}\,,
\end{equation*}
where we defined the event
\begin{equation*}
\Omega_x \deq \hb{\abs{\Phi_x} > \varphi_{\fa / 2}} \cup \hb{\abs{\Psi_x} > \varphi_{\fa / 2}} = \hb{\abs{\Phi_x} > \varphi_{\fa / 2}} \cup \hb{\abs{\Phi_x} \leq \varphi_{\fa / 2}, \abs{\Psi_x} > \varphi_{\fa / 2}}\,.
\end{equation*}
We have the inclusions
\begin{align*}
\hb{\abs{\Phi_x} > \varphi_{\fa / 2}} &\subset \hb{\abs{\Phi_x^{(T)}} > \varphi_{\fa / 4}} \cup \hb{\abs{\Phi_x - \Phi_x^{(T)}} > \varphi_{\fa / 4}}\,,
\\
\hb{\abs{\Phi_x} \leq \varphi_{\fa / 2}, \abs{\Psi_x} > \varphi_{\fa / 2}} &\subset \hb{\abs{\Psi_x^{(T)}} > \varphi_{\fa / 4}} \cup \hb{\abs{\Phi_x} \leq \varphi_{\fa / 2}, \abs{\Psi_x - \Psi_x^{(T)}} > \varphi_{\fa / 4}}\,.
\end{align*}
Defining the event
\begin{equation*}
\Omega_x^{(T)} \deq \hb{\abs{\Phi_x^{(T)}} > \varphi_{\fa / 4}} \cup \hb{\abs{\Psi_x^{(T)}} > \varphi_{\fa / 4}}\,,
\end{equation*}
we therefore deduce by a union bound that
\begin{multline} \label{Omega_prod}
\P_\theta \pBB{\bigcap_{x \in T}\Omega_x} \leq \P_\theta \pBB{\bigcap_{x \in T}\Omega_x^{(T)}} + \sum_{x \in T} \P_\theta \pb{\abs{\Phi_x - \Phi_x^{(T)}} > \varphi_{\fa / 4}}
\\
+ \sum_{x \in T} \P_\theta \pb{\abs{\Phi_x} \leq \varphi_{\fa / 2}, \abs{\Psi_x - \Psi_x^{(T)}} > \varphi_{\fa / 4}}\,.
\end{multline}

We begin by estimating the first term of \eqref{Omega_prod}. To that end, we observe that, conditioned on $H^{(T)}$, the family $(\Omega_x^{(T)})_{x \in T}$ is independent. Using Lemma \ref{lem:xi_leq_xi_T} we therefore get
\begin{equation*}
\P_\theta \pBB{\bigcap_{x \in T}\Omega_x^{(T)}} \leq \E \qBB{ \theta^{(T)} \P \pBB{\bigcap_{x \in T}\Omega_x^{(T)} \biggm| H^{(T)}}} + \cal C N^{-\nu} = \E \qbb{\theta^{(T)} \prod_{x \in T} \P(\Omega_x^{(T)} | H^{(T)})} + \cal C N^{-\nu}\,,
\end{equation*}
and we estimate each factor using \eqref{eq:bound_Phi_Psi_x} from Lemma \ref{lem:prob_typicality_vertex} as
\begin{multline*}
\theta^{(T)} \P(\Omega_x^{(T)} | H^{(T)}) \leq \theta^{(T)}  \P \big( \abs{\Phi_x^{(T)}} > \varphi_{\fa /4} \bigm\vert H^{(T)} \big) + \theta^{(T)}  \P \big( \abs{\Psi_x^{(T)}} > \varphi_{\fa / 4} \bigm\vert H^{(T)} \big)
\\
\leq 2 \ee^{-8 q \varphi_{\fa}^2 d} \leq \ee^{-4 q \varphi_{\fa}^2 d}\,,
\end{multline*}
where in the last step we used that $\ee^{-4 q \varphi_{\fa}^2 d} \leq 1/2$. We conclude that
\begin{equation*}
\P_\theta \pBB{\bigcap_{x \in T}\Omega_x^{(T)}} \leq \ee^{-4 q \varphi_{\fa}^2 d \abs{T}} + \cal C N^{-\nu}\,.
\end{equation*}

Next, we estimate the second term of \eqref{Omega_prod}. After renaming the vertices, we may assume that $T = [k]$ with $k \leq \varphi_\fa d / \cal C$, so that we get from \eqref{eq:bound_Phi_Psi_difference} from Lemma \ref{lem:prob_typicality_vertex} (using that $\varphi_\fa d / \cal C \leq \fc d / \Gamma^2$ provided that $\cal D$ in \eqref{d_condition_deloc2} is chosen large enough, depending on $\fa$), by telescoping and recalling Lemma \ref{lem:xi_leq_xi_T},
\begin{equation} \label{delta_Phi}
\abs{\Phi_x - \Phi_x^{(T)}} \leq \sum_{i = 0}^{k-1} \absb{\Phi_x^{([i])} - \Phi_x^{([i+1])}} \leq O \pbb{\frac{k}{d}} \leq \varphi_{\fa / 4}
\end{equation}
with very high probability on the event $\{\theta = 1\}$, if the constant $\cal C$ in the upper bound $\varphi_\fa d / \cal C$ on $k$ is large enough.

The last term of \eqref{Omega_prod} is estimated analogously, with the additional observation that, by definition of $\Phi_x$ and since $\varphi_{\fa /2} \leq 1/2$, on the event $\{\abs{\Phi_x} \leq \varphi_{\fa / 2}\}$ we have $\beta_x \leq 2$. Thus, on the event $\{\theta = 1\} \cap \{\abs{\Phi_x} \leq \varphi_{\fa / 2}\}$ we have, by Lemma \ref{lem:xi_leq_xi_T},
\begin{equation} \label{delta_Psi}
\abs{\Psi_x - \Psi_x^{(T)}} \leq \sum_{i = 0}^{k-1} \absb{\Psi_x^{([i])} - \Psi_x^{([i+1])}} \leq \cal O \pbb{\frac{k(1 + \beta_x)}{d}} \leq \varphi_{\fa / 4}
\end{equation}
with very high probability, for large enough $\cal C$ in the upper bound on $k$. We conclude that the two last terms of \eqref{Omega_prod} are bounded by $\cal C N^{-\nu}$, and the proof of \eqref{eq:bound_decoupling} is therefore complete.

The proof of \eqref{eq:bound_decoupling2} is identical, replacing the matrix $M$ with the matrix $M^{(x)}$.
\end{proof}

\subsection{Self-consistent equation and proof of Theorem~\ref{thm:local_law}}  \label{subsec:proof_local_law} 

In this subsection, we derive an approximate self-consistent equation for the Green function $G$, and use it to prove Theorem \ref{thm:local_law}. The key ingredient is Proposition~\ref{lem:bootstrapping_step} below, which provides a bootstrapping bound stating that if $G_{xx} - m_{\beta_x}$ is smaller than some constant then it is in fact bounded by $\varphi_\fa$ with very high probability. It is proved by first deriving and solving a self-consistent equation for the entries $G_{xx}$ indexed by typical vertices $x \in \cal T_\fa$, and using the obtained bounds to analyse $G_{xx}$ for atypical vertices $x \in \cal T^c_\fa$.

We begin with a simple algebraic observation.

\begin{lemma}[Approximate self-consistent equation] \label{lem:self_consistent_G}
For any $x \in [N]$ and $z \in \C_+$, we have 
\[ \frac{1}{G_{xx}}  = - z - \sum_{y}^{(x)} \abs{H_{xy}}^2 G_{yy}^{(x)} + Y_x\,, \] 
where we introduced the error term 
\begin{equation} \label{eq:def_error_term} 
 Y_x \deq H_{xx} + \frac{f}{N}  - \sum_{a \neq b}^{(x)} H_{xa} G_{ab}^{(x)} H_{bx} - \sum_{a,b}^{(x)} \bigg( \frac{f}{N} \Big( H_{xa}G_{ab}^{(x)} + G_{ab}^{(x)} H_{bx} \Big) + \frac{f^2}{N^2} G_{ab}^{(x)}\bigg)\,. 
\end{equation}
\end{lemma} 

\begin{proof} 
The lemma follows directly from \eqref{eq:Schur_complement} and the definition \eqref{eq:def_M}.
\end{proof} 

Let $\theta$ be defined as in \eqref{eq:def_xi} with some $\Gamma \geq 1$. The following lemma provides a priori bounds on the error terms appearing in the self-consistent equation.

\begin{lemma} \label{lem:auxiliary_bounds}
For all $z \in \C$ with $\im z \geq N^{-1 + \kappa}$, with very high probability,
\begin{subequations} 
\begin{align} 
\theta \max_x \abs{Y_x} & \leq \Cnu d^{-1/2}, \label{eq:bound_Y_x}\\ 
\theta \max_{x \neq y} \abs{G_{xy}} & \leq \Cnu d^{-1/2}, \label{eq:bound_offdiag} \\ 
\theta \max_{x \neq a \neq y} \abs{G_{xy} - G_{xy}^{(a)}} & \leq \Cnu d^{-1}. \label{eq:bound_minor_comparison} 
\end{align} 
\end{subequations} 
\end{lemma} 

\begin{proof} 
We first estimate $Y_x$. 
From Definition~\ref{def:sparse}, the upper bound on $f$, and \eqref{d_condition_deloc2}, we conclude that $\abs{H_{xx}} + f / N = O(d^{-1/2})$ almost surely. 
Moreover, the Cauchy--Schwarz inequality, Lemma \ref{lem:Ward}, \eqref{eq:boundedness_resolvent_minor} and the upper bound on $f$ imply
\[ \theta \frac{f^2}{N^2} \absbb{\sum_{a,b}^{(x)} G_{ab}^{(x)}} \leq C_\kappa \frac{f^2}{\sqrt{N \im z}} \leq C_\kappa N^{-\kappa/6} \leq \frac{\Cnu}{\sqrt{d}}\,, \]
for some constant $C_\kappa$ depending only on $\kappa$.
Next, we use the first estimate of \eqref{eq:LDB_wvhp}, Lemma \ref{lem:Ward}, and the upper bound on $f$ to conclude that 
\[ \frac{f}{N} \theta\absbb{\sum_{a,b}^{(x)} H_{xa} G_{ab}^{(x)} }  + \frac{f}{N} \theta \absbb{\sum_{a,b}^{(x)} G_{ab}^{(x)} H_{bx}} \leq \frac{\Cnu}{\sqrt{d}} \frac{f}{\sqrt{N \im z}} \leq \frac{\Cnu}{\sqrt{d}} N^{-\kappa/3}  \leq \frac{\Cnu}{\sqrt{d}} \] 
with very high probability (compare the proof of \eqref{eq:proof_bound_Gamma_k_aux5}). 
Moreover, from Lemma \ref{lem:Ward} and the second estimate of \eqref{eq:LDB_wvhp} we deduce that remaining term in \eqref{eq:def_error_term} is $\cal O(d^{-1}) = \cal O(d^{-1/2})$. This concludes the proof of \eqref{eq:bound_Y_x}. 

For the proof of \eqref{eq:bound_offdiag}, we start from \eqref{eq:expansion_G_off_diag} and use $M_{xa} = H_{xa} + f/ N $ to obtain 
\[ G_{xy} = - G_{xx} \sum_{a}^{(x)} H_{xa} G_{ay}^{(x)} - G_{xx} H_{xy} G_{yy}^{(x)} - \frac{f}{N} G_{xx} \sum_a^{(x)} G_{ay}^{(x)}. \] 
Similar arguments as in \eqref{eq:proof_bound_Gamma_k_aux5} and \eqref{eq:proof_bound_Gamma_k_aux4} show that the first and third term, respectively, are bounded by $\Cnu d^{-1/2}$ with very high probability.
The same bound for the second term follows from Definition~\ref{def:sparse} and \eqref{eq:boundedness_resolvent_minor} in Lemma~\ref{lem:bound_Gamma_k}. 
This proves \eqref{eq:bound_offdiag}. 

Finally, \eqref{eq:bound_minor_comparison} follows directly from \eqref{eq:smallness_diff_minor}. 
\end{proof}

Proposition \ref{lem:bootstrapping_step} below is the main tool behind the proof of Theorem \ref{thm:local_law}.
To formulate it, we introduce the $z$-dependent random control parameters
\begin{equation*}
\Lambda_{\r d} \deq \max_{x} \abs{G_{xx} - m_{\beta_x}}\,, \qquad
\Lambda_{\r o} \deq \max_{x \neq y} \abs{G_{xy}}\,, \qquad \Lambda \deq \Lambda_{\r d} \vee \Lambda_{\r o}\,,
\end{equation*}
and, for some constant $\lambda \leq 1$, the indicator function
\begin{equation} \label{eq:def_vartheta} 
\phi \deq \ind{\Lambda \leq \lambda }\,.
\end{equation}
Proposition \ref{lem:bootstrapping_step} below provides a strong bound on $\Lambda$ provided the a priori condition $\phi = 1$ is satisfied. 
Each step of its proof is valid provided $\lambda$ is chosen small enough depending on $\kappa$.  
Note that, owing to \eqref{eq:m_alpha_bounded}, there is a deterministic constant $\Gamma$, depending only on $\kappa$, such that, for all $z \in \mathbf{S}$, we have
\begin{equation} \label{eq:G_xy_bounded_Gamma} 
\phi \max_{x,y} \abs{G_{xy}} \leq \Gamma\,.
\end{equation}
In particular, if $\Gamma$ in the definition \eqref{eq:def_xi} of $\theta$ is chosen as in \eqref{eq:G_xy_bounded_Gamma} then
\begin{equation} \label{phi_theta}
\phi \leq \theta\,.
\end{equation}

\begin{proposition} \label{lem:bootstrapping_step} 
There exists $\lambda > 0$, depending only on $\kappa$, such that, for all  $z \in \mathbf{S}$, with very high probability,
\[ \phi \Lambda \leq \Cnu \varphi_\fa\,. \] 
\end{proposition} 

For the proof of Proposition~\ref{lem:bootstrapping_step}, we employ the results of the previous subsections to show that the diagonal entries $(G_{xx})_{x \in \cal T_\fa}$ of the Green function of $M$ at the typical vertices satisfy the approximate self-consistent equation \eqref{eq:self_consistent_eq_perturbed} below. 
This is a perturbed version of the relation \eqref{m_quadr}
for the Stieltjes transform $m$ of the semicircle law, which holds for all $z \in \C_+$. 
The stability estimate, \eqref{eq:stability_estimate} below, then implies that $G_{xx}$ and $m$ are close for all $x \in \cal T_\fa$. From this we shall, in a second step, deduce that $G_{xx}$ is close to $m_{\beta_x}$ for all $x$; this steps includes also the atypical vertices.

The next lemma is a relatively standard stability estimate of self-consistent equations in random matrix theory (compare e.g.\ to \cite[Lemma~3.5]{ErdosYauYin2012}). It is proved in Appendix~\ref{sec:proof_stability}. 

\begin{lemma}[Stability of the self-consistent equation for $m$] \label{lem:stability}
Let $\cal X$ be a finite set, $\kappa>0$, and $z \in \C_+$ satisfy $\abs{\Re z} \leq 2- \kappa$.  
We assume that, for two vectors $(g_x)_{x \in \cal X}$, $(\eps_x)_{x \in \cal X} \in \C^{\cal X}$, the identities
\begin{equation} \label{eq:self_consistent_eq_perturbed} 
  \frac{1}{g_x} = -z - \frac{1}{\abs{\cal X}}\sum_{y \in \cal X} g_y + \eps_x
\end{equation} 
hold for all $x \in \cal X$. Then there are constants $b, C \in (0,\infty)$, depending only on $\kappa$, such that if $\max_{x \in \cal X} \abs{g_x -m(z)} \leq b$ then
\begin{equation} \label{eq:stability_estimate} 
 \max_{x \in \cal X} \abs{g_x - m(z)} \leq C \max_{x \in \cal X} \abs{\eps_x}, 
\end{equation}
where $m(z)$ satisfies \eqref{m_quadr}. 
\end{lemma} 

\begin{proof}[Proof of Proposition~\ref{lem:bootstrapping_step}] 
Throughout the proof, we work on the event $\{\phi = 1\}$, which, by \eqref{phi_theta}, is contained in the event $\{\theta = 1\}$. Fix $\fa$ as in Proposition \ref{pro:typcial_vertices}. Throughout the proof we use that $d^{-1/2} \leq \varphi_\fa$ by the upper bound in \eqref{d_condition_deloc2}.
Owing to \eqref{eq:bound_offdiag}, it suffices to estimate $\Lambda_{\r d}$.
Let $b$ be chosen as in Lemma~\ref{lem:stability}, and set $\lambda \deq b/2$ in the definition \eqref{eq:def_vartheta} of $\phi$. 

For the analysis of $G_{xx}$ we distinguish the two cases $x \in \cal T_\fa$ and $x \notin \cal T_\fa$. 

If $x \in \cal T_\fa$ then we write using Lemma~\ref{lem:self_consistent_G} and the definition \eqref{eq:def_Phi_Psi} of $\Psi_x$ that 
\begin{equation*}
\frac{1}{G_{xx}} = -z - \sum_y^{(x)} \abs{H_{xy}}^2 G_{yy}^{(x)} + Y_x = -z - \frac{1}{N} \sum_y^{(x)} G_{yy}^{(x)} + Y_x - \Psi_x = -z - \frac{1}{\abs{\cal T_\fa}} \sum_{y \in \cal T_\fa} G_{yy} + \eps_x\,,
\end{equation*}
where the error term $\eps_x$ satisfies
\begin{equation} \label{epsilon_estimate}
\abs{\epsilon_x} = \cal O \pbb{d^{-1/2} + \frac{1}{N} \exp ( q \varphi_\fa^2 d) + \exp(-2 q \varphi_\fa^2 d ) + \varphi_\fa} = \cal O(\varphi_\fa)
\end{equation}
with very high probability. Here, in the first step of \eqref{epsilon_estimate} we used \eqref{eq:bound_Y_x}, \eqref{eq:bound_minor_comparison}, Proposition~\ref{pro:typcial_vertices} \ref{item:a1}, and the bound on $\Psi_x$ in the definition \eqref{eq:def_cal_T} of $\cal T_\fa$, and in the second step of \eqref{epsilon_estimate} we used that $\varphi_\fa^2 d = \fa^2 (\log N)^{2/3} d^{-1/3}$ and \eqref{d_condition_deloc2} imply $(\log N)^{1/6} / \cal C \leq \varphi^2_\fa d \leq \cal C (\log N)^{1/2}$, which yields
\begin{equation}  \label{eq:upper_bound_error_terms} 
\frac{1}{N} \exp ( q \varphi_\fa^2 d) + \exp(-2 q \varphi_\fa^2 d ) \leq \cal C d^{-10} \leq \varphi_\fa\,.
\end{equation}

Thus, for $(G_{xx})_{x \in \cal T_\fa}$ we get the self-consistent equation in \eqref{eq:self_consistent_eq_perturbed} with $g_x = G_{xx}$ and $\cal X = \cal T_\fa$. Moreover, by the bound on $\Phi_x$ in the definition \eqref{eq:def_cal_T} of $\cal T_\fa$, we have $\beta_x = 1 + \cal O(\varphi_\fa)$.  
Hence, by \eqref{eq:diff_m_alpha_m}, the assumption $\phi = 1$ and $d \geq \Cnu \sqrt{\log N}$, we find that
\begin{equation*}
\abs{G_{xx} - m} \leq \abs{G_{xx} - m_{\beta_x}} + \abs{m_{\beta_x} - m} \leq b\,,
\end{equation*}
choosing the constant $\cal D$ in \eqref{d_condition_deloc2} large enough that the right-hand side of \eqref{eq:diff_m_alpha_m}, i.e.\ $C \abs{\beta_x - 1}$, is bounded by $b/2$.
Hence Lemma~\ref{lem:stability} is applicable and we obtain $\abs{G_{xx} - m} = O(\max_{y \in \cal T_\fa} \abs{\eps_y})$.
Therefore, we obtain 
\begin{equation} \label{eq:estimate_G_xx_m_beta_x_typical} 
 \abs{G_{xx} - m_{\beta_x}} \leq \abs{G_{xx} - m} + \abs{m - m_{\beta_x} } \leq
\Cnu \varphi_\fa
\end{equation} 
with very high probability.  This concludes the proof in the case $x \in \cal T_\fa$.

What remains is the case $x \notin \cal T_\fa$. In that case, we obtain from Lemma~\ref{lem:self_consistent_G} that  
\begin{equation} \label{eq:estimate_G_xx_m_beta_x_atypical} 
\frac{1}{G_{xx}} = -z - \sum_{y \in \cal T_\fa}^{(x)} \abs{H_{xy}}^2 G_{yy}^{(x)} - \sum_{y \in \cal T_\fa^c}^{(x)} \abs{H_{xy}}^2 G_{yy}^{(x)} + Y_x = -z - \beta_x m + \epsilon_x\,,
\end{equation}
where the error term $\epsilon_x$ satisfies $\epsilon_x = \cal O ((1 + \beta_x) \varphi_\fa)$ with very high probability.
Here we used \eqref{eq:bound_Y_x} as well as \eqref{eq:bound_minor_comparison}, \eqref{eq:upper_bound_error_terms}, \eqref{eq:estimate_G_xx_m_beta_x_typical} and 
Proposition~\ref{pro:typcial_vertices} \ref{item:a2} twice to conclude that 
\[ 
\sum_{y \in \cal T_\fa}^{(x)} \abs{H_{xy}}^2 G_{yy}^{(x)} = \beta_x m + \cal O (\beta_x \varphi_\fa) \,
, \qquad 
\sum_{y \in \cal T_\fa^c}^{(x)} \abs{H_{xy}}^2 G_{yy}^{(x)} = \cal O\big(\varphi_\fa + d^4 \exp(-q \varphi_\fa^2 d )\big) 
= \cal O(\varphi_\fa)  \]
 with very high probability. From \eqref{eq:def_m_alpha} and \eqref{eq:estimate_G_xx_m_beta_x_atypical} we therefore get
 \begin{equation} \label{Gxx_estimate}
 G_{xx} - m_{\beta_x} =  - m_{\beta_x} \, \frac{1}{-z - \beta_x m + \epsilon_x} \, \epsilon_x\,.
\end{equation}
To estimate the right-hand side of \eqref{Gxx_estimate}, we consider the cases $\beta_x \leq 1$ and $\beta_x > 1$ separately.

If $\beta_x \leq 1$ then, by \eqref{eq:m_alpha_bounded}, the first factor of \eqref{Gxx_estimate} is bounded by $C$. Thus, by \eqref{eq:def_m_alpha}, the second factor is bounded by $2 C$ provided that $\abs{\epsilon_x} \leq 1/{2C}$ by choosing $\cal D$ in \eqref{d_condition_deloc2} large enough, and the third factor is bounded by $\cal C \varphi_\fa$. This yields the claim.

If $\beta_x > 1$, we use that $\im m \geq c$ for some constant $c > 0$ depending only on $\kappa$ and $L$. Thus, the right-hand side of \eqref{Gxx_estimate} is bounded in absolute value, again using \eqref{eq:m_alpha_bounded}, by $C \frac{1}{\beta_x c/2} \cal C \beta_x \varphi_\fa$, provided that $\cal D$ in \eqref{d_condition_deloc2} is chosen large enough. This yields the claim.
\end{proof}

\begin{proof}[Proof of Theorem~\ref{thm:local_law}]
After possibly increasing $L$, we can assume that $L$ in the definition of $\mathbf{S}$ in \eqref{eq:def_S_spectral_domain} satisfies $L \geq 2/\lambda + 1$, 
where $\lambda$ is chosen as in Proposition \ref{lem:bootstrapping_step}.

We first show that \eqref{eq:local_law_averaged} follows from \eqref{eq:local_law_entrywise}. 
Indeed, averaging the estimate on $\abs{G_{xx} - m_{\beta_x}}$ in \eqref{eq:local_law_entrywise} over $x \in [N]$, using that $m_{\beta_x} = m + O(\varphi_\fa)$ for $x \in \cal T_\fa$ by 
\eqref{eq:diff_m_alpha_m} and estimating the summands in $\cal T_\fa^c$ by Proposition~\ref{pro:typcial_vertices} \ref{item:a1} 
and \eqref{eq:m_alpha_bounded} yield \eqref{eq:local_law_averaged} due to \eqref{eq:upper_bound_error_terms}.  

What remains is the proof of \eqref{eq:local_law_entrywise}. 
Let $z_0 \in\mathbf{S}$, set $J \deq \min \{ j \in \N_0 \colon \Im z_0 + j N^{-3} \geq 2 / \lambda \}$, 
and define $z_j \deq z_0 + \ii j N^{-3}$ for $j \in [J]$. 
We shall prove the bound in \eqref{eq:local_law_entrywise} at $z = z_j$ by induction on $j$, starting from $j = J$ and going down to $j = 0$.
Since $\abs{G_{xy}(z)} \leq (\Im z)^{-1}$ and $\abs{m_{\beta_x}(z)} \leq (\Im z)^{-1}$ for all $x,y \in [N]$, 
we have $\max_x \abs{G_{xx}(z_J) - m_{\beta_x}(z_J)} \leq \lambda$ and $\phi(z_J) = 1$.

For the induction step $j \to j - 1$, suppose that $\phi(z_j) = 1$ with very high probability. Then, by Proposition \ref{lem:bootstrapping_step}, we deduce that $\Lambda(z_j) \leq \cal C \varphi_\fa$ with very high probability.
Since $G_{xy}$ and $m_{\beta_x}$ are Lipschitz-continuous on $\mathbf{S}$ with constant $N^2$, we conclude 
that $\Lambda(z_{j-1}) \leq \Cnu \varphi_\fa + N^{-1}$ with very high probability.  
If $N$ is sufficiently large and $\varphi_\fa$ is sufficiently small, obtained by choosing $\cal D$ in \eqref{d_condition_deloc2} large enough, then we deduce that $\Lambda(z_{j-1}) \leq \lambda$ with very high probability and hence $\phi(z_{j - 1}) = 1$ with very high probability. Using Proposition \ref{lem:bootstrapping_step}, this concludes the induction step, and hence establishes $\Lambda(z_0) \leq \cal C \varphi_\fa$ with very high probability. Here we used that the intersection of $J$ events of very high probability is an event of very high probability, since $J \leq C N^3$, where $C$ depends on $\kappa$.
\end{proof}

\appendix

\section{Appendices}
\addtocontents{toc}{\protect\setcounter{tocdepth}{1}}

In the following appendices we collect various tools and explanations used throughout the paper. 

\subsection{Simulation of the $\ell^\infty$-norms of eigenvectors} \label{sec:simulation}
In Figure \ref{fig:evector_simulation} we depict a simulation of the $\ell^\infty$-norms of the eigenvectors of the adjacency matrix $A / \sqrt{d}$ of the Erd\H{o}s-R\'enyi graph $\bb G(N,d/N)$ restricted to its giant component. We take $d = b \log N$ with $N = 10'000$ and $b = 0.6$. The eigenvalues and eigenvectors are drawn using a scatter plot, where the horizontal coordinate is the eigenvalue and the vertical coordinate the $\ell^\infty$-norm of the associated eigenvector. The higher a dot is located, the more localized the associated eigenvector is. Complete delocalization corresponds to a vertical coordinate $\approx 0.01$, and localization at a single site to a vertical coordinate $1$. Note the semilocalization near the origin and outside of $[-2,2]$. The two semilocalized blips around $\pm 0.4$ are a finite-$N$ effect and tend to $0$ as $N$ is increased. The Perron-Frobenius eigenvalue is an outlier near $2.8$ with delocalized eigenvector.
\begin{figure}[!ht]
\begin{center}
{\small 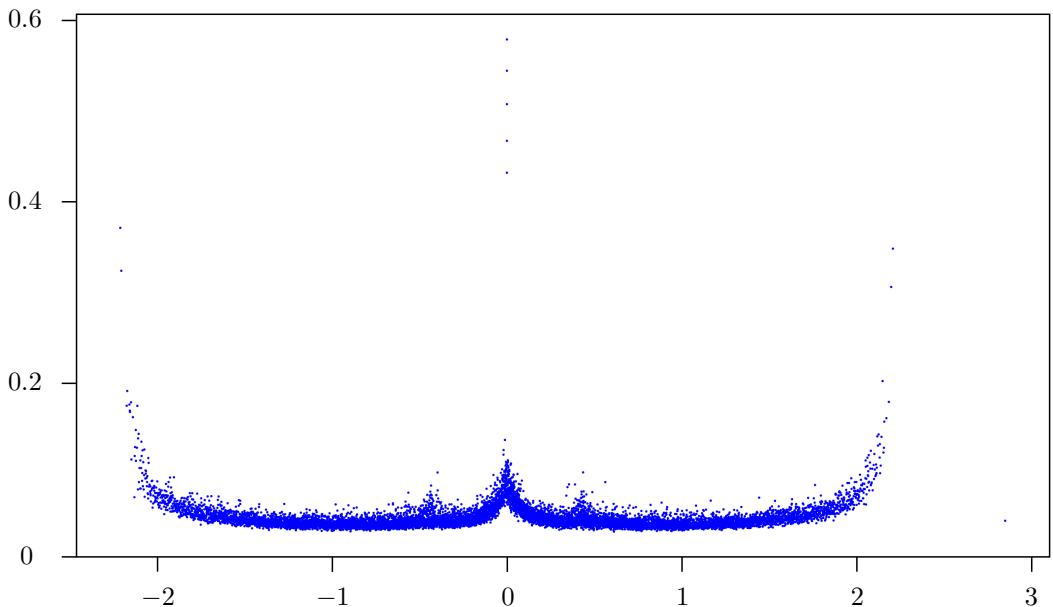}
\end{center}
\caption{A scatter plot of $(\lambda, \norm{\f w}_\infty)$ for all eigenvalue-eigenvector pairs $(\lambda, \f w)$ of the adjacency matrix $A / \sqrt{d}$ of the critical Erd\H{o}s-R\'enyi graph restricted to its giant component, where $N = 10'000$ and $d = 0.6 \log N$.
\label{fig:evector_simulation}}
\end{figure}

\subsection{Spectral analysis of the infinite rooted $(p,q)$-regular tree} \label{sec:mu}
In this appendix we describe the spectrum, eigenvectors, and spectral measure of the following simple graph.

\begin{definition}
For $p,q \in \N^*$ we define $\bb T_{p,q}$ as \emph{the infinite rooted $(p,q)$-regular tree}, whose root has $p$ children and all other vertices have $q$ children.
\end{definition}

A convenient way to analyse the adjacency matrix of $\bb T_{p,q}$ is by tridiagonalizing it around its root. To that end, we first review the tridiagonalization\footnote{The tridiagonalization algorithm that we use is the Lanczos algorithm. Tridiagonalizing matrices in numerical analysis and random matrix theory \cite{dumitriu2002matrix,Tro84} is usually performed using the numerically more stable Householder algorithm. However, when applied to the adjacency matrix $X = A$ of a graph, the Lanczos algorithm is more convenient because it can exploit the sparseness and local geometry of $A$.} of a general symmetric matrix $X \in \R^{N \times N}$ around a vertex $x \in [N]$; we refer to \cite[Appendices A--C]{ADK19} for details. Let $r \in \N$ and $x \in [N]$. Suppose that the vectors $\f 1_x, X \f 1_x, X^2 \f 1_x, \dots, X^r \f 1_x$ are linearly independent, and denote by $\f g_0, \f g_1, \f g_2, \dots, \f g_r$ the associated orthonormalized sequence. Then the tridiagonalization of $X$ around $x$ up to radius $r$ is the $(r + 1) \times (r+1)$ matrix $Z = (Z_{ij})_{i,j = 0}^r$ with $Z_{ij} \deq \scalar{\f g_i}{X \f g_j}$. By construction, $Z$ is tridiagonal and conjugate to $X$ restricted to the subspace $\Span\{\f g_0, \f g_1, \dots, \f g_r\}$.

Let now $X = A \equiv A^{\bb T_{p,q}}$ be the adjacency matrix of $\bb T_{p,q}$, whose root we denote by $o$. Then it is easy to see that $\f g_i = \f 1_{S_i(o)} / \norm{\f 1_{S_i(o)}}$ and the tridiagonalization of $A$ around the root up to radius $\infty$ is
the infinite matrix $\sqrt{q} Z(p/q)$, where
\begin{equation} \label{def_M}
Z(\alpha) \deq \begin{pmatrix} 0 & \sqrt{\alpha}\\
\sqrt{\alpha} &  0 & 1\\
 & 1 &  0 & 1\\
 & & 1 & 0 & \ddots
 \\
 & & & \ddots & \ddots
\end{pmatrix}\,.
\end{equation}
If $\alpha > 2$, a transfer matrix analysis (see \cite[Appendix C]{ADK19}) shows that $Z(\alpha)$ has precisely two eigenvalues in $\R \setminus [-2,2]$, which are $\pm \Lambda(\alpha)$. The associated eigenvectors are $((\pm)^i u_i)_{i \in \N}$, where $u_0 > 0$ and $u_i \deq \frac{\sqrt{\alpha}}{(\alpha - 1)^{i/2}} \, u_0$ for $i \geq 1$. Note that the eigenvector components are exponentially decaying since $\alpha > 2$, and hence $u_0$ can be chosen so that the eigenvectors are normalized. Going back to the original vertex basis of $\bb T_{p,q}$, setting $\alpha = p/q$, we conclude that the adjacency matrix $A$ has eigenvalues $\pm \sqrt{q} \Lambda(\alpha)$ with associated eigenvectors $\sum_{i \in \N} (\pm)^i u_i \f 1_{S_i(o)} / \norm{\f 1_{S_i(o)}}$.

Next, we show that the measure $\mu_\alpha$ from \eqref{mu_alpha} is the spectral measure at the root of $A^{\bb T_{p,q}} / \sqrt{d}$ and the spectral measure at $0$ of \eqref{def_M}.

\begin{lemma}
\begin{enumerate}[label=(\roman*)]
\item \label{itm:mu_M}
For any $\alpha \geq 0$ the measure $\mu_\alpha$ is the spectral measure of $Z(\alpha)$ at $0$.
\item \label{itm:mu_T}
For any $p,q \in \N^*$ the measure $\mu_{p/q}$ is the spectral measure of the normalized adjacency operator $A^{\bb T_{p,q}} / \sqrt{q}$ at the root.
\end{enumerate}
\end{lemma}

\begin{proof}
For \ref{itm:mu_M}, define the vector $\f e_0 = (1,0,0,\dots) \in \ell^2(\N)$. The spectral measure of $Z(\alpha)$ with respect to $\f e_0$ is characterized by its Stieltjes transform
\begin{equation} \label{tilde_m_short}
\scalarb{\f e_0}{(Z(\alpha) - z)^{-1} \f e_0} = \frac{1}{- z - \alpha \scalarb{\f e_0}{(Z(1) - z)^{-1} \f e_0}}\,.
\end{equation}
Here, we used Schur's complement formula on the Green function $(Z(\alpha) - z)^{-1}$, observing that the minor of $Z(\alpha)$ obtained by removing the zeroth row and column is $Z(1)$.  Setting $\alpha = 1$ in \eqref{tilde_m_short} and recalling the defining relation \eqref{m_quadr} of the Stieltjes transform $m$ of the semicircle law, we conclude that $\scalarb{\f e_0}{(Z(1) - z)^{-1} \f e_0}= m(z)$ and hence from \eqref{eq:def_m_alpha} and \eqref{tilde_m_short} we get $\scalarb{\f e_0}{(Z(\alpha) - z)^{-1} \f e_0} = m_\alpha(z)$, as desired.

The proof of \ref{itm:mu_T} is analogous. Denote the root of $\bb T_{p,q}$ by $o$. Again using Schur's complement formula to remove the $o$th row and column of $H = A^{\bb T_{p,q}} / \sqrt{q}$, we deduce that
\begin{equation} \label{A_T_Schur}
\scalarb{\f 1_o}{\pb{A^{\bb T_{p,q}} / \sqrt{q} - z}^{-1} \f 1_o} = \pbb{-z - \frac{p}{q} \scalarb{\f 1_o}{\pb{A^{\bb T_{q,q}} / \sqrt{q} - z}^{-1} \f 1_o}}^{-1}\,,
\end{equation}
where we used that $\bb T_{p,q}$ from which $o$ has been removed consists of $p$ disconnected copies of $\bb T_{q,q}$. Setting $p = q$ in \eqref{A_T_Schur} and comparing to \eqref{m_quadr} implies that the left-hand side of \eqref{A_T_Schur} is equal to $m(z)$ if $p = q$, and hence \ref{itm:mu_T} for general $p$ follows from \eqref{eq:def_m_alpha}.

Finally, we remark that the equality of the spectral measures of $Z(p/q)$ and $A^{\bb T_{p,q}} / \sqrt{q}$ can also be seen directly, by noting that $Z(p/q)$ is the tridiagonalization of $A^{\bb T_{p,q}} / \sqrt{q}$ around the root $o$.
\end{proof}

We conclude with some basic estimates for the Stieltjes transform $m_\alpha$ of $\mu_\alpha$ used in Section \ref{sec:delocalization}.

\begin{lemma} \label{lem:properties_m_alpha} 
For each $\kappa>0$ there is a constant $C>0$ depending only on $\kappa$ such that for all $z \in \f S$ and all $\alpha \geq 0$ we have
\begin{align} 
\label{eq:m_alpha_bounded} \abs{m_\alpha(z)} & \leq C\,, \\ 
\abs{m_\alpha(z) - m(z)} & \leq C \abs{\alpha - 1} \label{eq:diff_m_alpha_m}\,.
\end{align} 
\end{lemma}

\begin{proof} 
The simple facts follow directly from the corresponding properties of the semicircle law and its Stieltjes transform $m$ (see e.g.\ \cite[Lemma 3.3]{BenyachKnowles2017}). 
We leave the details to the reader. 
\end{proof}

\subsection{Bounds on adjacency matrices of trees}

In this appendix we derive estimates on the operator norm of a tree. We start with a standard estimate on the operator norm of a graph.

\begin{lemma} \label{lem:forest_bound}
Let $\bb T$ be a graph whose vertices have degree at most $q+1$ for some $q \geq 1$. Then $\norm{A^{\bb T}} \leq q+1$ and if in addition $\bb T$ is a tree then $\norm{A^{\bb T}} \leq 2 \sqrt{q}$.
\end{lemma}
\begin{proof}
The first claim is obvious by the Schur test for the operator norm. To prove the second claim, choose a root $o$ and denote by $C_x$ the set of children of the vertex $x$. Then for any vector $\f w = (w_x)$ we have
\begin{multline*}
\absb{\scalarb{\f w}{A^{\bb T} \f w}} = \absBB{\sum_{x,y} w_x A_{xy}^{\bb T} w_y} = 2 \absBB{\sum_x \sum_{y \in C_x} w_x w_y} \leq \sum_{x} \sum_{y \in C_x} \pbb{\frac{1}{\sqrt{q}} w_x^2 + \sqrt{q} w_y^2}
\\
\leq \frac{q+1}{\sqrt{q}} w_o^2 + \sum_{x \neq o} \pbb{\frac{q}{\sqrt{q}} w_x^2 + \sqrt{q} w_x^2} \leq 2 \sqrt{q} \sum_x w_x^2\,,
\end{multline*}
where in third step we used Young's inequality and in the fourth step that each vertex in the sum appears once as a child and at most $q$ times as a parent. This concludes the proof.
\end{proof}

The same proof shows that if $\bb T$ is a rooted tree whose root has at most $p$ children and all other vertices at most $q$ children, then $\norm{A^{\bb T}} \leq \sqrt{q} (p/q \vee 2)$. This bound is sharp for $p \leq 2q$ but not for $p > 2q$. The sharp bound in the latter case is established in the following result.

\begin{lemma}\label{lem:normTree}
Let $p,q \in \N^*$.
Let $\bb T$ be a tree whose root has $p$ children and all the other vertices have at most $q$ children.
Then the adjacency matrix $A^{\bb T}$ of $\bb T$ satisfies $\|A^{\bb T}\|\leq \sqrt{q} \Lambda(p/q \vee 2)$.
\end{lemma}

\begin{proof}
Let $r \in \N$ and denote by $\bb T_{p,q}(r)$ the rooted $(p,q)$-regular tree of depth $r$, whose root $x$ has $p$ children, all vertices at distance $1 \leq i \leq r$ from $x$ have $q$ children, and all vertices at distance $r+1$ from $x$ are leaves. For large enough $r$, we can exhibit $\bb T$ as a subgraph of $\bb T_{p,q}(r)$. By the Perron-Frobenius theorem,
\begin{equation} \label{PF_quad1}
\norm{A^{\bb T}} = \scalar{\f w}{A^{\bb T} \f w}
\end{equation}
for the some normalized eigenvector $\f w$ whose entries are nonnegative. We extend $\f w$ to a vector indexed by the vertex set of $\bb T_{p,q}(r)$ by setting $w_y = 0$ for $y$ not in the vertex set of $\bb T$. Clearly,
\begin{equation} \label{PF_quad2}
\scalar{\f w}{A^{\bb T} \f w} \leq \scalar{\f w}{A^{\bb T_{p,q}(r)} \f w}\,.
\end{equation}
Abbreviating $A \equiv A^{\bb T_{p,q}(r)}$, it therefore remains to estimate the right-hand side of \eqref{PF_quad2} for large enough $r$. To that end, we define $Z$ as the tridiagonalization of $A$ around the root up to radius $r$ (see Appendix \ref{sec:mu}). The associated orthonormal set $\f g_0, \f g_1, \dots, \f g_r$ is given by $\f g_i = \f 1_{S_i(x)}/\|\f 1_{S_i(x)}\|$, and $Z = \sqrt{q} Z_r(p/q)$, where $Z_r(\alpha)$ is the upper-left $(r+1) \times (r+1)$ block of \eqref{def_M}. 
We introduce the orthogonal projections $P_0 \deq \f g_0 \f g_0^*$ and $P \deq \sum_{i = 0}^r \f g_i \f g_i^*$. Clearly, $P_0 P = P_0$ and hence $(1 - P) (1 - P_0) = 1 - P$. For large enough $r$ the vectors $\f g_r$ and $\f w$ have disjoint support, and hence $(1 - P) A P \f w = (1 - P) A \sum_{i = 0}^{r - 1} \f g_i \scalar{\f g_i}{\f w} = 0$,
since $A \f g_i \subset \Span \{\f g_{i-i}, \f g_{i+1}\}$ for $i < r$. Thus we have
\begin{align}
\scalar{\f w}{A \f w} &= \scalar{\f w}{PAP \f w} + \scalar{\f w}{(1 - P) A (1 - P) \f w}
\notag \\ \label{w_quad_est}
&= \scalar{\f w}{PAP \f w} + \scalar{\f w}{(1 - P) (1 - P_0) A (1 - P_0) (1 - P) \f w}\,.
\end{align}
From \cite[Appendices B and C]{ADK19} we find
\begin{equation}
\lim_{r \to \infty} \norm{P A P} = \lim_{r \to \infty} \norm{Z} = \sqrt{q} \Lambda(p/q \vee 2)\,.
\end{equation}
Moreover, the operator $(1 - P_0) A (1 - P_0)$ is the adjacency matrix of a forest whose vertices have degree at most $q$. By Lemma \ref{lem:forest_bound}, we therefore obtain $\norm{(1 - P_0) A (1 - P_0)} \leq 2 \sqrt{q}$.
From \eqref{w_quad_est} we therefore get
\begin{equation*}
\limsup_{r \to \infty} \scalar{\f w}{A \f w} \leq  \sqrt{q} \Lambda(p/q \vee 2) \norm{P \f w}^2 + 2 \sqrt{q} \norm{(1 - P) \f w}^2 \leq \sqrt{q} \Lambda(p/q \vee 2) \norm{\f w}^2\,.
\end{equation*}
By \eqref{PF_quad1} and \eqref{PF_quad2}, the proof is complete.
\end{proof}

\subsection{Degree distribution and number of resonant vertices} \label{sec:degrees}

In this appendix we record some basic facts about the distribution of degrees of the graph $\bb G(N,d/N)$, and use them to estimate the number of resonant vertices $\cal W_{\lambda, \delta}$.

The following is a quantitative version of the Poisson approximation of a binomial random variable.

\begin{lemma}[Poisson approximation] \label{lem:binomial_estimate} 
If $D$ is a random variable with law $\op{Binom}(n,p)$ then for $k\leq \sqrt{n}$ and $p \leq 1 / \sqrt{n}$ we have
\[
\mathbb{P}(D = k) = \frac{(pn)^k}{k!} \ee^{-pn} \pbb{1+O\pbb{\frac{k^2}{n} + p^2 n}}\,.
\]
\end{lemma}
\begin{proof}
Plugging the estimates $\p{1-p}^{n-k}= \ee^{(n-k)\log \p{1-p}} = \ee^{-np + O\p{pk + p^2n}}$ and 
\[
\frac{n!}{(n-k)!} = n^k\prod_{i=0}^{k-1}\pbb{1-\frac{i}{n}}=n^k \ee^{\sum_{i=0}^{k-1} \log \pb{1-\frac{i}{n}}}= n^k \ee^{O \pb{\frac{k^2}{n}}}\,,
\]
into
$\mathbb{P}(D_x = k) = \frac{n!}{k! (n-k)!} p^k \p{1-p}^{n-k}$
yields the claim, since $pk \leq k^2/n + p^2 n$.
\end{proof}

\begin{lemma} \label{lem:upper_bound_degrees}
For $\bb G(N,d/N)$ we have $\alpha_x \leq \Cnu \pb{1 +  \frac{ \log N} {d}}$ with very high probability.
\end{lemma} 
\begin{proof}
This is a simple application of Bennett's inequality; see \cite[Lemma 3.3]{ADK19} for details.
\end{proof}

Next, we recall some standard facts about the distribution of the degrees. Define the function $f_d : [1,\infty) \to \big[\frac{1}{2} \log (2 \pi d), \infty\big)$ through
\begin{equation} \label{eq:def_f_d} 
 f_d(\alpha) \defeq d( \alpha \log \alpha - \alpha + 1) + \frac{1}{2} \log (2 \pi \alpha d) \,,
\end{equation}
which is bijective and increasing.
For its interpretation, we note that if $Y \eqdist \op{Poisson}(d)$ then by Stirling's formula we have $\P(Y = k) = \exp\pb{-f_d(k/d) + O \pb{\frac{1}{k}}}$ for any $k \in \N$.
There is a universal constant $C > 0$ such that for $1 \leq l \leq \frac{N}{C \sqrt{d}}$ the equation $f_d(\beta) = \log (N/l)$ has a unique solution $\beta \equiv \beta_l(d)$.
The interpretation of $\beta_l(d)$  is the typical location of $\alpha_{\sigma(l)}$. By the implicit function theorem, we find that $d \mapsto \beta_l(d)$ on the interval $\bigl(0, \frac{N^2}{C l^2}\bigr]$ is a decreasing bijective function.

\begin{definition} \label{def:hp}
An event $\Xi \equiv \Xi_N$ holds with \emph{high probability} if $\P(\Xi) = 1 - o(1)$.
\end{definition}

The following result is a slight generalization of \cite[Proposition D.1]{ADK19}, which can be established with the same proof.
We note that the qualitative notion of high probability can be made stronger and quantitative with some extra effort, which we however refrain from doing here.

\begin{lemma} \label{lem:degree_distr}
If $d \geq 1$ and $l \geq 1$ satisfies $\beta_l(d) \geq 3/2$ then
\begin{equation} \label{deg_est_1}
\abs{\alpha_{\sigma(l)} - \beta_l(d)} \leq \frac{1 \vee (\zeta / \log \beta_l(d))}{d}
\end{equation}
with high probability, where $\zeta$ is any sequence tending to infinity with $N$.
\end{lemma}

The following result\footnote{The assumption $d \gg \log \log N$ in Lemma \ref{lem:degree_counting} is tailored so that it covers the entire range $\alpha \geq 2$, which is what we need in this paper. The assumption on $d$ could also be removed at the expense of introducing a nontrivial lower bound on $\alpha$.} gives bounds on the counting function of the normalized degrees $(\alpha_x)_{x \in [N]}$.

\begin{lemma} \label{lem:degree_counting}
Suppose that $\zeta$ satisfies
\begin{equation} \label{zeta_conditions}
1 \ll \zeta \leq \frac{d}{C \log \log N}
\end{equation}
for some large enough universal constant $C$.
Then for any $\alpha \geq 2$ we have with high probability
\begin{equation} \label{a_counting}
\floorb{(N \ee^{-f_d(\alpha)} - 1) (\log N)^{-2 \zeta}} \leq \abs{\h{x \in [N] \col \alpha_x \geq \alpha}} \leq \ceilb{(N \ee^{-f_d(\alpha)} + 1) (\log N)^{2 \zeta}}\,.
\end{equation}
\end{lemma}
\begin{proof}
If $d > 3 \log N$, then an elementary analysis using Bennett's inequality shows that $\abs{\h{x \in [N] \col \alpha_x \geq \alpha}} = 0$ with high probability. Since $N \ee^{-f_d(\alpha)} \leq 1$ for $\alpha \geq 2$, the claim follows.
Thus, for the following we assume that $d \leq 3 \log N$.

Abbreviate $\Upsilon \deq \frac{3}{2} \frac{\zeta}{d}$, which is an upper bound for the right-hand side of \eqref{deg_est_1}.
For the following we adopt the convention that $\beta_0(d) = \infty$. Choose $l \geq 0$ such that
\begin{equation} \label{alpha_beta_estimate}
\beta_{l+1}(d) < \alpha  \leq \beta_l(d)\,,
\end{equation}
and define
\begin{equation*}
\ul k \deq \floorb{l (\log N)^{-2\zeta}}\,, \qquad \ol k \deq \ceilb{(l+1) (\log N)^{2\zeta}}\,.
\end{equation*}
We shall show that
\begin{equation} \label{ul_k_est}
\beta_{\ul k}(d) - \Upsilon \geq \beta_l(d)
\end{equation}
for $\ul k \geq 1$,
\begin{equation} \label{ol_k_est}
\beta_{\ol k}(d) + \Upsilon \leq \beta_{l+1}(d)\,,
\end{equation}
and
\begin{equation} \label{cond_l}
\ol k \leq N \ee^{-f_d(3/2)}\,.
\end{equation}
Thus $\beta_{\ol k}(d) \geq 3/2$ and, assuming $\ul k \geq 1$, Lemma \ref{lem:degree_distr} is applicable to the indices $\ol k$ and $\ul k$.  We obtain, with high probability,
\begin{equation}
\alpha_{\sigma(\ol k)} \leq \beta_{\ol k}(d) + \Upsilon \leq \beta_{l+1}(d) \leq \alpha \leq \beta_l(d) \leq \beta_{\ul k}(d) - \Upsilon \leq \alpha_{\sigma(\ul k)}\,,
\end{equation}
from which we deduce that
\begin{equation} \label{kk_alpha_est}
\ul k \leq \abs{\h{x \in [N] \col \alpha_x \geq \alpha}} \leq \ol k\,,
\end{equation}
which also holds trivially also for the case $\ul k = 0$.
By applying the function $f_d$ to \eqref{alpha_beta_estimate} we obtain $l \leq N \ee^{-f_d(\alpha)} \leq l+1$,
so that \eqref{kk_alpha_est} yields \eqref{a_counting}.

Next, we verify \eqref{cond_l}.
We consider the cases $l = 0$ and $l \geq 1$ separately. If $l = 0$ then, by the definition of $\beta_{\ol k}(d)$, for \eqref{cond_l} we require $(\log N)^{2 \zeta} + 1 \leq N \ee^{-f_d(3/2)}$, which holds by the assumption $d \leq 3 \log N$ and the upper bound on $\zeta$. Let us therefore suppose that $l \geq 1$. By \eqref{alpha_beta_estimate}, $\alpha \geq 2$, and the definition of $\beta_l(d)$, we have $l \leq N \ee^{-f_d(2)}$, and we have to ensure that $(l+2) (\log N)^{2 \zeta} \leq N \ee^{-f_d(3/2)}$. Since $l \geq 1$, this is satisfied provided that $3 \ee^{-f_d(2)} (\log N)^{2 \zeta} \leq \ee^{-f_d(3/2)}$, which holds provided that $f_d(2) - f_d(3/2) \geq 3 \zeta \log \log N$. This inequality is true because  $f_d(2) - f_d(3/2) \geq f'_d(3/2) /2 \geq d/C$, where we used that $f_d'(\alpha) = d \log \alpha + \frac{1}{2 \alpha}$.

What remains, therefore, is the proof of \eqref{ul_k_est} and \eqref{ol_k_est}. We begin with the proof of  \eqref{ul_k_est}.
We get from the mean value theorem that
\begin{equation} \label{beta_lower_bound_1}
\beta_{\ul k}(d) - \beta_l(d) = f_d^{-1}\pbb{\log \pbb{\frac{N}{\ul k}}} - f_d^{-1}\pbb{\log \pbb{\frac{N}{l}}}\geq  \frac{3}{4 d \log \beta_{\ul k}(d)} \log \pbb{\frac{l}{\ul k}}\,.
\end{equation}
The right-hand side of \eqref{beta_lower_bound_1} is bounded from below by $\Upsilon$ provided that
\begin{equation} \label{logklzeta}
\log \pbb{\frac{l}{\ul k}} \geq 2 \zeta \log \beta_{\ul k}(d)\,.
\end{equation}
We estimate $\beta_{\ul k}(d) \leq \beta_1(d)$ using the elementary bound $f_d(\beta) \geq \frac{d}{10} \beta$ for $\beta \geq 2$, which yields $\log N = f_d(\beta_1(d)) \geq \frac{d}{10} \beta_1(d)$. By assumption on $d$ we therefore get
\begin{equation} \label{beta_1_bound}
\beta_1(d) \leq \log N\,.
\end{equation}
Thus, \eqref{logklzeta} holds by $\ul k \leq l / (\log N)^{2 \zeta}$. This concludes the proof of \eqref{ul_k_est}.

Next, we prove  \eqref{ol_k_est}. As in \eqref{beta_lower_bound_1}, we find
\begin{equation} \label{beta_l1k}
\beta_{l + 1}(d) - \beta_{\ol k}(d) = f_{d}^{-1}\pbb{\log \pbb{\frac{N}{l+1}}} - f_{d}^{-1}\pbb{\log \pbb{\frac{N}{\ol k}}}\geq  \frac{3}{4 d \log \beta_{l+1}(d)} \log \pbb{\frac{\ol k}{l+1}}\,.
\end{equation}
Together with $\beta_{l+1}(d) \leq \beta_1(d) \leq \log N$ from \eqref{beta_1_bound}, we deduce that the right-hand side of \eqref{beta_l1k} is bounded from below by $\Upsilon$ provided that $\log \pb{\frac{\ol k}{l+1}} \geq 2 \zeta \log \log N$,
which is true by definition of $\ol k$. This concludes the proof of \eqref{ol_k_est}.
\end{proof}

The following result follows easily from Lemma \ref{lem:degree_counting}. Recall the definition \eqref{def_theta} of the exponent $\theta_b(\alpha)$.

\begin{corollary} \label{lem:degrees}
Suppose that $\zeta$ satisfies \eqref{zeta_conditions}. Write $d = b \log N$. Then for any $\alpha \geq 2$ we have
\begin{equation*}
\abs{\h{x \in [N] \col \alpha_x \geq \alpha}} \vee 1 = N^{\theta_b(\alpha) + \epsilon}\,, \qquad \epsilon = O \pbb{\frac{\zeta \log \log N}{\log N}}
\end{equation*}
with high probability.
\end{corollary}

Using the exponent $\theta_b(\alpha)$ from \eqref{def_theta} and $\alpha_{\max}(b)$ defined below it, we may state the following estimate on the density of the normalized degrees and the number of resonant vertices. 

\begin{lemma}
\label{lem:alpha_distr}
The following holds for a large enough universal constant $C$.
Suppose that $\zeta$ satisfies \eqref{zeta_conditions}.
Write $d = b \log N$.
\begin{enumerate}[label=(\roman*)]
\item \label{itm:alpha_density}
For $2 \leq \alpha < \beta \leq \alpha_{\max}(b)$ satisfying $\beta - \alpha \geq C \frac{\zeta \log \log N}{d \log \alpha}$, with high probability we have
\begin{equation} \label{alpha_beta_counting}
\abs{\h{x \in [N] \col \alpha \leq \alpha_x \leq \beta}} = N^{\theta_b(\alpha) + \epsilon}\,, \qquad \epsilon = O \pbb{\frac{\zeta \log \log N}{\log N}}\,.
\end{equation}
\item \label{itm:W_size}
For $\delta \geq C \frac{\zeta \log \log N}{d}$ and $2 + \delta \leq \lambda \leq \Lambda(\alpha_{\max}(b))$, with high probability we have
\begin{equation*}
\abs{\cal W_{\lambda,\delta}} = N^{\theta_b(\Lambda^{-1}(\lambda - \delta)) + \epsilon}\,, \qquad \epsilon = O \pbb{\frac{\zeta \log \log N}{\log N}}\,.
\end{equation*}
\end{enumerate}
\end{lemma}

Note that, since $\xi \geq d^{-1/2}$, if the conclusion of Theorem \ref{thm:localisation} is nontrivial then $\delta \geq d^{-1/2}$, and hence the assumption on $\delta$ in Lemma \ref{lem:alpha_distr} \ref{itm:W_size} is automatically satisfied for suitably chosen $\zeta$.

\begin{proof}[Proof of Lemma \ref{lem:alpha_distr}]
Part \ref{itm:alpha_density} follows Corollary \ref{lem:degrees} below by noting that the assumption on $\beta$ implies $\theta_b(\alpha) - \theta_b(\beta) \geq C \frac{\zeta \log \log N}{\log N}$ 
by the mean value theorem.

Part \ref{itm:W_size} follows from Part \ref{itm:alpha_density}, using that $\log (\lambda - \delta) \geq \log 2$, that $\Lambda'$ is bounded on $[2,\infty)$, and the mean value theorem.
\end{proof}

\begin{corollary} \label{cor:lower_bound_w}
The following holds for large enough universal constants $C, \cal C$. Suppose that \eqref{d_assumption_localization} holds. Write $d = b \log N$. Let $\f w = (w_x)_{x \in [N]}$ be a normalized eigenvector of $A/\sqrt{d}$ with nontrivial eigenvalue $2+\cal C \xi^{1/2} \leq \lambda \leq \Lambda(\alpha_{\max}(b))$. Then with high probability for any $2 \leq p \leq \infty$  we have
\begin{equation*}
\norm{\f w}_p^{2} \geq N^{(2/p - 1)\theta_b(\Lambda^{-1}(\lambda)) + \epsilon}\,, \qquad \epsilon = O \qBB{ \frac{\log \log N}{\sqrt{\log N}} + b (\log \lambda) \pbb{\lambda + \frac{1}{\sqrt{\lambda - 2}}} (\xi + \xi_{\lambda - 2})}\,.
\end{equation*}

\end{corollary}

\begin{proof}
We choose $\delta \deq C (\xi + \xi_{\lambda - 2})$. Then by assumption on $\lambda$ we have $\delta \leq (\lambda - 2)/2$, and hence Theorem \ref{thm:localisation} yields, using that $\f v(x)$ is supported in $B_{r_\star}(x)$, $\sum_{x \in \cal W_{\lambda,\delta}} \sum_{y \in B_{r_{\star}}(x)} w_y^2 \geq  \frac{1}{2}$
with high probability. Using that for any vector $\f x \in \R^n$ we have $\norm{\f x}_p^2 \geq n^{2/p - 1} \norm{\f x}_2^2$ (by Hölder's inequality), with the choice $n = \sum_{x \in \cal W_{\lambda,\delta}} \abs{B_{r_{\star}}(x)}$, we get
\begin{equation} \label{w_infty_lower_est}
\norm{\f w}_{p}^{2} \geq \frac{1}{2} \pBB{\sum_{x \in \cal W_{\lambda,\delta}} \abs{B_{r_{\star}}(x)}}^{2/p - 1} \geq \frac{1}{2} \pB{\abs{\cal W_{\lambda,\delta}} N^{C \log \log N / \sqrt{\log N}}}^{2/p - 1}
\end{equation}
with high probability, where we used Lemma \ref{lem:upper_bound_degrees} to estimate $\max_{x \in [N]}\abs{B_{r_\star}(x)} \leq N^{C \log \log N / \sqrt{\log N}}$ with high probability.

Next, using the mean value theorem and elementary estimates on the derivatives of $\theta_b$ and $\Lambda^{-1}$, we estimate
\begin{equation*}
\theta_b(\Lambda^{-1}(\lambda - \delta)) - \theta_b(\Lambda^{-1}(\lambda)) \leq C b (\log \lambda) \pbb{\lambda + \frac{1}{\sqrt{\lambda - 2}}} \delta\,.
\end{equation*}
Invoking Lemma \ref{lem:alpha_distr} \ref{itm:W_size} with $\zeta \deq \log \log N$, and recalling \eqref{w_infty_lower_est}, therefore yields the claim.
\end{proof}

\subsection{Connected components of $\bb G(N,d/N)$} \label{sec:components}
In this appendix we give some basic estimates on the sizes of connected components of $\bb G(N,d/N)$. These are needed for the analysis of the tuning forks in Appendix \ref{sec:tuning_forks} below.
The arguments are standard and are tailored to work well in the regime $1 \ll d \leq \log N$ that we are interested in. For smaller values of $d$, see e.g.\ \cite{Bol01}.

\begin{lemma} \label{lem:W_k_estimate}
Let $W_k$ be the number of connected components that have $k$ vertices and $\wh W_k$ the number of connected components that have $k$ vertices and are not a tree. Then for $k \leq N/2$ we have
\begin{equation*}
\E[W_k] \leq N \ee^{-k (d/2 - \log d - 1)}\,, \qquad
\E[\wh W_k] \leq \ee^{-k (d/2 - \log d - 1)}\,.
\end{equation*}
\end{lemma}
\begin{proof}
For a set $X \subset [N]$, denote by $\cal T(X)$ the set of spanning trees of $X$.
If $X$ is a connected component of $\bb G$ then there exists $\bb T \in \cal T(X)$ a subgraph of $\bb G$ such that no vertex of $X$ is connected to a vertex of $[N] \setminus X$. Hence,
\begin{equation*}
W_k \leq \sum_{X \subset [N]} \ind{\abs{X} = k} \sum_{\bb T \in \cal T(X)} \ind{\bb T \subset \bb G} \prod_{x \in X} \prod_{y \in [N] \setminus X} (1 - A_{xy})\,.
\end{equation*}
Taking the expectation now easily yields the claim, using $\abs{\cal T(X)} = \abs{X}^{\abs{X} - 2}$ by Cayley's theorem, that a tree on $k$ vertices has $k - 1$ edges, Stirling's approximation, and $1 - x \leq \ee^{-x}$.

The argument to estimate $\wh W_k$ is similar, noting that in addition to a spanning tree $\bb T$ of $X$, we also have to have at least one edge not in $\bb T$ connecting two vertices of $X$. Thus,
\begin{equation*}
\wh W_k \leq \sum_{X \subset [N]} \ind{\abs{X} = k} \sum_{\bb T \in \cal T(X)} \ind{\bb T \subset \bb G} \prod_{x \in X} \prod_{y \in [N] \setminus X} (1 - A_{xy}) \sum_{\{u,v\} \in X^2 \setminus E(\bb T)} A_{uv}\,,
\end{equation*}
and we may estimate the expectation as before.
\end{proof}

We call a connected component of $\bb G$ \emph{small} if it is not the giant component.
For the following statement we recall the definition of high probability from Definition \ref{def:hp}.

\begin{corollary} \label{cor:components}
Suppose that $d \gg 1$. All small components of $\bb G$ have at most $O\pb{\frac{\log N}{d}}$ 
vertices with very high probability. All small components of $\bb G$ are trees with high probability. The giant component of $\bb G$ has at least $N (1 - \ee^{-d/4})$ vertices with high probability.
\end{corollary}
\begin{proof}
Any small component has at most $N/2$ vertices. Using Lemma \ref{lem:W_k_estimate} we therefore get that the probability that there exists a small component with at least $K$ vertices is bounded by
\begin{equation*}
\P (\exists k \in [K,N/2] \,, W_k \geq 1) \leq \sum_{k = K}^{N/2} \E[W_k] \leq 2 N \ee^{-K (d/2 - \log d - 1)}\,,
\end{equation*}
by summing the geometric series. Since $d/2 - \log d - 1 \geq c d$ for some universal constant $c$, we obtain the first claim. To obtain the second claim, we use Lemma \ref{lem:W_k_estimate} to estimate the probability that there exists a small component that is not a tree by $\sum_{k = 1}^{N/2} \E{\wh W_k} \leq \ee^{-d/3}$. To obtain the last claim, we estimate the expected number of vertices in small components by $\E \qb{\sum_{k = 1}^{N/2} k W_k} \leq N \sum_{k = 1}^\infty k \ee^{-k (d/2 - \log d - 1)} \leq C N \ee^{- d/3}$
using Lemma \ref{lem:W_k_estimate}, and the third claim follows from Chebyshev's inequality.
\end{proof}

We may now estimate the adjacency matrix on the small components of $\bb G(N,d/N)$. The following result follows immediately from Corollary \ref{cor:components} and Lemma \ref{lem:forest_bound}.
\begin{corollary} \label{cor:small_components}
Suppose that $d \gg 1$. Then the operator norm of $A / \sqrt{d}$ restricted to the small components of $\bb G$ is bounded by $O\pb{\frac{\sqrt{\log N}}{d}}$ with high probability.
\end{corollary}

Corollary \ref{cor:small_components} makes it explicit that Theorem \ref{thm:delocalization} excludes all eigenvectors on small components of $\bb G$, whose eigenvalues lie outside $\cal S_\kappa$ precisely under the lower bound from \eqref{d_condition_deloc}.

\subsection{Tuning forks and proof of Lemma \ref{lem:star_localization}} \label{sec:tuning_forks}

In this appendix we give a precise definition of the $D$-tuning forks from Section \ref{sec:forks_intro} and prove Lemma \ref{lem:star_localization}.

\begin{definition} \label{def:star}
A \emph{star of degree} $D \in \N$ consists of a vertex, the \emph{hub}, and $D$ leaves adjacent to the hub, the \emph{spokes}. A \emph{star tuning fork of degree} $D$ is obtained by taking two disjoint stars of degree $D$ along with an additional vertex, the \emph{base}, and connecting both hubs to the base. We say that a star tuning fork is \emph{rooted in a graph $\bb H$} if it is a subgraph of $\bb H$ in which both hubs have degree $D+1$ and all spokes are leaves.
\end{definition}

\begin{lemma} \label{lem:forks}
If a star tuning fork of degree $D$ is rooted in some graph $\bb H$, then the adjacency matrix of $\bb H$ has eigenvalues $\pm \sqrt{D}$ with corresponding eigenvectors supported on the stars of the tuning fork, i.e.\ on $2D + 2$ vertices.
\end{lemma}

\begin{proof}
Suppose first that $D \geq 1$.
Note first that the adjacency matrix of a star of degree $D$ has rank two and has the two nonzero eigenvalues $\pm \sqrt{D}$, with associated eigenvector equal to $\pm \sqrt{D}$ at the hub and $1$ at the spokes. Now take a star tuning fork of degree $D$ rooted in a graph $\bb H$. Define a vector on the vertex set of $\bb H$ by setting it to be $\pm \sqrt{D}$ at the hub of the first star, $1$ at the spokes of the first star, $\mp \sqrt{D}$ at the hub of the second star, $-1$ at the spokes of the second star, and $0$ everywhere else. Then it is easy to check that this vector is an eigenvector of the adjacency matrix of $\bb H$ with eigenvalue $\pm \sqrt{D}$. If $D = 0$ the construction is analogous, defining the vector to be $+1$ at one hub and $-1$ at the other.
\end{proof}

We recall from Section \ref{sec:forks_intro} that $F(d,D)$ denotes the number of star tuning forks of degree $D$ rooted in $\bb G_{\mathrm{giant}}$.

\begin{lemma} \label{lem:stars_density}
Suppose that $1 \ll d \ll \sqrt{N}$ and $0 \leq D \ll \sqrt{N}$. Then
\begin{equation} \label{E_calF}
\E [F(d,D)] = \frac{N d^2 \ee^{-2d}}{2 D!^2} \p{d \ee^{-d + 1}}^{2D} (1 + o(1))
\end{equation}
and $\E [F(d,D)^2] \leq \E [F(d,D)]^2 (1 + o(1))$.
\end{lemma}

\begin{proof}[Proof of Lemma \ref{lem:star_localization}]
From Lemma \ref{lem:stars_density} we deduce that if $1 \ll d = b \log N = O(\log N)$ and $D \ll \log N / \log \log N$, then $\E[F(d,D)] = N^{1 - 2b - 2b D + o(1)}$. The claim then follows from the second moment estimate in Lemma \ref{lem:stars_density} and Chebyshev's inequality.
\end{proof}

\begin{proof}[Proof of Lemma \ref{lem:stars_density}]
Let $x_1,x_2 \in [N]$ be distinct vertices and $R_1, R_2 \subset [N] \setminus \{x_1,x_2\}$ be disjoint subsets of size $D$. We abbreviate $U = (x_1, x_2, R_1, R_2)$ and sometimes identify $U$ with $\{x_1, x_2\} \cup R_1 \cup R_2$. The family $U$ and a vertex $o \in [N] \setminus U$ define a star tuning fork of degree $D$ with base $o$, hubs $x_1$ and $x_2$, and associated spokes $R_1$ and $R_2$. Let $\scr C_k(\bb H)$ denote the vertex set of the $k$th largest connected component of the graph $\bb H$.  Then $F(d,D) = \frac{1}{2}\sum_U \sum_{o \in [N] \setminus U} \ind{o \in \scr C_1(\bb G)} S_{o,U}$, where
\begin{equation*}
S_{o,U} \deq  \prod_{i = 1}^2 \pBB{\prod_{u \in R_i \cup \{o\}} A_{x_i u} \prod_{u \in [N] \setminus (R_i \cup \{o\})} (1 - A_{x_i u}) \prod_{u \in R_i} \prod_{v \in [N] \setminus \{x_i\}} (1 - A_{uv})}.
\end{equation*}
The factor $\frac{1}{2}$ corrects the overcounting from the labelling of the two stars.

For disjoint deterministic $U$, we split the random variables $A = (A', A'')$ into two independent families, where $A' \deq (A_{uv} \col u \in U \text{ or } v \in U)$ and $A'' \deq (A_{uv} \col u,v \in [N] \setminus U)$. Note that $S_{o,U}$ is $A'$-measurable. We define the event
\begin{equation*}
\Xi \deq \h{\abs{\scr C_1(\bb G \vert_{[N] \setminus U})} > \abs{\scr C_2(\bb G \vert_{[N] \setminus U})} + 2 D + 2}\,,
\end{equation*}
which is $A''$-measurable. By Corollary \ref{cor:components} and the assumption on $D$, the event $\Xi$ holds with high probability. Moreover, we have $\ind{\Xi} \ind{o \in \scr C_1(\bb G)} S_{o,U} = \ind{\Xi} \ind{o \in \scr C_1(\bb G \vert_{[N] \setminus U})} S_{o,U}$,
since the component of $o$ in $\bb G$ and $\bb G \vert_{[N] \setminus U}$ differ by $2D + 2$ vertices. Thus, for fixed $o \in [N] \setminus U$, using the independence of $A'$ and $A''$, we get
\begin{align*}
\E [\ind{o \in \scr C_1(\bb G)} S_{o,U}] &= \E [\ind{\Xi} \ind{o \in \scr C_1(\bb G \vert_{[N] \setminus U})} S_{o,U}] + \E [\ind{\Xi^c} \ind{o \in \scr C_1(\bb G)} S_{o,U}]
\\
&= \E [S_{o,U}] \qb{\P \pb{o \in \scr C_1(\bb G \vert_{[N] \setminus U})} + O\pb{\P(\Xi^c)}}\,.
\end{align*}
We have $\P(\Xi^c)  = o(1)$ and $\P \pb{o \in \scr C_1(\bb G \vert_{[N] \setminus U})} = 1 - o(1)$ by Corollary \ref{cor:components} and the assumption on $D$. Computing $\E[S_{o,U}]$ and performing the sum over $o$ and $U$, we therefore conclude that
\begin{equation*}
\E [F(d,D)] = \frac{N (N - 1) \cdots (N - 2D - 3 + 1)}{2 D!^2} \pbb{\frac{d}{N}}^{2D + 2} \pbb{1 - \frac{d}{N}}^{2 (N - D - 1) + 2 D (N - 1)} (1 + o(1))\,,
\end{equation*}
from which \eqref{E_calF} follows. The estimate of the second moment is similar; one can even disregard the restriction to the giant component by estimating $\E [F(d,D)^2] \leq \frac{1}{4} \sum_{U, \tilde U} \sum_{o,\tilde o \in [N]} \E [S_{o,U} S_{\tilde o, \tilde U}]$; we omit the details.
\end{proof}

\subsection{Multilinear large deviation bounds for sparse random vectors} \label{sec:lde}
In this appendix we collect basic large deviation bounds for multilinear functions of sparse random vectors, which are proved in \cite{HeKnowlesMarcozzi2018}.
The following result is proved in Propositions 3.1, 3.2, and 3.5 of \cite{HeKnowlesMarcozzi2018}. We denote by $\norm{X}_r \deq \p{\E \abs{X}^r}^{1/r}$ the $L^r$-norm of a random variable $X$.

\begin{proposition}
\label{pro:large_deviation} 
Let $r$ be even and $1 \leq d \leq N$. Let $X_1, \ldots, X_N$ 
be independent random variables satisfying 
\[ \E X_i = 0, \qquad \E \abs{X_i}^k \leq \frac{1}{N d^{(k-2)/2}}
\] 
for all $i \in [N]$ and $2 \leq k \leq r$. Let $a_i \in \C$ and $b_{ij} \in \C$ be deterministic for all $i,j \in [N]$. Suppose that 
\begin{equation*}
\bigg(\frac{1}{N} \sum_i \abs{a_i}^2 \bigg)^{1/2} \leq \gamma\,,  \qquad \frac{\max_i \abs{a_i}}{\sqrt{d}}  \leq \psi,
\end{equation*}
and
\begin{equation*}
\bigg( \max_i \frac{1}{N} \sum_{j} \abs{b_{ij}}^2 \bigg)^{1/2} \vee \bigg( \max_j \frac{1}{N}
\sum_i \abs{b_{ij}}^2 \bigg)^{1/2} \leq \gamma, \qquad \frac{\max_{i,j} \abs{b_{ij}}}{d}  \leq \psi 
\end{equation*}
for some $\gamma, \psi \geq 0$. Then
\begin{subequations} 
\begin{align} 
\normbb{\sum_i a_i X_i}_r & \leq \bigg( \frac{ 2r}{1 + 2 (\log (\psi/\gamma))_+} \vee 2 \bigg) \big( \gamma \vee \psi \big), \label{eq:LDB_linear} \\ 
\normbb{\sum_i a_i \big( \abs{X_i}^2 - \E \abs{X_i}^2 \big)}_r & \leq 2 \bigg( 1 + \frac{2d}{N} \bigg) \max_i \abs{a_i} \bigg( \frac{r}{d} \vee \sqrt{\frac{r}{d}} \bigg),\label{eq:LDB_linear_noncentered}  \\ 
\normbb{\sum_{i\neq j} b_{ij} X_iX_j}_r & \leq \bigg( \frac{ 4r}{1 + (\log (\psi/\gamma))_+} \vee 4 \bigg) ^2\big( \gamma \vee \psi \big). \label{eq:LDB_quadratic} 
\end{align} 
\end{subequations} 
\end{proposition}

The $L^r$-norm bounds in Proposition \ref{pro:large_deviation} induce bounds that hold with very high probability.

\begin{corollary}
\label{cor:large_deviation_very_high_probabilty} 
Fix $\kappa \in (0,1)$. 
Let the assumptions of Proposition~\ref{pro:large_deviation} be satisfied. 
If $\psi/\gamma \geq N^{\kappa/4}$ then with very high probability
\begin{equation} \label{eq:LDB_wvhp}
 \absbb{\sum_i a_i X_i}  \leq \Cnu\psi\,, \qquad \absbb{\sum_{i \neq j} b_{ij} X_iX_j} \leq \Cnu \psi\,.
\end{equation}
\end{corollary} 

\begin{remark} 
Our proof of Corollary \ref{cor:large_deviation_very_high_probabilty} shows that $\Cnu$ can be chosen as a linear function of $\nu$ for the first estimate of \eqref{eq:LDB_wvhp} and as a quadratic function of $\nu$
 for the second estimate of \eqref{eq:LDB_wvhp}.
\end{remark}

\begin{proof} 
Fix $\nu \geq 1$. We choose $r= \nu \log N$ in \eqref{eq:LDB_linear} of Proposition~\ref{pro:large_deviation} and obtain from Cheybshev's inequality that 
\[ \P\bigg( \absbb{\sum_i a_i X_i} > \Cnu \psi \bigg) \leq N^{-\nu}, \qquad \Cnu \deq \frac{4\ee}{\kappa} \nu \] 
as $\kappa \in (0,1)$.
Similarly, choosing $r = \frac 1 2 \nu \log N$ in
\eqref{eq:LDB_quadratic} yields 
\[
\P \bigg( \absbb{\sum_{i \neq j} b_{ij} X_i X_j} > 4 \Cnu \psi \bigg) \leq N^{-\nu},  
\qquad \Cnu \deq \frac{16\ee^2}{\kappa^2} \nu^2\,. \qedhere \] 
\end{proof}

\subsection{Resolvent identities} 

In this appendix we record some well-known identities for the Green function \eqref{eq:def_G} and its minors from  Definition \ref{def:minors}. We assume throughout that $z \in \C \setminus \R$.

\begin{lemma}[Ward identity] \label{lem:Ward}
For $x \notin T \subset [N]$ we have
\begin{equation*}
\sum_y^{(T)} \abs{G_{xy}^{(T)}}^2 = \frac{1}{\im z} \im G_{xx}^{(T)} \,.
\end{equation*}
\end{lemma}
\begin{proof}
This is a standard identity for resolvents, see e.g.\ \cite[Eq.\ (3.6)]{BenyachKnowles2017}.
\end{proof}

\begin{lemma} \label{lem:resolvent_expansions} 
Let $T \subset [N]$. For $x, y \notin T$ and $x \neq y$, we have 
\begin{equation} \label{eq:expansion_G_off_diag} 
 G_{xy}^{(T)} = - G_{yy}^{(T)} \sum_{a}^{(Ty)} G_{xa}^{(T y)} {M_{ay}} = - G_{xx}^{(T)} \sum_{b}^{(Tx)} {M_{xb}} G_{by}^{(Tx)} . 
\end{equation}
For $x, y, a \notin T$ and $x \neq a \neq y$, we have 
\begin{equation} \label{eq:expansion_G_general} 
 G_{xy}^{(Ta)} = G_{xy}^{(T)} - \frac{G_{xa}^{(T)} G_{ay}^{(T)}}{G_{aa}^{(T)}} . 
\end{equation}
For any $x \in [N]$, we have 
\begin{equation} \label{eq:Schur_complement} 
 \frac{1}{G_{xx}} = M_{xx} - z- \sum_{a,b}^{(x)} M_{xa} G_{ab}^{(x)} M_{bx}. 
\end{equation}
\end{lemma} 

\begin{proof} 
All identities are standard and proved e.g.\ in \cite{BenyachKnowles2017}: \eqref{eq:expansion_G_off_diag} in \cite[Eq.~(3.5)]{BenyachKnowles2017}, \eqref{eq:expansion_G_general} in \cite[Eq.~(3.4)]{BenyachKnowles2017} 
and \eqref{eq:Schur_complement} in \cite[Lemma~A.1 and (5.1)]{BenyachKnowles2017}. 
\end{proof} 

We recall \eqref{eq:def_M} and derive two expansions used in Section~\ref{sec:delocalization}. 
For any $T \subset [N]$ and $x,y, u \notin T$, $x \neq u \neq y$, we have 
\begin{subequations} 
\begin{equation} 
G_{xy}^{(Tu)} = G_{xy}^{(T)} + \sum_a^{(Tu)} G_{xa}^{(Tu)} H_{au} G_{uy}^{(T)} + \frac{f}{N} G_{uy}^{(T)} \sum_{a}^{(Tu)} G_{xa}^{(Tu)} \label{eq:expansion_G_xy_Tu1}, 
\end{equation}  
which follows from \eqref{eq:expansion_G_general} and \eqref{eq:expansion_G_off_diag}. 
Under the same assumptions, applying \eqref{eq:expansion_G_off_diag} to \eqref{eq:expansion_G_xy_Tu1} yields  
\begin{equation} 
\begin{aligned} 
 G_{xy}^{(Tu)} = \, & \phantom{-} G_{xy}^{(T)} -G_{uu}^{(T)} \sum_a^{(Tu)} G_{xa}^{(Tu)} H_{au} \sum_{b}^{(Tu)} H_{ub} G_{by}^{(Tu)} \\ 
 & - \frac{f}{N} G_{uu}^{(T)} \sum_a^{(Tu)} G_{xa}^{(Tu)} H_{au} \sum_b^{(Tu)} G_{by}^{(Tu)} + \frac{f}{N} G_{uy}^{(T)} \sum_{a}^{(Tu)} G_{xa}^{(Tu)}. \label{eq:expansion_G_xy_Tu2} 
\end{aligned} 
\end{equation} 
\end{subequations}

\subsection{Stability estimate -- proof of Lemma~\ref{lem:stability}}  \label{sec:proof_stability} 
In this appendix we prove Lemma~\ref{lem:stability}.
The estimate in \cite[Lemma~3.5]{ErdosYauYin2012} corresponding to \eqref{eq:stability_estimate} 
has logarithmic factors, which are not affordable for our purposes: they have to be replaced with constants. The following proof of Lemma~\ref{lem:stability} is analogous to that of the more complicated bulk stability estimate from \cite[Lemma~5.11]{AjankiQVE}. 

\begin{proof}[Proof of Lemma~\ref{lem:stability}]  
We introduce the vectors $\f g \deq (g_x)_{x \in \cal X}$ and $\f \eps \deq (\eps_x)_{x \in \cal X}$. Moreover, with the abbreviation $m \deq m(z)$ we introduce the constant vectors $\f m = (m)_{x \in \cal X}$ and $\f e \deq \abs{\cal X}^{-1/2} (1)_{x \in \cal X}$. We regard all vectors as column vectors.
A simple computation starting from the difference of  \eqref{m_quadr} and \eqref{eq:self_consistent_eq_perturbed} 
reveals that
\begin{equation} \label{eq:stability_equation} 
 B(\f g - \f m) = m (\f g - \f m) \pb{\mathbf{e} \mathbf{e}^* (\f g - \f m)} - (\f g  - \f m) m \f \eps - m^2 \f \eps, 
\end{equation}
where $B \deq 1 - m^2 \f e \f e^*$, and column vectors are multiplied entrywise.
The inverse of $B$ is
\[ B^{-1} = 1 + \frac{m^2}{1- m^2} \mathbf{e} \mathbf{e}^*. \] 
For a matrix $R \in \C^{\cal X \times \cal X}$, we write $\norm{R}_{\infty \to \infty}$ for the operator norm induced by the norm $\norm{\f r}_\infty = \max_{x \in \cal X} \abs{r_x}$ on $\C^{\cal X}$.
It is easy to see that there is $c>0$, depending only on $\kappa$, such that $\abs{1- m(w)^2}\geq c$ for all $w \in \C_+$ satisfying $\abs{\Re w} \leq 2 - \kappa$. 
Hence, owing to $\norm{\f e \f e^*}_{\infty \to \infty} = 1$, 
we obtain
$\norm{B^{-1}}_{\infty \to \infty} \leq 1 + \abs{1- m^2}^{-1} \leq 1+ c^{-1}$. 
Therefore, inverting $B$ in \eqref{eq:stability_equation} and choosing $b$, depending only on $\kappa$, 
sufficiently small to absorb the term quadratic in $\f g - \f m$ into the left-hand side of the resulting bound 
yields \eqref{eq:stability_estimate} for some sufficiently large $C>0$, depending only on $\kappa$. 
This concludes the proof of Lemma~\ref{lem:stability}. 
\end{proof} 

\subsection{Instability estimate -- proof of \eqref{instability_intro}} \label{sec:instability}
In this appendix we prove \eqref{instability_intro}, which shows that the self-consistent equation \eqref{sc_intro_naive} is unstable with a logarithmic factor, which renders it useless for the analysis of sparse random graphs. More precisely, we show that the norm $\norm{(I - m^2 S)^{-1}}_{\infty \to \infty}$ is ill-behaved precisely in the situation where we need it. For simplicity, we replace $m^2$ with a phase $\alpha^{-1} \in S^1$ separated from $\pm 1$, since for $\re z \in \cal S_\kappa$ we have
\begin{equation} \label{m_estimate}
\abs{m(z)}^2 = 1 - O(\im z) \,, \qquad \im m(z) \asymp 1\,,
\end{equation}
by \cite[Lemma 3.5]{EKYY3}.
Moreover, for definiteness, recalling that with very high probability most of the $d (1 + o(1))$ neighbours of any vertex in $\cal T$ are again in $\cal T$, we assume that $S$ is the adjacency matrix of a $d$-regular graph on $\cal T$ divided by $d$.

By the spectral theorem and because $S$ is Hermitian, $\norm{(\alpha - S)^{-1}}_{2 \to 2}$ is bounded, but, as we now show, the same does not apply to $\norm{(\alpha - S)^{-1}}_{\infty \to \infty}$. Indeed, the upper bound of \eqref{instability_intro} follows from \cite[Proposition A.2]{EKYY4}, and the lower bound from the following result.

\begin{lemma}[Instability of \eqref{sc_intro_naive}] \label{lem:instability}
Let $S$ be $1/d$ times the adjacency matrix of a graph whose restriction to the ball of radius $r \in \N^*$ around some distinguished vertex is a $d$-regular tree. Let $\alpha \in S^1$ be an arbitrary phase. Then
\begin{equation} \label{norm_infty_lb}
\norm{(\alpha - S)^{-1}}_{\infty \to \infty} \geq c \pbb{\frac{r}{\log r} \wedge d}
\end{equation}
for some universal constant $c > 0$.
\end{lemma}
In particular, denoting by $N$ the number of vertices in the tree (which may be completed to a $d$-regular graph by connecting the leaves to each other), for $d \asymp \log N$ and $r \asymp \frac{\log N}{\log d}$ we find
\begin{equation} \label{norm_infty_lb2}
\norm{(\alpha - S)^{-1}}_{\infty \to \infty} \geq \frac{c \log N}{(\log \log N)^2}\,,
\end{equation}
which is the lower bound of \eqref{instability_intro}.

\begin{proof}[Proof of Lemma \ref{lem:instability}]
After making $r$ smaller if needed, we may assume that $\frac{r}{\log r} \leq d$. We shall construct a vector $\f u$ satisfying $\norm{\f u}_\infty = 1$ and $\norm{(\alpha - S) \f u}_\infty = O\pb{\frac{\log r}{r}}$, from which \eqref{norm_infty_lb} will follow. To that end, we construct the sequence $a_0, a_1, \dots, a_r$ by setting
\begin{equation*}
a_0 \deq 1\,, \qquad a_1 \deq \alpha\,, \qquad a_{k+1} \deq \frac{d}{d - 1} \alpha a_{k} - \frac{1}{d - 1} a_{k - 1} \quad \text{for} \quad 1 \leq k \leq r - 1\,.
\end{equation*}
A short transfer matrix analysis shows that
$\abs{a_k} \leq \ee^{C_1 k /d}$ for some constant $C_1$. Now choose $\mu \deq C_2 \frac{\log r}{r}$ with $C_2 \deq 2 \vee 2 C_1$, and define $b_k \deq \ee^{-\mu k} a_k$. Calling $o$ the distinguished vertex,  we define $u_x \deq b_k$ if $k = \dist(o,x) \leq r$ and $u_x = 0$ otherwise. It is now easy to check that $\norm{(\alpha - S) \f u}_\infty = O\pb{\frac{\log r}{r}}$, by considering the cases $k = 0$, $1 \leq k \leq r - 1$, and $k \geq r$ separately. The basic idea of the construction is that if $\mu$ were zero, then $(\alpha - S) \f u$ would vanish exactly on $B_{r - 1}(o)$, but it would be large on the boundary $S_r(o)$. The factor $\ee^{-\mu k}$ introduces exponential decay in the radius which dampens the contribution of the boundary $S_r(o)$ at the expense of introducing errors in the interior $B_{r - 1}(o)$.
\end{proof}

\bibliography{bibliography} 
\bibliographystyle{amsplain}

\bigskip

\noindent
Johannes Alt (\href{mailto:johannes.alt@unige.ch}{johannes.alt@unige.ch})
\\
Rapha\"el Ducatez (\href{mailto:raphael.ducatez@unige.ch}{raphael.ducatez@unige.ch})
\\
Antti Knowles (\href{mailto:antti.knowles@unige.ch}{antti.knowles@unige.ch})
\\
University of Geneva, Section of Mathematics, 2-4 Rue du Li\`evre, 1211 Gen\`eve 4, Switzerland.

\end{document}